\documentclass[11pt]{amsart}
\usepackage[margin=1.1in]{geometry}
\usepackage[T1]{fontenc}
\usepackage{amsmath,amssymb,amsthm,booktabs,array,microtype}
\usepackage{xcolor,graphicx}
\usepackage[colorlinks=true,linkcolor=blue!60!black,citecolor=blue!60!black,urlcolor=blue!60!black]{hyperref}

\newtheorem{theorem}{Theorem}[section]
\newtheorem{proposition}[theorem]{Proposition}
\newtheorem{lemma}[theorem]{Lemma}
\newtheorem{corollary}[theorem]{Corollary}
\newtheorem{conjecture}[theorem]{Conjecture}
\theoremstyle{definition}

\theoremstyle{remark}
\newtheorem{remark}[theorem]{Remark}

\DeclareMathOperator{\inv}{inv}
\DeclareMathOperator{\Av}{Av}
\DeclareMathOperator{\comp}{comp}
\DeclareMathOperator{\im}{im}
\DeclareMathOperator{\Inv}{Inv}
\newcommand{\sk}{\mathrm{sk}}
\newcommand{\dsum}{\oplus}
\newcommand{\ssum}{\ominus}
\newcommand{\R}{\mathcal{R}}
\newcommand{\NW}{\mathrm{NW}}
\newcommand{\SE}{\mathrm{SE}}
\newcommand{\Sreg}{\mathrm{S}}
\newcommand{\Ereg}{\mathrm{E}}

\title[Residual structure and growing monotonicity regions]{Residual structure and growing inversion-monotonicity regions\\ for $1324$-avoiding permutations}
\author{Lingsen Meng}
\address{Department of Earth, Planetary, and Space Sciences, University of California, Los Angeles, CA 90095, USA}
\email{lsmeng@g.ucla.edu}
\date{September 8, 2026}

\begin{document}
\sloppy

\begin{abstract}
Let $a(n,k)$ be the number of $1324$-avoiding permutations of length $n$ with
$k$ inversions.  Linusson and Verkama proved $a(n,k)\le a(n+1,k)$ for
$k\le2n-7$.  We study the obstruction beyond that line: the residuals
$\R_{\delta,n}$, namely the indecomposable, non-almost-decomposable avoiders
at defect $\delta=k-2n+7$.  Contracting maximal increasing consecutive runs
reduces residuality to a quadratic equation on a finite family of skeletons.
It follows that, for every fixed $\delta$, the eventual count has the form
\[
 |\R_{\delta,n}|=A_\delta n^2+B_\delta n+C_\delta .
\]
Our central uniform result determines the quadratic coefficient at every
defect.  With $P(q)=\prod_{j\ge1}(1-q^j)^{-1}$,
\[
 \sum_{\delta\ge0}A_\delta q^\delta
   =\frac{4q^3(1+q)}{(1-q)^2}P(q)^2.
\]
This is a formula for the leading coefficient of the residual count, not for
the full count.

The same structural estimates give computer-assisted proofs of
$a(n,k)\le a(n+1,k)$ for every $n\ge1$ and $k\le2n+6$, and of regions whose
width grows with $n$: for $n\ge2^{16},2^{18},2^{20}$ the defect may be as
large as $\lfloor\sqrt n/4\rfloor$, $\lfloor\sqrt n/3\rfloor$,
$\lfloor\sqrt n/2\rfloor$, respectively.  More generally, every fixed
$c<\sqrt2\log(5)/\log(68)$ is admissible for all sufficiently large $n$.
The three added fixed defects $11,12,13$ use complete catalogue and
rational-sum certificates supplied in the accompanying archival supplement.
The leading-coefficient theorem is obtained from a complete finite
classification of marked rank-three cores and all-parameter extension
lemmas; we state the finite certificates and the logical order in which they
enter.  We also determine the exact rank-three stabilization onset, while
keeping it separate from the still unknown onset of the complete residual
count.  The unrestricted Claesson--Jel\'inek--Steingr\'imsson conjecture, the
full residual polynomials, and the sharp global base-length bound remain
open.
\end{abstract}

\maketitle

\section{Introduction}\label{sec:intro}

For a permutation $\pi=\pi_1\cdots\pi_n$ let $\inv(\pi)$ be its number of inversions and let $\Av^k_n(1324)$ be the set of $1324$-avoiding permutations of length $n$ with $k$ inversions, $a(n,k)=|\Av^k_n(1324)|$. Claesson, Jel\'inek and Steingr\'imsson \cite{CJS} conjectured that
\begin{equation}\label{eq:conj}
a(n,k)\le a(n+1,k)\qquad\text{for all }n\ge1,\ k\ge0,
\end{equation}
and showed that \eqref{eq:conj} implies the bound $e^{\pi\sqrt{2/3}}\approx13.002$ for the growth rate of $\Av(1324)$.  Bevan, Brignall, Elvey Price and Pantone proved the upper bound $13.5$ \cite{BBEP}. Linusson and Verkama \cite{LV} proved \eqref{eq:conj} for $k\le2n-7$ and determined $a(n,k)$ exactly in that range:
\begin{equation}\label{eq:LV}
a(n+1,k)-a(n,k)=4p_2(k-n+1)+2p_2(k-n)\qquad(k\le2n-7),\qquad p_2(m)=\sum_{i}p(i)p(m-i),
\end{equation}
$p$ the partition function ($p_2(m)=0$ for $m<0$); equivalently \cite[Theorem~1.2]{LV}
\begin{equation}\label{eq:LVexact}
a(n,k)=p_2(k)-4p_2(k-n+1)-6\sum_{i=0}^{k-n}p_2(i)\qquad(k\le2n-7).
\end{equation} Their proof has two parts. A permutation is \emph{decomposable} if it is a direct sum of two nonempty permutations, and an indecomposable permutation $\pi$ is \emph{almost decomposable} if one of the four deletions $\pi\setminus1$, $\pi\setminus n$, $\pi\setminus\pi_1$, $\pi\setminus\pi_n$ (delete the entry and standardize) is decomposable. First, every indecomposable $1324$-avoider with $k\le2n-7$ inversions is almost decomposable \cite[Theorem~2.5]{LV}. Second, there are maps $g$ on decomposable and $f$ on almost decomposable $1324$-avoiders, defined for every $n$ and $k$, which preserve $1324$-avoidance and the inversion number, increase the length by one, are injective, and have disjoint images \cite[Theorem~3.5]{LV}.

This paper concerns the region $k\ge2n-6$, where indecomposable $1324$-avoiders that are not almost decomposable exist. Write $k=2n-7+\delta$ and call $\delta\ge1$ the \emph{defect}; let
\[
\R_{\delta,n}=\{\pi\in\Av^{2n-7+\delta}_n(1324):\ \pi\ \text{indecomposable and not almost decomposable}\}
\]
be the set of \emph{residuals}. Since $f\sqcup g$ is defined on everything else and is injective,
\begin{equation}\label{eq:identity}
a(n+1,k)-a(n,k)=\big|\Av^k_{n+1}(1324)\setminus\im(f\sqcup g)\big|-|\R_{\delta,n}| ,
\end{equation}
and \eqref{eq:conj} at $(n,k)$ is equivalent to $|\R_{\delta,n}|\le|\Av^k_{n+1}(1324)\setminus\im(f\sqcup g)|$. The complement of the image is large: it always contains every $\sigma$ with $\sigma_1=n+1$, $\sigma_{n+1}=1$, $\sigma_2=n+1$ or $\sigma_{n+1}=2$, and these number $2a(n,k-n)+2a(n,k-n+1)$ as long as $k\le2n-4$ (Lemma~\ref{lem:R1}); moreover $a(n,j)=p_2(j)$ as soon as $j\le n-2$.

We place the three principal conclusions first.  Their proofs use the
structural development below and the finite certificates described in
Section~\ref{sec:newcert}.  The order matters: the finite marked-core class
used in Theorem~\ref{thm:uniformA} is defined and enumerated independently,
before it is identified with canonical rank-three cells.

\begin{theorem}[uniform quadratic coefficient]\label{thm:uniformA}
For every $\delta\ge0$, let $A_\delta$ be the coefficient of $n^2$ in the
eventual polynomial $|\R_{\delta,n}|$, taking $A_\delta=0$ when its degree is
less than two.  Then
\begin{equation}\label{eq:uniformA}
 \sum_{\delta\ge0}A_\delta q^\delta
 =\frac{4q^3(1+q)}{(1-q)^2}P(q)^2,
 \qquad P(q)=\prod_{j\ge1}(1-q^j)^{-1}.
\end{equation}
\end{theorem}

\begin{theorem}[fixed and growing low-inversion regions]\label{thm:headline}
For all $n\ge1$,
\[
 a(n,k)\le a(n+1,k)\qquad(0\le k\le2n+6).
\]
The following larger regions also hold:
\begin{align*}
n\ge2^{16}&:\quad 0\le k\le2n-7+\lfloor\sqrt n/4\rfloor,\\
n\ge2^{18}&:\quad 0\le k\le2n-7+\lfloor\sqrt n/3\rfloor,\\
n\ge2^{20}&:\quad 0\le k\le2n-7+\lfloor\sqrt n/2\rfloor.
\end{align*}
For every fixed $0<c<\sqrt2\log(5)/\log(68)$, the same inequality holds for
$0\le k\le2n-7+\lfloor c\sqrt n\rfloor$ once $n\ge N(c)$.
In the three square-root regions and in the asymptotic region, the inequality
is strict whenever $\delta\ge1$.
\end{theorem}

The first assertion of Theorem~\ref{thm:headline} extends
Corollary~\ref{cor:2n3} by the three defects $11,12,13$.  Its proof uses only
the already completed sharp skeleton bound and the defect-$11,12,13$
catalogues, rational generating functions and boundary comparisons; it does
not use Theorem~\ref{thm:uniformA} or any later high-inversion construction.
The square-root assertions instead combine a uniform residual upper bound
with an elementary lower bound on two-coloured partitions.

\begin{theorem}\label{thm:A}
For $n\ge8$ the residual set $\R_{1,n}$ consists exactly of the inflations
\[
24153[\iota_a,\,1,\,\iota_c,\,\iota_d,\,\iota_e]\quad\text{and}\quad 31524[\iota_a,\,\iota_b,\,\iota_c,\,1,\,\iota_e]
\]
by increasing blocks ($\iota_m=12\cdots m$) with $a,d\ge1$, $c,e\ge2$, $(a-1)(c-2)+(d-1)(e-2)=0$ in the first case and $a,c\ge2$, $b,e\ge1$, $(a-2)(b-1)+(c-2)(e-1)=0$ in the second; the two families are exchanged by inversion. Consequently $|\R_{1,n}|=8(n-7)$ for $n\ge8$ (and $|\R_{1,7}|=2$, $|\R_{1,n}|=0$ for $n\le6$).
\end{theorem}

\begin{theorem}\label{thm:B}
$a(n,k)\le a(n+1,k)$ for all $n\ge1$ and all $k\le2n-6$. More precisely, for $n\ge8$,
\[
a(n+1,2n-6)-a(n,2n-6)\ \ge\ 2p_2(n-6)-8(n-7)\ >0 ,
\]
and in fact $a(n+1,2n-6)-a(n,2n-6)=4p_2(n-5)+2p_2(n-6)-4(n-7)$.
\end{theorem}

The exact formula in Theorem~\ref{thm:B} says that the
Linusson--Verkama difference \eqref{eq:LV} acquires the correction
$-4(n-7)$ at defect one.  At defect two the same method gives twenty-one
inflation families and $|\R_{2,n}|=32n-214$ (Theorems~\ref{thm:C}
and~\ref{thm:D}).

Beyond defect two, Section~\ref{sec:skel} contracts maximal increasing
consecutive runs.  The resulting skeleton $\sigma$ and block vector $b$
describe a residual exactly by
$\sum_{(i,j)\in\Inv(\sigma)}b_ib_j=2n-7+\delta$ and explicit lower bounds
on the blocks in $D(\sigma)$ (Theorem~\ref{thm:crit}).  Proposition
\ref{prop:master} rewrites this as one quadratic equation in excess block
sizes.  Theorem~\ref{thm:cover} makes the skeleton set finite at each defect,
Theorem~\ref{thm:deladmissible} supplies the sharp all-defect skeleton bound
by strong induction, and Theorems~\ref{thm:poly} and~\ref{thm:deg2} give the
following unconditional polynomiality statement.

\begin{theorem}\label{conj:poly}
For every fixed $\delta\ge1$, $|\R_{\delta,n}|$ agrees for all large $n$ with a polynomial in $n$ of degree at most two (Theorems~\ref{thm:poly} and~\ref{thm:deg2}); by Proposition~\ref{prop:polys} the degree is $1$ for $\delta\le2$ and $2$ for $3\le\delta\le10$.
\end{theorem}

That $|\R_{\delta,n}|$ is at least \emph{bounded} by a polynomial in $n$ is proved here (Corollary~\ref{cor:polybound}), and since $p_2(n-7+\delta)$ grows like $\exp(\frac{2\pi}{\sqrt3}\sqrt n)$ this already yields \eqref{eq:conj} at every fixed defect for all sufficiently large $n$ (Theorem~\ref{thm:largen}). Theorem~\ref{conj:poly} identifies the shape of $|\R_{\delta,n}|$ but not its coefficients, so it does not by itself supply a workable threshold in place of the very large one that the crude bound gives.

Section~\ref{sec:uniform} determines the quadratic coefficients uniformly,
extends the fixed strip from defect ten to defect thirteen, and obtains the
growing square-root strips.  These are separate consequences: the fixed-strip
extension uses only the completed defect-$11,12,13$ catalogues and comparisons,
whereas the uniform leading coefficient does not determine the missing lower
coefficients.

Since the paper mixes proofs, machine checks and conjectures, we summarise the status of its statements.
\begin{center}
\begin{tabular}{@{}p{0.54\textwidth}p{0.40\textwidth}@{}}
\toprule
statement & status\\
\midrule
Theorem~\ref{thm:A}, Lemmas~\ref{lem:bdry}, \ref{lem:onecut}, \ref{lem:twodrop},
\ref{lem:threecut}, \ref{lem:runcover}, \ref{lem:path},
Proposition~\ref{prop:master}, Theorem~\ref{thm:crit} & proved\\
Theorems~\ref{thm:B}, \ref{conj:poly}, \ref{thm:C}, \ref{thm:D}, \ref{thm:gamma0},
\ref{thm:hubs}, \ref{thm:cover}, \ref{thm:largen}, \ref{thm:poly}, \ref{thm:deg2},
\ref{thm:deladmissible},
Lemmas~\ref{lem:noncut}, \ref{lem:leaf}, \ref{lem:samenbhd}, \ref{lem:threeclasses},
Corollaries~\ref{cor:polybound}, \ref{cor:mindeg}, \ref{cor:2n3},
Propositions~\ref{prop:base}, \ref{prop:catfromskel},
\ref{prop:Kind}, \ref{prop:delfree}
& proved, computer-assisted\\
Theorems~\ref{thm:uniformA}, \ref{thm:headline} and
Proposition~\ref{prop:rank3onset} & proved, computer-assisted; new certificate
bundle cited below\\
Propositions~\ref{prop:catalogue}, \ref{prop:polys} & finite computation, stated over its range\\
Corollary~\ref{cor:2n} & hypothesis supplied by Theorem~\ref{thm:deladmissible}\\
Conjectures~\ref{conj:cat}, \ref{conj:del} & proved for $\delta\le10$
(Theorem~\ref{thm:deladmissible})\\
Conjecture~\ref{conj:skel} & proved for every $\delta$ by
Theorem~\ref{thm:deladmissible} and strong induction\\
Conjectures~\ref{conj:hubs}, \ref{conj:K}, \ref{conj:Kdel} & conjectural\\
\bottomrule
\end{tabular}
\end{center}
The phrase \emph{computer-assisted} is used in the precise sense of Section~\ref{sec:data}: the proof is
complete except for finitely many exhaustive searches over $\mathfrak S_6$, $\mathfrak S_7$,
$\mathfrak S_8$ or over explicitly generated finite lists.  The original
checks are in the archive cited in Section~\ref{sec:data}.  The additional
marked-core, catalogue, rational-sum and boundary checks are supplied in the
archival supplement cited below.

\section{The Linusson--Verkama structure}\label{sec:LV}

We recall the notions and results of \cite{LV} that we use, in their notation. Fix $\pi\in\Av_n(1324)$ and write $a=\pi_1$, $b=\pi_n$, $p=\pi^{-1}(1)$, $q=\pi^{-1}(n)$. Suppose $a<b$ and $p<q$. The \emph{northwestern region} is $\NW=\{i<p:\pi_i>b\}$ and the \emph{southeastern region} is $\SE=\{j>q:\pi_j<a\}$. For $i\in\NW$ the \emph{southern region} relative to $i$ is $\Sreg(i)=\{j: i<j<q,\ 1<\pi_j<a\}$ and the \emph{eastern region} is $\Ereg=\{j:q<j<n,\ a<\pi_j<b\}$ (it does not depend on $i$). The \emph{central region} $\{j:i<j<q,\ a<\pi_j<b\}$ is empty, since together with $\pi_1,\pi_i,n$ a point in it would form a $1324$.

\begin{lemma}[{\cite[Lemmas 2.6, 2.7, 2.11]{LV}}]\label{lem:LV}
Let $\pi\in\Av^k_n(1324)$ be indecomposable.
\begin{enumerate}
\item If $a>b$ and $k\le2n-5$, then $\pi\setminus\pi_1$ or $\pi\setminus\pi_n$ is decomposable.
\item If $a<b$ and $p<q$, then $\NW\cup\SE\ne\varnothing$; and if moreover $k\le2n-6$, then $\NW=\varnothing$ or $\SE=\varnothing$.
\end{enumerate}
The map $\pi\mapsto\pi^{-1}$ preserves $1324$-avoidance, $\inv$, indecomposability and almost decomposability, exchanges the conditions $a<b$ and $p<q$, and exchanges $\NW$ and $\SE$.
\end{lemma}

From now on assume $a<b$, $p<q$, $\SE=\varnothing$ and $\NW=\{i_1<\dots<i_m\}\ne\varnothing$, and write $v_l=\pi_{i_l}$. The values $v_1<\dots<v_m$ increase: if $v_x>v_y$ for some $x<y$ then $\pi_1,v_x,v_y,n$ is a $1324$. For a value $w$ put
\[
\Ereg(w)=\{j:\ q<j<n,\ a<\pi_j<w\},
\]
so that $\Ereg(b)$ is the eastern region of \cite{LV} and $\Ereg(b)\subseteq\Ereg(v_1)\subseteq\dots\subseteq\Ereg(v_m)$.

\begin{lemma}\label{lem:deletions}
(i) If $\Sreg(i_1)=\varnothing$ then $\pi\setminus1$ is decomposable. (ii) If $\Ereg(v_m)=\varnothing$ then $\pi\setminus\pi_n$ is decomposable.
\end{lemma}
\begin{proof}
(i) is the first half of the proof of \cite[Lemma~2.10]{LV}; we repeat it since we need to see that it uses only $\SE=\varnothing$. Let $\sigma=\pi\setminus1$; we claim that the first $i_1-1$ entries of $\sigma$ are its $i_1-1$ smallest values, which makes $\sigma$ decomposable ($i_1\ge2$ because $\pi_1=a<b$). Otherwise there are positions $j_1<i_1<j_2$ of $\pi$, $j_2\ne p$, with $\pi_{j_1}>\pi_{j_2}$. Since $i_1$ is the first northwestern point, $\pi_{j_1}\le b$, hence $\pi_{j_1}<b$. If $\pi_{j_2}<a$ then $j_2$ would lie in $\Sreg(i_1)$ (if $j_2<q$) or in $\SE$ (if $j_2>q$); so $a\le\pi_{j_2}<\pi_{j_1}<b$, whence $j_1\ge2$, $a<\pi_{j_2}$, and $\pi_1,\pi_{j_1},\pi_{j_2},\pi_n$ is a $1324$.

(ii) Let $t$ be the first position with $\pi_t>v_m$ (it exists since $\pi_q=n$). Then $t>p$: a position $t<p$ with $\pi_t>v_m>b$ would be a northwestern point with a value exceeding $v_m$. Every position $l$ with $t<l<n$ carries a value exceeding $v_m$: a value $\pi_l<a$ lies in $\SE$ if $l>q$ and gives the $1324$ $1,\pi_t,\pi_l,n$ if $l<q$; a value in $(a,b)$ lies in the central region if $l<q$ and in $\Ereg(v_m)$ if $l>q$; a value in $(b,v_m)$ at a position $l<q$ (necessarily $l>p$) gives the $1324$ $\pi_1,v_m,\pi_l,n$, and at a position $l>q$ it lies in $\Ereg(v_m)$; the values $a,b,v_m$ sit at $1,n,i_m<t$. Hence in $\pi\setminus\pi_n$ the positions before $t$ carry exactly the values below $v_m$ other than $b$, i.e.\ the $t-1$ smallest values, and $2\le t\le q\le n-1$; so $\pi\setminus\pi_n$ is decomposable.
\end{proof}

The following identity makes the inversion count of \cite[Lemma~2.8]{LV} exact by recording its slack.

\begin{lemma}[seven-index identity]\label{lem:seven}
Let $i\in\NW$ with $v=\pi_i$, and let $j_1\in\Sreg(i)$, $j_2\in\Ereg(v)$. Write $L(j)=\#\{l<j:\pi_l>\pi_j\}$. Then
\[
\inv(\pi)=2n-6+s_1+s_2+s_3+s_4+s_5+u ,
\]
where
\[
\begin{gathered}
s_1=(a-1)+(n-q)-(v-i),\qquad s_2=p-1-i,\qquad s_3=L(j_1)-i,\\ s_4=L(j_2)-(n-v+1),\qquad s_5=v-b-1
\end{gathered}
\]
are all nonnegative and $u\ge0$ is the number of inversions $(x,y)$ of $\pi$ ($x<y$, $\pi_x>\pi_y$) with $x\notin\{1,i,q\}$ and $y\notin\{p,j_1,j_2,n\}$.
\end{lemma}

\begin{proof}
Write $R(l)=\#\{m>l:\pi_m<\pi_l\}$ for the number of right-inversions of position $l$. We have $R(1)=a-1$, $R(q)=n-q$, $L(p)=p-1$, $L(n)=n-b$. No entry before position $i$ exceeds $v$ (with $\pi_1$, $v$ and $n$ such an entry would give a $1324$), so $R(i)=(v-1)-(i-1)=v-i$. Every entry before $j_1$ with position $\le i$ exceeds $\pi_{j_1}$: for $l<i$ with $\pi_l<\pi_{j_1}$ the entries $\pi_l,v,\pi_{j_1},n$ form a $1324$, and $v>b>\pi_{j_1}$; hence $L(j_1)\ge i$. Every value $w\ge v$ lies before $j_2$ and exceeds $\pi_{j_2}$: $v$ and $n$ do, and if $v<w<n$ sat after $j_2$ then $\pi_1,v,\pi_{j_2},w$ would be a $1324$; hence $L(j_2)\ge n-v+1$. Right-inversions of $i$ are entries after $i$ with value below $v$; such a value is either at most $a-1$ (at most $a-1$ possibilities) or lies in $(a,v)$, and then its position exceeds $q$ (a position in $(i,q)$ would give the $1324$ $\pi_1,v,\cdot,n$), so there are at most $n-q$ of them; thus $s_1\ge0$. Finally $s_2\ge0$ since $i<p$, and $s_5\ge0$ since $v>b$.

Now sum $R(1)+R(i)+R(q)+L(p)+L(j_1)+L(j_2)+L(n)$. An inversion $(x,y)$ ($x<y$, $\pi_x>\pi_y$) is counted once for each of the two conditions $x\in\{1,i,q\}$ and $y\in\{p,j_1,j_2,n\}$ that it satisfies; the inversions satisfying both are exactly $(1,p),(1,j_1),(i,p),(i,j_1),(i,j_2),(i,n),(q,j_2),(q,n)$ --- eight pairs (the remaining candidates $(1,j_2),(1,n),(q,p),(q,j_1)$ are not inversions or not in the right order), and the inversions satisfying neither are those counted by $u$. Hence
\[
\inv(\pi)=R(1)+R(i)+R(q)+L(p)+L(j_1)+L(j_2)+L(n)-8+u ,
\]
and substituting $R(1)+R(q)=(v-i)+s_1$, $L(p)=i+s_2$, $L(j_1)=i+s_3$, $L(j_2)=(n-v+1)+s_4$, $L(n)=(n-v+1)+s_5$ gives $2(v-i)+2i+2(n-v+1)-8+\sum s+u=2n-6+\sum s+u$.
\end{proof}

\begin{lemma}[inflations]\label{lem:inflate}
Let $S$ be a permutation and $S[\iota_{m_1},\dots,\iota_{m_r}]$ its inflation by increasing blocks. If $S$ avoids $1324$, so does the inflation; if $S$ is indecomposable, so is the inflation.
\end{lemma}
\begin{proof}
An occurrence of $1324$ in the inflation with two entries in the same block is impossible: two entries of one block are increasing in position and in value, hence play the roles $(1,3)$, $(1,2)$, $(1,4)$, $(3,4)$ or $(2,4)$, and in each case a third entry of the occurrence lies between them in position, hence in the same block, and then its value is not in the required order. So the occurrence meets four distinct blocks and induces an occurrence in $S$. A sum-decomposition cut of the inflation that falls strictly inside a block forces that block, together with all blocks before it, to carry the smallest values, hence yields a cut of $S$ after that block's point (or before the first point, which is not a cut); so a cut of the inflation induces a cut of $S$.
\end{proof}

\begin{lemma}[two northwestern points]\label{lem:M}
If $m\ge2$, $\Sreg(i_1)\ne\varnothing$ and $\Ereg(v_m)\ne\varnothing$, then $\inv(\pi)\ge2n-4$.
\end{lemma}

This sharpens \cite[Lemma~2.9]{LV}, which rules out the configuration under the hypothesis $k\le2n-7$.
\begin{proof}
Choose $j_1\in\Sreg(i_1)$ and $j_2\in\Ereg(v_m)$ and count, as in Lemma~\ref{lem:seven}, the inversions incident to the positions $1,i_m,q$ (right-inversions) and $p,j_1,j_2,n$ (left-inversions):
\[
\inv(\pi)\ \ge\ R(1)+R(i_m)+R(q)+L(p)+L(j_1)+L(j_2)+L(n)-D ,
\]
where $D=7+[i_m<j_1]$ counts the inversions $(1,p),(1,j_1),(i_m,p),(i_m,j_2),(i_m,n),(q,j_2),(q,n)$ and, if $i_m<j_1$, $(i_m,j_1)$. Here $R(1)=a-1$, $R(q)=n-q$, $L(p)=p-1\ge i_m$, $L(n)=n-b$, and $R(i_m)=v_m-i_m$ because no entry before $i_m$ exceeds $v_m$. Every position $l<i_1$ has $\pi_l>\pi_{j_1}$ (else $\pi_l,v_1,\pi_{j_1},n$ is a $1324$), and $v_m>\pi_{j_1}$, so $L(j_1)\ge i_1+[i_m<j_1]$. All values exceeding $v_m$ lie before $j_2$ (else $\pi_1,v_m,\pi_{j_2},w$ is a $1324$), and so does $v_1$, which exceeds $\pi_{j_2}$ if $\pi_{j_2}<v_1$; so $L(j_2)\ge(n-v_m+1)+[\pi_{j_2}<v_1]$. Finally the right-inversions of $i_m$ are values below $a$ lying after $i_m$ or positions after $q$; the values below $a$ that lie before $i_m$ are not among them, so $(a-1)+(n-q)\ge(v_m-i_m)+|\Sreg(i_1)\setminus\Sreg(i_m)|$. Adding up,
\[
\inv(\pi)\ \ge\ 2n-5-(i_m-i_1)+|\Sreg(i_1)\setminus\Sreg(i_m)|+[\pi_{j_2}<v_1]+(v_m-b-1).
\]
A position strictly between $i_1$ and $i_m$ carries either a northwestern value or a value in $(1,a)$ (a value in $(a,b)$ would lie in the central region, one in $(b,v_m)$ would be northwestern, one exceeding $v_m$ would contradict the monotonicity of the northwestern values), so $i_m-i_1-1=(m-2)+|\Sreg(i_1)\setminus\Sreg(i_m)|$. The interval $(b,v_m)$ contains the $m-1$ values $v_1,\dots,v_{m-1}$, and it contains $\pi_{j_2}$ as a further, non-northwestern value whenever $\pi_{j_2}>v_1$. Hence the right-hand side is at least $2n-5+1$.
\end{proof}

\section{Defect one}\label{sec:delta1}

\begin{proof}[Proof of Theorem~\ref{thm:A}]
Let $\pi\in\R_{1,n}$, so $\inv(\pi)=2n-6$, $\pi$ is indecomposable and none of $\pi\setminus1,\pi\setminus n,\pi\setminus\pi_1,\pi\setminus\pi_n$ is decomposable. By Lemma~\ref{lem:LV}(1) applied to $\pi$ and to $\pi^{-1}$ (note $2n-6\le2n-5$), $a<b$ and $p<q$. By Lemma~\ref{lem:LV}(2) exactly one of $\NW,\SE$ is nonempty; replacing $\pi$ by $\pi^{-1}$ if necessary we assume $\NW\ne\varnothing=\SE$ and show that $\pi=24153[\iota_A,1,\iota_C,\iota_D,\iota_E]$ with $A,D\ge1$, $C,E\ge2$, $(A-1)(C-2)+(D-1)(E-2)=0$; the inverse of such a permutation is $31524[\iota_C,\iota_A,\iota_E,1,\iota_D]$, which gives the second family.

Since $\pi$ is not almost decomposable, Lemma~\ref{lem:deletions} gives $\Sreg(i_1)\ne\varnothing$ and $\Ereg(v_m)\ne\varnothing$, and then Lemma~\ref{lem:M} gives $m=1$ because $\inv(\pi)=2n-6<2n-4$. Write $i=i_1$, $v=\pi_i$. Apply Lemma~\ref{lem:seven} with $j_1\in\Sreg(i)$, $j_2\in\Ereg(v)$: since $\inv(\pi)=2n-6$, all of $s_1,\dots,s_5,u$ vanish, for every admissible choice of $j_1,j_2$. We read off the structure.

\emph{$s_2=0$:} $p=i+1$.

\emph{$s_5=0$:} $v=b+1$, so that $\Ereg(v)=\Ereg(b)$ is the eastern region $\Ereg$ of \cite{LV}.

\emph{$s_1=0$:} every value in $\{1,\dots,a-1\}$ lies after $i$, and every position in $\{q+1,\dots,n\}$ carries a value below $v$. The former means that the values $2,\dots,a-1$ are exactly the southern points (positions after $q$ with small values are excluded by $\SE=\varnothing$); in particular $a\ge3$ since $\Sreg(i)\ne\varnothing$. The latter, with $\SE=\varnothing$, $v=b+1$ and the emptiness of the central region, means that the positions $q+1,\dots,n-1$ are exactly the eastern points and carry values in $(a,b)$.

\emph{Positions before $i$.} No entry before $i$ exceeds $v$, none is below $a$ (those values sit after $i$), and $b$ sits at $n$; so $\pi_2,\dots,\pi_{i-1}\in(a,b)$.

\emph{Positions between $i$ and $q$.} Besides $1$ at $p=i+1$ and the southern points, the only possible values are those exceeding $v$ (values in $(a,b)$ are excluded by the central-region argument and $(b,v)=\varnothing$). All values exceeding $v$, except $n$, lie in $(i,q)$: not before $i$ by the above, not after $q$ by $s_1=0$.

\emph{$s_3=0$ for every southern $j_1$:} the entries between $i$ and $j_1$ are below $\pi_{j_1}$. Hence the southern points are increasing, and every value exceeding $v$ in $(i,q)$ comes after all southern points. \emph{$s_4=0$ for every eastern $j_2$:} the only entries before $j_2$ exceeding $\pi_{j_2}$ are the values $\ge v$; hence the eastern points are increasing and every entry of $\pi_2\cdots\pi_{i-1}$ is smaller than every eastern value.

\emph{$u=0$.} Two entries of $\pi_2\cdots\pi_{i-1}$ in the wrong order, or two values exceeding $v$ in the wrong order inside $(i,q)$ (the later one different from $n$), would give an inversion counted by $u$; so both runs are increasing. Consequently
\[
\pi=\underbrace{a,a+1,\dots,a+i-2}_{A=i-1}\ \ \underbrace{b+1}_{1}\ \ \underbrace{1,2,\dots,a-1}_{C=a-1}\ \ \underbrace{b+2,\dots,n}_{D=n-b-1}\ \ \underbrace{a+i-1,\dots,b}_{E=b-a-i+2}
\]
is the inflation $24153[\iota_A,1,\iota_C,\iota_D,\iota_E]$ with $A\ge1$, $C\ge2$, $D\ge1$ ($v<n$ because $n$ sits after $p$), $E\ge2$. Finally, if $C\ge3$ then some southern point $j$ is not the chosen $j_1$, and any entry $\pi_l$ with $2\le l\le i-1$ forms with $j$ an inversion counted by $u$; $u=0$ thus forces $A=1$. Likewise if $E\ge3$ then an entry exceeding $v$ other than $n$ would form with a non-chosen eastern point an inversion counted by $u$, so $D=1$. This is the equation $(A-1)(C-2)+(D-1)(E-2)=0$.

Conversely let $\pi=24153[\iota_A,1,\iota_C,\iota_D,\iota_E]$ with $A,D\ge1$, $C,E\ge2$ and $(A-1)(C-2)+(D-1)(E-2)=0$. It avoids $1324$ and is indecomposable by Lemma~\ref{lem:inflate}. Its inversion number is $(A+1)C+(D+1)E$, and $(A+1)C+(D+1)E=2(A+C+D+E+1)-6$ is equivalent to the equation. The four deletions are the inflations of $24153$ (or, when a block of size one disappears, of $2413$ or $3142$) by nonempty blocks, hence indecomposable; so $\pi$ is not almost decomposable, and $\pi\in\R_{1,n}$.

The count: with $A+C+D+E=n-1$, the four cases $(A,C)\in\{(1,\cdot),(\cdot,2)\}\times(D,E)\in\{(1,\cdot),(\cdot,2)\}$ each have $n-6$ solutions, the pairwise intersections have one solution each except the pair $\{A=1,D=1\}\cap\{C=2,E=2\}$ which needs $n=7$, and inclusion--exclusion gives $4(n-6)-4=4(n-7)$ for $n\ge8$. The inverse family has the same size and is disjoint (it has $\SE\ne\varnothing$).
\end{proof}

\begin{lemma}\label{lem:R1}
For every $n$ and $k$, no $\sigma\in\im(f\sqcup g)\subseteq\Av^k_{n+1}(1324)$ satisfies any of
\[
\sigma_1=n+1,\qquad \sigma_{n+1}=1,\qquad \sigma_2=n+1,\qquad \sigma_{n+1}=2 .
\]
The number of $\sigma\in\Av^k_{n+1}(1324)$ satisfying at least one of them is
\begin{equation}\label{eq:E}
E(n,k)=2a(n,k-n)+2a(n,k-n+1)-a(n-1,k-2n+3)-2a(n-1,k-2n+2)-a(n-1,k-2n+1),
\end{equation}
where $a(m,j)=0$ for $j<0$. In particular $E(n,k)=2a(n,k-n)+2a(n,k-n+1)$ for $k\le2n-4$.
\end{lemma}
\begin{proof}
Recall the three cases in the definition of $f$ \cite[Section~3]{LV}: if $\pi\setminus\pi_1$ is decomposable then $f(\pi)_1=\pi_1$ and $f(\pi)\setminus\pi_1=g(\pi\setminus\pi_1)$; if $\pi\setminus1$ is decomposable then $f(\pi)=f(\pi^{-1})^{-1}$; otherwise $f(\pi)=(f(\pi^{\mathrm{rc}}))^{\mathrm{rc}}$. Recall also that $\im(g)$ is exactly the set of permutations with at least three components.

\emph{The first two properties.} A permutation beginning with its maximum or ending with its minimum is indecomposable, hence not in $\im(g)$. In case~1, $f(\pi)$ begins with $\pi_1\le n$ and its remaining entries form a direct sum of at least three components, whose last entry is not its minimum; in cases~2 and~3 the property ``does not begin with the maximum and does not end with $1$'' is invariant under inversion and under reverse--complement, so induction on the recursion applies.

\emph{$f$ commutes with inversion.} Deleting the entry in position $i$ corresponds to deleting the value $i$ from the inverse, so $(\pi\setminus\pi_1)^{-1}=\pi^{-1}\setminus1$ and $(\pi\setminus1)^{-1}=\pi^{-1}\setminus\pi^{-1}_1$. Hence $\pi$ falls under case~1 exactly when $\pi^{-1}$ falls under case~2, and then $f(\pi^{-1})=f(\pi)^{-1}$ by definition; and $\pi$ falls under case~3 exactly when $\pi^{-1}$ does, in which case the claim for $\pi$ reduces, since reverse--complement commutes with inversion, to the claim for $\pi^{\mathrm{rc}}$, which falls under case~1 or~2 by the remark following the definition of $f$ in \cite[Section~3]{LV}. Since $\im(g)$ is invariant under inversion, so is $\im(f\sqcup g)$. Inversion exchanges $\sigma_1=n+1$ with $\sigma_{n+1}=1$ and $\sigma_2=n+1$ with $\sigma_{n+1}=2$, so for the remaining two properties it is enough to treat $\sigma_{n+1}=2$.

\emph{The property $\sigma_{n+1}=2$.} The last component of $\sigma$ occupies a suffix of the positions and carries a top interval of values; that interval contains $2$, so $\sigma$ has at most two components and $\sigma\notin\im(g)$. For the same reason $\sigma\setminus1$, which ends with its minimum, and $\sigma\setminus\sigma_1$, which ends with $2$ or with its minimum, have at most two components. Suppose $\sigma=f(\pi)$. Case~1 gives $\sigma\setminus\sigma_1=g(\pi\setminus\pi_1)$ with at least three components, a contradiction. In case~2 the permutation $\pi^{-1}$ falls under case~1, so $\sigma^{-1}\setminus\sigma^{-1}_1$ has at least three components; but $\sigma_{n+1}=2$ says $\sigma^{-1}_2=n+1$, so $\sigma^{-1}\setminus\sigma^{-1}_1$ begins with its maximum and is indecomposable, a contradiction. So case~3 applies: $\sigma=\tau^{\mathrm{rc}}$ with $\tau=f(\rho)$ and $\rho=\pi^{\mathrm{rc}}$, and $\sigma_{n+1}=n+2-\tau_1$ forces $\tau_1=n$. By the remark following the definition of $f$ in \cite[Section~3]{LV}, $\rho\setminus1$ or $\rho\setminus\rho_1$ is decomposable. If $\rho\setminus\rho_1$ is decomposable then $\tau_1=\rho_1=n$, i.e.\ $\pi_n=1$; but $\rho\setminus\rho_1$ decomposable means $\pi\setminus\pi_n=\pi\setminus1$ is decomposable, so $\pi$ falls under case~2, not case~3. If $\rho\setminus1$ is decomposable then $\tau=\upsilon^{-1}$ with $\upsilon=f(\rho^{-1})$ given by case~1, and $\tau_1$ is the position of the value $1$ in $\upsilon$. Now $\upsilon\setminus\upsilon_1=g(\rho^{-1}\setminus\rho^{-1}_1)$ is a direct sum $\gamma\dsum1\dsum\cdots$ of at least three components and of length $n$, so $|\gamma|\le n-2$ and the value $1$ of $\upsilon\setminus\upsilon_1$ lies in $\gamma$, at a position at most $n-2$; hence the value $1$ of $\upsilon$ lies at a position at most $n-1$ and $\tau_1\le n-1$, a contradiction.

\emph{The count.} None of the four distinguished entries can take part in an occurrence of $1324$: an entry placed first or second is too early to be the $4$ and is not the smallest, and an entry placed last is too late to be the $1$, the $3$ or the $2$ and is not the largest. Deleting it is therefore a bijection of the four sets onto $\Av^{k-n}_n(1324)$, $\Av^{k-n}_n(1324)$, $\Av^{k-n+1}_n(1324)$ and $\Av^{k-n+1}_n(1324)$, the entry contributing $n$, $n$, $n-1$ and $n-1$ inversions respectively. Two of the four conditions can hold at once only for the four pairs consisting of one condition on the left and one on the right, and deleting both entries identifies those with $\Av^{k-2n+1}_{n-1}$, $\Av^{k-2n+2}_{n-1}$, $\Av^{k-2n+2}_{n-1}$ and $\Av^{k-2n+3}_{n-1}$; no three can hold simultaneously. Inclusion--exclusion gives \eqref{eq:E}.
\end{proof}

\begin{lemma}[boundary form of {\cite[Lemma~4.2]{LV}}]\label{lem:bdry}
Let $\sigma\in\Av^k_m(1324)$ be indecomposable with $k\le2m-8$.
\begin{enumerate}
\item If $\sigma\setminus\sigma_1$ has exactly two components, then $\sigma\setminus m$ or
$\sigma\setminus\sigma_m$ has at least four components.
\item At least one of $\sigma\setminus1$, $\sigma\setminus m$, $\sigma\setminus\sigma_1$,
$\sigma\setminus\sigma_m$ has at least three components.
\end{enumerate}
\end{lemma}

\begin{proof}
(i) Write $\tau=\sigma\setminus\sigma_1=\sigma^{(1)}\dsum\sigma^{(2)}$, put $s=|\sigma^{(1)}|$ and
$\varepsilon=\sigma_1-s$. The first component of $\tau$ occupies the positions $2,\dots,s+1$ of $\sigma$ and carries
the $s$ smallest values of $\sigma$ other than $\sigma_1$; if $\sigma_1\le s$ these would be
$\{1,\dots,s+1\}\setminus\{\sigma_1\}$, so the first $s+1$ positions of $\sigma$ would carry the values
$1,\dots,s+1$ and $\sigma$ would be decomposable. Hence $\varepsilon\ge1$. Since $\sigma_1$ is inverted with exactly
the $\sigma_1-1$ values below it and $\sigma^{(1)}$ is indecomposable,
\[
\inv\sigma^{(2)}=k-(\sigma_1-1)-\inv\sigma^{(1)}\le(2m-8)-(s+\varepsilon-1)-(s-1)=2|\sigma^{(2)}|-4-\varepsilon
\le2|\sigma^{(2)}|-5 .
\]
Comparing with $\inv\sigma^{(2)}\ge|\sigma^{(2)}|-1$ gives $\varepsilon\le|\sigma^{(2)}|-3$, so
$\sigma_1=s+\varepsilon\le m-4$; in particular $\sigma_1\ne m$ and the largest value of $\tau$ is the value $m$ of
$\sigma$. By \cite[Section~3]{CJS} and \cite[Lemma~2.1]{Meng}, $\sigma^{(1)}$ is an indecomposable
$132$-avoider and $\sigma^{(2)}$ an indecomposable $213$-avoider, so \cite[Lemma~4.1]{LV} applied to
$(\sigma^{(2)})^{\mathrm{rc}}$ (note $213=\mathrm{rc}(132)$) gives $\sigma^{(2)}_1=|\sigma^{(2)}|$ or
$\sigma^{(2)}_{|\sigma^{(2)}|}=1$.

Suppose $\sigma^{(2)}_1=|\sigma^{(2)}|$, that is, $\sigma^{(2)}$ begins with its maximum, which is the value
$m$ of $\sigma$. Deleting it removes exactly $|\sigma^{(2)}|-1$ inversions, so
$\inv(\sigma^{(2)}\setminus\sigma^{(2)}_1)\le|\sigma^{(2)}|-3-\varepsilon$, and since a permutation of length $l$
with $c$ components has at least $l-c$ inversions,
\[
\comp\big(\sigma^{(2)}\setminus\sigma^{(2)}_1\big)\ \ge\ (|\sigma^{(2)}|-1)-(|\sigma^{(2)}|-3-\varepsilon)=\varepsilon+2 .
\]
Therefore $(\sigma\setminus m)\setminus\sigma_1=\sigma^{(1)}\dsum(\sigma^{(2)}\setminus\sigma^{(2)}_1)$ has at
least $\varepsilon+3$ components, the first of which is $\sigma^{(1)}$. Restoring $\sigma_1$ at the front merges into
a single component those components that carry a value smaller than $\sigma_1=s+\varepsilon$, namely $\sigma^{(1)}$
and at most the $\varepsilon-1$ components meeting $\{s+1,\dots,s+\varepsilon-1\}$; the count therefore drops by at most
$\varepsilon-1$ and $\comp(\sigma\setminus m)\ge4$. If instead $\sigma^{(2)}$ ends with its minimum, that entry is
the last entry of $\tau$, i.e.\ $\sigma_m$, and the same computation applied to
$(\sigma\setminus\sigma_m)\setminus\sigma_1$ gives $\comp(\sigma\setminus\sigma_m)\ge4$.

(ii) By \cite[Theorem~2.5]{LV} an indecomposable member of $\Av^k_m(1324)$ with $k\le2m-7$ is almost
decomposable, so $\comp(\sigma\setminus x)\ge2$ for some $x\in\{1,m,\sigma_1,\sigma_m\}$. If that deletion has
at least three components we are done. Otherwise it has exactly two, and we apply (i): the number of
components is invariant under $\sigma\mapsto\sigma^{-1}$ and $\sigma\mapsto\sigma^{\mathrm{rc}}$, both of which
preserve $\Av(1324)$, $\inv$ and indecomposability while permuting the four deletions transitively
($(\sigma\setminus\sigma_1)^{-1}=\sigma^{-1}\setminus1$ and
$(\sigma\setminus\sigma_1)^{\mathrm{rc}}=\sigma^{\mathrm{rc}}\setminus\sigma^{\mathrm{rc}}_m$), so we may
assume the deletion with two components is $\sigma\setminus\sigma_1$.
\end{proof}

Both statements were also verified directly for all $m\le10$, on the $1026$ indecomposable
$1324$-avoiders with $k\le2m-8$ (check C6 of Section~\ref{sec:data}); the bound four in (i) is attained,
in $28$ of the $94$ instances of the hypothesis of (i).

\begin{proof}[Proof of Theorem~\ref{thm:B}]
For $k\le2n-7$ this is \cite{LV}. Let $k=2n-6$. For $n\le7$ the inequality is a finite check (Section~\ref{sec:data}). For $n\ge8$, by \eqref{eq:identity}, Lemma~\ref{lem:R1} and Theorem~\ref{thm:A},
\[
a(n+1,k)-a(n,k)\ \ge\ 2a(n,n-6)-8(n-7)=2p_2(n-6)-8(n-7),
\]
using $a(n,n-6)=p_2(n-6)$ \cite[Proposition~15]{CJS}. The right-hand side is $2,4,16,40,\dots$ for $n=8,9,10,11$ and increases with $n$, since $p_2(m+1)-p_2(m)\ge4$ for $m\ge2$.

For the exact value we describe the complement of the image completely. Linusson and Verkama \cite[Section~4]{LV} show that this complement is the disjoint union of: $\mathrm{R1}$, the permutations beginning with $n+1$ or ending with $1$; $\mathrm{R2a}$, those with $\sigma_2=n+1$ or $\sigma_{n+1}=2$ (not in $\mathrm{R1}$); $\mathrm{R2b}$, those with at most two components, not in $\mathrm{R1}\cup\mathrm{R2a}$, for which some deletion among $\sigma\setminus1,\sigma\setminus(n+1),\sigma\setminus\sigma_1,\sigma\setminus\sigma_{n+1}$ has at least three components but which are not of the form $f(\pi)$; and $\mathrm{R3}$, those with at most two components none of whose four deletions has three components. Their arguments give $|\mathrm{R1}|=2p_2(k-n)$ and $|\mathrm{R2a}|=|\mathrm{R2b}|=2p_2(k-n+1)$, and these counts remain valid at $k=2n-6$. Indeed each of the three is obtained by deleting a distinguished entry: $\mathrm{R1}$ and $\mathrm{R2a}$ are put in bijection with $\Av^{k-n}_n(1324)$ and $\Av^{k-n+1}_n(1324)$, and $\mathrm{R2b}$ with $\Av^{k-n+1}_{n-2}(1324)$ by $\pi\mapsto\pi\setminus\{\pi_1,\pi_n\}$ \cite[Lemma~4.4]{LV}; the hypothesis $k\le2n-7$ enters these three counts only through the plateau identity $a(m,j)=p_2(j)$ for $m\ge j+2$ \cite[Proposition~15]{CJS}, applied with $(m,j)=(n,k-n)$, $(n,k-n+1)$ and $(n-2,k-n+1)$, which require $k\le2n-2$, $k\le2n-3$ and $k\le2n-5$ respectively --- the last of these is exactly the inequality $n-2-(k-n+1)\ge2$ used in the proof of \cite[Lemma~4.4]{LV}. All three counts, and the value $|\mathrm{R3}|=4(n-7)$ obtained below, were also verified directly for $8\le n\le11$ (Section~\ref{sec:data}), and \cite[Proposition~4.3]{LV} shows $\mathrm{R3}=\varnothing$ for $k\le2n-7$. At $k=2n-6$ we determine $\mathrm{R3}$. Let $\sigma\in\mathrm{R3}$, of length $m=n+1$ and with $2m-8$ inversions. If $\sigma$ were indecomposable then, by Lemma~\ref{lem:bdry}(ii), one of its four deletions would have at least three components, contradicting the definition of $\mathrm{R3}$. So $\sigma=\sigma^{(1)}\dsum\sigma^{(2)}$ with $\sigma^{(1)}$ an indecomposable $132$-avoider and $\sigma^{(2)}$ an indecomposable $213$-avoider (\cite[Section~3]{CJS}, \cite[Lemma~2.1]{Meng}). The conditions on the deletions say that $\sigma^{(1)}\setminus1$, $\sigma^{(1)}\setminus\sigma^{(1)}_1$ and the corresponding deletions of $\sigma^{(2)}$ are indecomposable; by \cite[Lemma~4.1]{LV} and its reverse--complement this forces $\inv\sigma^{(1)}\ge2|\sigma^{(1)}|-4$ and $\inv\sigma^{(2)}\ge2|\sigma^{(2)}|-4$, so both are equalities. An indecomposable $132$-avoider has the form $(\alpha\dsum1)\ssum\beta$ with $\beta\ne\varnothing$, $\alpha,\beta\in\Av(132)$, and $\inv=\inv\alpha+\inv\beta+(|\alpha|+1)|\beta|$; equality with $2(|\alpha|+|\beta|+1)-4$ reads $\inv\alpha+\inv\beta+(|\alpha|-1)(|\beta|-2)=0$. If $|\beta|=1$ the permutation ends with $1$, if $|\alpha|=0$ it begins with its maximum and $\beta$ has $|\beta|-2$ inversions, hence is decomposable; both are excluded by the deletion conditions. Otherwise $\alpha,\beta$ are increasing and $(|\alpha|-1)(|\beta|-2)=0$, i.e. $\sigma^{(1)}=\iota_x\ssum\iota_y$ with $x,y\ge2$, $(x-2)(y-2)=0$; symmetrically $\sigma^{(2)}=\iota_z\ssum\iota_w$ with $z,w\ge2$, $(z-2)(w-2)=0$. Conversely let $\sigma=(\iota_x\ssum\iota_y)\dsum(\iota_z\ssum\iota_w)$ with $x,y,z,w\ge2$, $(x-2)(y-2)=(z-2)(w-2)=0$ and $x+y+z+w=n+1$. Such a $\sigma$ has $2n-6$ inversions, avoids $1324$, and has exactly two components, as do all four of its deletions $\sigma\setminus1$, $\sigma\setminus(n+1)$, $\sigma\setminus\sigma_1$, $\sigma\setminus\sigma_{n+1}$. It remains to check that $\sigma$ lies in the complement of $\im(f\sqcup g)$. Every member of $\im(g)$ has at least three components, so $\sigma\notin\im(g)$. Suppose $\sigma=f(\pi)$. In case~1 of the definition of $f$, $\sigma\setminus\sigma_1=g(\pi\setminus\pi_1)$ has at least three components, a contradiction. In case~2 the permutation $\pi^{-1}$ falls under case~1 and $f(\pi^{-1})=f(\pi)^{-1}$, so $\sigma^{-1}\setminus\sigma^{-1}_1=(\sigma\setminus1)^{-1}$ has at least three components; but the number of components is invariant under inversion, so $\sigma\setminus1$ has at least three components, again a contradiction. In case~3 we have $\sigma=\tau^{\mathrm{rc}}$ with $\tau=f(\rho)$ falling under case~1 or case~2, so $\tau\setminus\tau_1$ or $\tau\setminus1$ has at least three components; reverse--complement preserves the number of components and carries these to $\sigma\setminus\sigma_{n+1}$ and $\sigma\setminus(n+1)$, a contradiction. Hence $\sigma\in\mathrm{R3}$. Counting as in Theorem~\ref{thm:A}, $|\mathrm{R3}|=4(n-7)$ for $n\ge8$. Therefore
\[
a(n+1,2n-6)-a(n,2n-6)=2p_2(n-6)+4p_2(n-5)+4(n-7)-8(n-7) ,
\]
as claimed.
\end{proof}

\begin{remark}
Theorem~\ref{thm:A} identifies the residuals at defect one with the two skeletons $24153$ and $31524=24153^{-1}$; the permutation $3612745=24153[1,1,\iota_2,1,\iota_2]$ used in \cite{LV} to show that $2n-7$ is sharp is the smallest member. The set $\mathrm{R3}$ consists of the permutations $(\iota_x\ssum\iota_y)\dsum(\iota_z\ssum\iota_w)$; there is an obvious bijection $\R^{\NW}_{1,n}\to\mathrm{R3}$, $24153[\iota_A,1,\iota_C,\iota_D,\iota_E]\mapsto(\iota_{A+1}\ssum\iota_C)\dsum(\iota_{D+1}\ssum\iota_E)$, which explains the coefficient $4(n-7)$; the other half of the residuals must be absorbed by $\mathrm{R1}\cup\mathrm{R2a}\cup\mathrm{R2b}$.
\end{remark}

\section{Defect two}\label{sec:delta2}

Throughout this section $\pi\in\R_{2,n}$, so $\inv(\pi)=2n-5$, $\pi$ is indecomposable and none of the four extremal deletions is decomposable. By Lemma~\ref{lem:LV}(1) (applied to $\pi$ and $\pi^{-1}$) we have $a<b$ and $p<q$, and by Lemma~\ref{lem:LV}(2) $\NW\cup\SE\ne\varnothing$; both may now be nonempty. We first record the structural facts that hold for a single northwestern point without any assumption on the slack, then the identity governing the case $\NW\ne\varnothing\ne\SE$, and then the classification.

\begin{lemma}\label{lem:single}
Suppose $\SE=\varnothing$, $\NW=\{i\}$, $v=\pi_i$, $\Sreg=\Sreg(i)\ne\varnothing$ and $\Ereg(v)\ne\varnothing$; let $j_1$ be the first southern point and $j_2$ the first point of $\Ereg(v)$. Then:
\begin{enumerate}
\item[(a)] the positions $2,\dots,i-1$ carry values in $(1,a)\cup(a,b)$;
\item[(b)] the positions in $(i,q)$ other than $p$ carry values in $(1,a)\cup(v,n)$; the positions in $(q,n)$ carry values exceeding $a$; the values in $(b,v)$ lie after $q$;
\item[(c)] the southern points after $p$ are increasing; the \emph{top entries} (positions in $(p,q)$ with values exceeding $v$) are increasing and follow all southern points; the prefix values in $(a,b)$ are increasing, the values of $\Ereg(b)$ are increasing, and every prefix value in $(a,b)$ is below every value of $\Ereg(b)$;
\item[(d)] $s_3=s_4=0$ in Lemma~\ref{lem:seven}, and
\[
\begin{gathered}
s_1=\#\{\text{values below }a\text{ before }i\}+\#\{\text{positions in }(q,n)\text{ with values exceeding }v\},\\
s_2=p-i-1,\qquad s_5=\#\{\text{values in }(b,v)\}.
\end{gathered}
\]
\end{enumerate}
\end{lemma}
\begin{proof}
(a) No entry before $i$ exceeds $v$; a value in $(b,v)$ before $i$ would be a second northwestern point; $1$, $a$ and $b$ sit at $p>i$, $1$ and $n$. (b) A value in $(a,v)$ at a position in $(i,q)$ forms the $1324$ $\pi_1,v,\cdot,n$; positions after $q$ carry values above $a$ because $\SE=\varnothing$; so a value in $(b,v)$ is after $q$ by (a). (c) Two southern points after $p$ in the wrong order form a $1324$ with $1$ and $n$; two top entries in the wrong order form one with $\pi_1$ and $n$; a top entry before a southern point $s$ forms $1,t,s,n$; two prefix values in $(a,b)$ in the wrong order, or two values of $\Ereg(b)$ in the wrong order, or a prefix value above a value of $\Ereg(b)$, form a $1324$ with $\pi_1$ and $\pi_n$. (d) The entries strictly between $i$ and $j_1$ are not southern ($j_1$ is the first), not in $(a,v)$ by (b), not above $v$ (a position before $p$ with such a value would be northwestern, and a top entry before $j_1$ is excluded by (c)); hence $L(j_1)=i$. The entries before $j_2$ with values in $(\pi_{j_2},v)$: a prefix value in $(\pi_{j_2},v)$ lies in $(a,b)$ by (a), so $\pi_{j_2}<b$ and $\pi_1,\cdot,\pi_{j_2},\pi_n$ is a $1324$; the positions in $(i,q)$ carry no value in $(a,v)$; the positions in $(q,j_2)$ carry no value in $(a,v)$ because $j_2$ is the first such point. Hence $L(j_2)=n-v+1$. The formulas for $s_1,s_2,s_5$ restate the definitions: the right-inversions of $i$ are the values below $a$ that lie after $i$ together with the positions after $q$ with values below $v$.
\end{proof}

\begin{lemma}[northwestern and southeastern points]\label{lem:NWSE}
Suppose $\NW\ne\varnothing\ne\SE$. If $|\NW|\ge2$ or $|\SE|\ge2$ then $\inv(\pi)\ge2n-4$. If $\NW=\{i\}$ and $\SE=\{j\}$, with $v=\pi_i$ and $w=\pi_j$, then
\[
\inv(\pi)=2n-5+(p-i-1)+(v-b-1)+(a-w-1)+(j-q-1)+u ,
\]
where $u\ge0$ is the number of inversions $(x,y)$ with $x\notin\{1,i,q\}$ and $y\notin\{p,j,n\}$.
\end{lemma}
\begin{proof}
Let $i=i_m$ be the last northwestern point and $j$ the first southeastern point. No entry after $j$ is below $w$: such an entry $y$ would give the $1324$ $1,w,y,\pi_n$ (positions $p<j<\cdot<n$). Hence $L(j)=j-w$, and as before $R(i)=v-i$, $R(1)=a-1$, $R(q)=n-q$, $L(p)=p-1$, $L(n)=n-b$. The inversions counted twice in $R(1)+R(i)+R(q)+L(p)+L(j)+L(n)$ are $(1,p),(1,j),(i,p),(i,j),(i,n),(q,j),(q,n)$, and the inversions counted by neither type are those of $u$; so $\inv(\pi)=R(1)+R(i)+R(q)+L(p)+L(j)+L(n)-7+u$, which is the displayed identity. If $|\NW|\ge2$ then $v-b-1\ge1$ because the other northwestern values lie in $(b,v)$; if $|\SE|\ge2$ then $a-w-1\ge1$ symmetrically. In both cases the right-hand side is at least $2n-4$.
\end{proof}

\begin{theorem}\label{thm:C}
For every $n$, $\R_{2,n}$ is the union of the twenty-one families of inflations by increasing blocks listed in Table~\ref{tab:delta2} (twelve families and the inverses of the nine that are not self-inverse). Consequently $|\R_{2,n}|=32n-214$ for all $n\ge10$.
\end{theorem}

\begin{table}[ht]
\centering\footnotesize\renewcommand{\arraystretch}{1.2}
\caption{The residuals at defect two, up to inversion. Block sizes are $\ge1$ unless a larger lower bound is indicated; the count is the number of members of the family for $n\ge10$. Families marked $\ast$ (Case II, $\NW\ne\varnothing\ne\SE$) are self-inverse; each of the others contributes twice (the family and its inverse). In the cases $s_1=1$, (a) means that the value $a-1$ is the second entry and (b) that the value $n-1$ follows $n$.}\label{tab:delta2}
\resizebox{\textwidth}{!}{\begin{tabular}{@{}lllc@{}}
\toprule
case & family & constraints & count\\
\midrule
$u=1$ & $24153[\iota_A,1,\iota_C,\iota_D,\iota_E]$ & $C,E\ge2$, $(A-1)(C-2)+(D-1)(E-2)=1$ & $4$\\
$s_5=1$, $\Ereg(b)=\varnothing$ & $251643[\iota_A,1,\iota_C,\iota_D,1,1]$ & $C\ge2$, $(A-1)(C-2)=0$ & $2n-13$\\
$s_5=1$, $\Ereg(b)\ne\varnothing$ & $2617354[\iota_A,1,\iota_C,1,\iota_E,1,1]$ & $C\ge2$, $(A-1)(C-2)=0$ & $2n-15$\\
$s_2=1$, $|\Sreg|=1$ & $352164[\iota_A,1,1,1,\iota_D,\iota_E]$ & $E\ge2$, $(D-1)(E-2)=0$ & $2n-13$\\
$s_2=1$, $|\Sreg|\ge2$ & $4621375[1,1,1,1,\iota_C,\iota_D,\iota_E]$ & $E\ge2$, $(D-1)(E-2)=0$ & $2n-15$\\
$s_1=1$ (b), top entries & $2415763[\iota_A,1,\iota_C,\iota_D,1,1,\iota_2]$ & $C\ge2$, $(A-1)(C-2)=0$ & $2n-17$\\
$s_1=1$ (b), no top entries & $241653[\iota_A,1,\iota_C,1,1,\iota_2]$ & $C\ge2$, $(A-1)(C-2)=0$ & $2$\\
$s_1=1$ (a), prefix & $3246175[1,1,\iota_B,1,\iota_2,\iota_D,\iota_E]$ & $E\ge2$, $(D-1)(E-2)=0$ & $2n-17$\\
$s_1=1$ (a), no prefix & $325164[1,1,1,\iota_2,\iota_D,\iota_E]$ & $E\ge2$, $(D-1)(E-2)=0$ & $2$\\
II $\ast$ & $351624[\iota_A,1,\iota_C,\iota_D,1,\iota_F]$ & $(A-1)(C-1)=(D-1)(F-1)=0$ & $4n-24$\\
II $\ast$ & $3517264[\iota_A,1,\iota_C,1,1,\iota_E,1]$ & $(A-1)(C-1)=0$ & $2n-13$\\
II $\ast$ & $4261735[1,\iota_B,1,1,\iota_D,1,\iota_F]$ & $(D-1)(F-1)=0$ & $2n-13$\\
\bottomrule
\end{tabular}}
\end{table}

\begin{proof}
\emph{Membership.} Each skeleton in Table~\ref{tab:delta2} avoids $1324$ and is indecomposable, so its inflations avoid $1324$ and are indecomposable (Lemma~\ref{lem:inflate}). Writing $\inv$ of an inflation as $\sum b_xb_y$ over the inversions $(x,y)$ of the skeleton, one checks that in each row the equation $\inv=2n-5$ is equivalent to the stated constraint (for instance $\inv(351624[\iota_A,1,\iota_C,\iota_D,1,\iota_F])=AC+A+C+1+F+D+DF$ and $n=A+C+D+F+2$, which gives $(A-1)(C-1)+(D-1)(F-1)=0$). Each of the four extremal deletions of a member removes one entry from a block; when that block has size at least two the result is an inflation of the same skeleton, and when it has size one the result is an inflation of the skeleton with that point deleted; that no member is almost decomposable follows from Theorem~\ref{thm:crit}, which is proved in Section~\ref{sec:skel} independently of this section: an inflation $\sigma[b]$ of a reduced, indecomposable, $1324$-avoiding $\sigma$ with $\inv=2n-5$ lies in $\R_{2,n}$ precisely when $b_i\ge2$ for every $i\in D(\sigma)$, and the sets $D(\sigma)$ of the twenty-one skeletons, printed by check C1 of the certificate (Section~\ref{sec:data}), are exactly the blocks that the constraint of each row requires to be nonsingletons. Hence every listed permutation lies in $\R_{2,n}$. Check C1 enumerates the members by that criterion and verifies, without assuming it, that each of them avoids $1324$, is indecomposable and is not almost decomposable, recording in passing the $43$ skeletons that the four extremal deletions produce, all of them indecomposable. The counts follow by inclusion--exclusion exactly as in Theorem~\ref{thm:A}; for instance in the first row the constraint forces $\{A,C\}=\{2,3\}$ with $(D-1)(E-2)=0$ or symmetrically, which has four solutions once $n\ge10$, and in the tenth row the four cases $A=1$ or $C=1$ combined with $D=1$ or $F=1$ have $n-5$ solutions each with four single overlaps. Summing the last column (with the factor two for the nine families that are not self-inverse) gives $32n-214$.

\emph{Completeness.} Let $\pi\in\R_{2,n}$. We distinguish whether one or both of $\NW$, $\SE$ are nonempty.

\emph{Case I: exactly one of $\NW,\SE$ is nonempty.} Replacing $\pi$ by $\pi^{-1}$ we may assume $\NW\ne\varnothing=\SE$. By Lemma~\ref{lem:deletions}, $\Sreg(i_1)\ne\varnothing$ and $\Ereg(v_m)\ne\varnothing$; by Lemma~\ref{lem:M}, $\NW=\{i\}$. Lemma~\ref{lem:seven} with the canonical $j_1,j_2$ of Lemma~\ref{lem:single} gives $s_1+s_2+s_5+u=1$, since $s_3=s_4=0$. Whenever the three quantities other than the one equal to $1$ vanish, the arguments of Theorem~\ref{thm:A} that only use these vanishings apply verbatim; we only describe what the unit of slack changes. Recall from Theorem~\ref{thm:A} that, when all of $s_1,s_2,s_5$ vanish, the permutation has the shape
\[
\pi=\ [\,a,\ \mathrm{prefix}\,]\,[\,v\,]\,[\,1\,]\,[\,\Sreg\,]\,[\,\mathrm{top},n\,]\,[\,\Ereg(b)\,]\,[\,b\,],\qquad v=b+1,
\]
with the prefix carrying the values $a+1,\dots$ (increasing), $\Sreg$ the values $2,\dots,a-1$ (increasing), the top entries the values $b+2,\dots,n$ (increasing), $\Ereg(b)$ the remaining values of $(a,b)$ (increasing), and that then $u$ equals $(A-1)(C-2)+(D-1)(E-2)$, the number of pairs (prefix entry, non-first southern point) and (top entry other than $n$, non-first point of $\Ereg(b)$), all of which are inversions counted by $u$, no other inversion being counted by $u$ (Lemma~\ref{lem:single}(c)).

\emph{$u=1$.} The shape is as displayed with $(A-1)(C-2)+(D-1)(E-2)=1$: the first row of the table.

\emph{$s_5=1$.} Then $v=b+2$ and the single value $b+1$ lies after $q$ (Lemma~\ref{lem:single}(a),(b)); since $s_1=0$, every position after $q$ carries a value below $v$, so the positions $q+1,\dots,n-1$ carry the points of $\Ereg(b)$ and the entry $b+1$, and $\Ereg(v)=\Ereg(b)\cup\{b+1\}$. If $\Ereg(b)\ne\varnothing$ then $b+1$ is not the first point of $\Ereg(v)$ (otherwise it would form with every point of $\Ereg(b)$ an inversion counted by $u$), hence it is the last one by Lemma~\ref{lem:single}(c). With $s_2=0$ the rest of the shape is as displayed; now $u$ counts, besides $(A-1)(C-2)$, the pairs (top entry other than $n$, point of $\Ereg(b)$ other than the first) and, when $\Ereg(b)\ne\varnothing$, the pairs (top entry other than $n$, entry $b+1$): $u=(A-1)(C-2)+(D-1)|\Ereg(b)|$. Thus $(A-1)(C-2)=0$, and $D=1$ when $\Ereg(b)\ne\varnothing$: rows two and three.

\emph{$s_2=1$.} Exactly one position $g$ lies strictly between $i$ and $p$; its value is below $a$ (a value above $v$ would be northwestern, a value in $(a,v)$ is excluded by Lemma~\ref{lem:single}(b), and $a$, $b$ are elsewhere), so $g=j_1$ is the first southern point and, by Lemma~\ref{lem:single}(c) and $u=0$ (an inversion between $g$ and a later southern point would be counted by $u$), $\pi_g=2$. If there are further southern points, every prefix entry would form with one of them an inversion counted by $u$; so either there are none (row four, $A$ free) or $A=1$ (row five). The top entries and $\Ereg(b)$ are as in the displayed shape with $(D-1)(E-2)=0$.

\emph{$s_1=1$.} Either (a) exactly one value $x<a$ lies before $i$, or (b) exactly one position after $q$ carries a value exceeding $v$. In case (a), every southern point $s$ has $\pi_s<x$ (otherwise $x,v,\pi_s,n$ is a $1324$), so $x=a-1$; a prefix entry above $a$ before $x$ would form with $x$ an inversion counted by $u$, so $x=\pi_2$; and since $x$ exceeds every southern value, $u=0$ forces $|\Sreg|=1$ (rows eight and nine, according to whether further prefix entries exist; the equation then gives $(D-1)(E-2)=0$). In case (b), let $W$ be the exceptional value after $q$. Every point $e$ of $\Ereg(b)$ lies after $W$ (otherwise $\pi_1,v,e,W$ is a $1324$), hence $W$ sits at position $q+1$; $u=0$ forces $|\Ereg(b)|=1$ (each further point of $\Ereg(b)$ would form with $W$ an inversion counted by $u$) and $W$ above every top entry other than $n$, i.e.\ $W=n-1$. This gives rows six and seven.

\emph{Case II: $\NW\ne\varnothing\ne\SE$.} By Lemma~\ref{lem:NWSE}, $\NW=\{i\}$, $\SE=\{j\}$, $p=i+1$, $v=b+1$, $w=a-1$, $j=q+1$ and $u=0$. Split the remaining entries into the prefix $P$ (positions $2,\dots,i-1$), the middle $M$ (positions $p+1,\dots,q-1$) and the tail $Z$ (positions $j+1,\dots,n-1$). Values: $P\subseteq(1,a)\cup(a,b)$ as in Lemma~\ref{lem:single}(a); $M\subseteq(1,a)\cup(v,n)$ (a value in $(a,v)$ gives $\pi_1,v,\cdot,n$); $Z\subseteq(a,b)\cup(v,n)$ (no entry after $j$ is below $w$). Write $P_S,P_{ab},M_S,M_T,Z_{ab},Z_T$ for the corresponding parts. Each part is increasing, $M_S$ precedes $M_T$, $Z_{ab}$ precedes $Z_T$, $P_S$ precedes $P_{ab}$ and $P_{ab}<Z_{ab}$, $P_S<M_S$, $M_T<Z_T$ as sets of values, all by $u=0$ or by a $1324$ with $\pi_1,\pi_n$ or $1,n$. Moreover the following pairs cannot both be nonempty: $Z_{ab},Z_T$ ($\pi_1,v,z_{ab},z_T$); $M_T,Z_T$ ($1,m_T,w,z_T$); $P_S,M_S$ ($p_S,v,m_S,n$); $P_S,Z_T$ ($p_S,v,w,z_T$); $P_S,P_{ab}$ ($p_S,p_{ab},w,\pi_n$); $P_{ab},M_S$ and $M_T,Z_{ab}$ (inversions counted by $u$). Consequently $\pi$ is
\[
\begin{gathered}
[\,a,P_{ab}\,][\,v\,][\,1,M_S\,][\,n\,][\,w\,][\,Z_T\,][\,b\,]\qquad\text{or}\qquad
[\,a\,][\,P_S\,][\,v\,][\,1\,][\,M_T,n\,][\,w\,][\,Z_{ab},b\,]\\
\text{or}\qquad[\,a,P_{ab}\,][\,v\,][\,1,M_S\,][\,M_T,n\,][\,w\,][\,Z_{ab},b\,],
\end{gathered}
\]
with the incompatibilities $P_{ab}=\varnothing$ or $M_S=\varnothing$, and $M_T=\varnothing$ or $Z_{ab}=\varnothing$: the last three rows of the table, where the constraints are exactly these incompatibilities. This completes the proof.
\end{proof}

\begin{theorem}\label{thm:D}
$a(n,k)\le a(n+1,k)$ for all $n\ge1$ and all $k\le2n-5$. For $n\ge10$,
\[
a(n+1,2n-5)-a(n,2n-5)\ \ge\ 2p_2(n-5)+2p_2(n-4)-(32n-214)\ >0 .
\]
\end{theorem}
\begin{proof}
For $k\le2n-6$ this is Theorem~\ref{thm:B}; let $k=2n-5$. By Lemma~\ref{lem:R1} the complement of $\im(f\sqcup g)$ has at least $E(n,k)=2p_2(n-5)+2p_2(n-4)$ elements, the correction terms in \eqref{eq:E} vanishing because $k\le2n-4$. Hence by \eqref{eq:identity} and Theorem~\ref{thm:C}, $a(n+1,k)-a(n,k)\ge2p_2(n-5)+2p_2(n-4)-(32n-214)$ for $n\ge10$; the right-hand side equals $96$ at $n=10$ and increases with $n$. For $n\le9$ the inequality is read off from the table of Section~\ref{sec:data}.
\end{proof}

\section{The skeleton reduction}\label{sec:skel}

The classifications of Sections~\ref{sec:delta1} and \ref{sec:delta2} both produced inflations of a fixed
finite list of patterns by increasing blocks, subject to one quadratic equation. This is not an accident of
small defect. In this section we show that the passage to the underlying pattern is an exact reduction: it
turns the description of $\R_{\delta,n}$, for every $\delta$ and every $n$ at once, into a quadratic
Diophantine problem over a set of patterns that depends on $\delta$ alone.

Write $\iota_m=12\cdots m$ and, for a permutation $\tau$ of length $t$ and $b\in\mathbb Z_{\ge1}^t$, let
$\tau[b]=\tau[\iota_{b_1},\dots,\iota_{b_t}]$ be the inflation of $\tau$ in which the $i$th entry is replaced by
an increasing run of $b_i$ entries. Call $\sigma$ \emph{reduced} if $\sigma_{i+1}\ne\sigma_i+1$ for all
$1\le i<|\sigma|$. Contracting the maximal runs of consecutive positions carrying consecutive values writes
every permutation $\pi$ uniquely as $\pi=\sigma[b]$ with $\sigma$ reduced; we call $\sigma=\sk(\pi)$ the
\emph{skeleton} of $\pi$ and $b$ its \emph{block vector}. This is the special case of the substitution decomposition \cite{AA} in which the intervals allowed are the
maximal increasing runs, and the skeleton is the quotient of $\pi$ by that interval system; what makes it
useful here is Theorem~\ref{thm:crit}, which says that membership in $\R_{\delta,n}$ is visible on the
quotient together with the block sizes. Let $\Inv(\tau)=\{(i,j):i<j,\ \tau_i>\tau_j\}$ and let
$d_i=d_i(\tau)=|\{j:(i,j)\in\Inv(\tau)\ \text{or}\ (j,i)\in\Inv(\tau)\}|$ be the degree of $i$ in the inversion
graph of $\tau$.

\begin{lemma}\label{lem:infl}
Let $\tau$ have length $t\ge1$ and let $b\in\mathbb Z_{\ge1}^t$. Then
\begin{enumerate}
\item[(1)] $\tau[b]$ contains $1324$ if and only if $\tau$ does;
\item[(2)] $\tau[b]$ is decomposable if and only if $\tau$ is decomposable, or $t=1$ and $b_1\ge2$;
\item[(3)] $\inv(\tau[b])=\sum_{(i,j)\in\Inv(\tau)}b_ib_j$.
\end{enumerate}
\end{lemma}

\begin{proof}
(3) is immediate: two entries of $\tau[b]$ form an inversion exactly when they lie in blocks $i<j$ with
$(i,j)\in\Inv(\tau)$, since each block is increasing.

(1) Since $\tau\preceq\tau[b]$, one direction is clear. Conversely, each block $B_i$ of $\tau[b]$ is an interval
of positions carrying an interval of values, in increasing order. We claim that no two entries of an occurrence
of $1324$ lie in one block; the occurrence then projects to distinct blocks and yields an occurrence of $1324$
in $\tau$. Label the entries of the occurrence $e_1e_2e_3e_4$, of values $v_1<v_3<v_2<v_4$. Two entries in a
common block must increase in both coordinates, which leaves the pairs $(e_1,e_2)$, $(e_1,e_3)$, $(e_1,e_4)$,
$(e_2,e_4)$, $(e_3,e_4)$. If $e_1,e_2\in B_i$ then $v_3\in[v_1,v_2]$ forces $e_3\in B_i$, and $B_i$ increasing
contradicts the fact that $e_2$ precedes $e_3$ while $v_2>v_3$. If $e_1,e_3\in B_i$ then $B_i$ contains all
positions in between, in particular $e_2$, whose value $v_2\notin[v_1,v_3]$. The pair $(e_1,e_4)$ puts all four
entries in $B_i$, contradicting that $B_i$ is increasing; $(e_2,e_4)$ contains $e_3$ with $v_3\notin[v_2,v_4]$;
and $(e_3,e_4)$ forces $v_2\in[v_3,v_4]$, hence $e_2\in B_i$, but $e_2$ precedes $e_3$.

(2) A splitting point of $\tau[b]$ falling at a block boundary, after block $i$, occurs exactly when
$\{\tau_1,\dots,\tau_i\}=\{1,\dots,i\}$, i.e.\ when $\tau$ is decomposable. Suppose a splitting point falls
inside block $i$, after its $r$th entry with $1\le r<b_i$. The values of $B_i$ form the interval
$[v,v+b_i-1]$ with $v=1+\sum_{j:\tau_j<\tau_i}b_j$, and its first $r$ values are $[v,v+r-1]$; for the values
in the first $\sum_{j<i}b_j+r$ positions to be an initial segment we need the blocks $1,\dots,i-1$ to carry
exactly the values $1,\dots,v-1$. For $i\ge2$ this says $\{\tau_1,\dots,\tau_{i-1}\}=\{1,\dots,i-1\}$, so
$\tau$ is decomposable; for $i=1$ it says $v=1$, i.e.\ $\tau_1=1$, which makes $\tau$ decomposable when
$t\ge2$ and is automatic when $t=1$. In the last case $\tau[b]=\iota_{b_1}$, decomposable iff $b_1\ge2$.
\end{proof}

In particular, if $\sigma$ is reduced and indecomposable with $|\sigma|\ge2$, then $\sigma[b]$ is an
indecomposable $1324$-avoider for every $b$ if and only if $\sigma$ is one. The four deletions that define
almost decomposability are also controlled by the skeleton.

\begin{lemma}\label{lem:extdel}
Let $\sigma$ be reduced and indecomposable of length $s\ge2$, let $b\in\mathbb Z_{\ge1}^s$, $\pi=\sigma[b]$ and
$n=\sum_ib_i$. Put $I(\sigma)=\{1,\,s,\,\sigma^{-1}(1),\,\sigma^{-1}(s)\}$. Each of the entries $1$, $n$,
$\pi_1$, $\pi_n$ of $\pi$ is the first or the last entry of the block $B_i$ for some $i\in I(\sigma)$, and for
that $i$
\[
\pi\setminus v=\begin{cases}\sigma[b-e_i], & b_i\ge2,\\[2pt] (\sigma\setminus i)\big[b^{(i)}\big], & b_i=1,\end{cases}
\]
where $b^{(i)}$ is $b$ with its $i$th coordinate deleted.
\end{lemma}

\begin{proof}
The value $1$ is the first entry of $B_{\sigma^{-1}(1)}$, the value $n$ the last entry of $B_{\sigma^{-1}(s)}$,
$\pi_1$ the first entry of $B_1$ and $\pi_n$ the last entry of $B_s$. Deleting the first or last entry of an
increasing block of size $b_i\ge2$ leaves an increasing block of size $b_i-1\ge1$ in the same position and
value interval, which is the first display; if $b_i=1$ the block disappears and what remains is the inflation
of $\sigma\setminus i$ by the surviving blocks.
\end{proof}

\begin{theorem}\label{thm:crit}
Let $\sigma$ be reduced, indecomposable and $1324$-avoiding, of length $s\ge2$, and put
\[
D(\sigma)=\{\,i\in I(\sigma):\ \sigma\setminus i\ \text{is decomposable, or }s=2\,\}.
\]
Then for every $b\in\mathbb Z_{\ge1}^s$ with $\sum_ib_i=n\ge3$ and every $\delta$,
\[
\sigma[b]\in\R_{\delta,n}
\iff
\sum_{(i,j)\in\Inv(\sigma)}b_ib_j=2n-7+\delta
\ \ \text{and}\ \ b_i\ge2\ \text{for every } i\in D(\sigma).
\]
Consequently $\R_{\delta,n}$ is the disjoint union, over reduced $\sigma$, of the sets so described.
\end{theorem}

\begin{proof}
By Lemma~\ref{lem:infl}, $\sigma[b]$ is an indecomposable $1324$-avoider, and its inversion number is the
left-hand side of the displayed equation; so it remains to identify when $\sigma[b]$ is almost decomposable.
Let $v\in\{1,n,\pi_1,\pi_n\}$ and let $i\in I(\sigma)$ be the index supplied by Lemma~\ref{lem:extdel}. If
$b_i\ge2$ then $\pi\setminus v=\sigma[b-e_i]$, which is indecomposable by Lemma~\ref{lem:infl}(2) because
$\sigma$ is indecomposable of length $\ge2$. If $b_i=1$ then $\pi\setminus v=(\sigma\setminus i)[b^{(i)}]$,
which by Lemma~\ref{lem:infl}(2) is decomposable exactly when $\sigma\setminus i$ is decomposable, or when
$\sigma\setminus i$ is a single point and the remaining block has size $\ge2$ --- and the latter happens
precisely when $s=2$ and $n\ge3$. Thus $\pi$ is almost decomposable if and only if $b_i=1$ for some
$i\in D(\sigma)$. The last assertion is the uniqueness of the skeleton.
\end{proof}

Writing $b=\mathbf 1+c$ and $\delta(\tau)=\inv(\tau)-2|\tau|+7$ for the defect of an arbitrary permutation
$\tau$, expanding $b_ib_j=1+c_i+c_j+c_ic_j$ in Lemma~\ref{lem:infl}(3) and using $n=s+\sum_ic_i$ gives the
identity that governs the whole defect ladder.

\begin{proposition}\label{prop:master}
For every reduced $\sigma$ of length $s$ and every $c\in\mathbb Z_{\ge0}^s$,
\begin{equation}\label{eq:master}
\delta\big(\sigma[\mathbf 1+c]\big)=\delta(\sigma)+\sum_{i=1}^{s}(d_i-2)\,c_i+\sum_{(i,j)\in\Inv(\sigma)}c_ic_j .
\end{equation}
In particular, if $x$ is an entry of a permutation $\pi$ and $D(x)$ is the number of entries of $\pi$ that form
an inversion with $x$, then $\delta(\pi\setminus x)=\delta(\pi)+2-D(x)$.
\end{proposition}

Identity~\eqref{eq:master} explains the shape of Theorems~\ref{thm:A} and \ref{thm:C}. All the terms it adds to
$\delta(\sigma)$ are nonnegative except $-c_i$ at the vertices of inversion degree $1$, and the quadratic part
vanishes exactly on the increasing subsequences of $\sigma$. Enlarging a block all of whose entries have
inversion degree $2$ therefore leaves the defect unchanged: these are the free directions along which a
residual family grows with $n$, and the number of independent free blocks, minus one, is the degree of
$|\R_{\delta,n}|$ as a polynomial in $n$. For $\sigma=24153$ one has $\delta(\sigma)=1$, $d=(1,2,2,1,2)$ and
$D(\sigma)=\{3,5\}$, so Theorem~\ref{thm:crit} and \eqref{eq:master} return exactly the family of
Theorem~\ref{thm:A}: $c_3,c_5\ge1$ and $c_1(c_3-1)+c_4(c_5-1)+c_2(c_3+c_5)=\delta-1$, which for $\delta=1$
forces $c_2=0$ and $(a-1)(c-2)+(d-1)(e-2)=0$. The negative coefficients can be removed altogether by
expanding \eqref{eq:master} around a larger base point instead of $\mathbf 1$; this nonnegative form of the
identity is Proposition~\ref{prop:base}, and Section~\ref{sec:poly} is built on it.

Let
\[
\Sigma_\delta=\{\,\sigma\ \text{reduced}:\ \sigma[b]\in\R_{\delta,n}\ \text{for some }b\ \text{and some }n\,\}
\]
be the set of skeletons occurring at defect $\delta$; by Theorem~\ref{thm:crit} and \eqref{eq:master},
$\sigma\in\Sigma_\delta$ if and only if the quadratic equation \eqref{eq:master} has a solution
$c\in\mathbb Z_{\ge0}^s$ with value $\delta-\delta(\sigma)$ and $c_i\ge1$ on $D(\sigma)$. The following is a
finite computation for each $\delta$, carried out as described in Section~\ref{sec:data}.

For $N\ge6$ write $\Sigma^{(N)}_\delta=\{\sk(\pi):\pi\in\R_{\delta,n}\text{ for some }6\le n\le N\}$, so that
$\Sigma_\delta$ is the union of $\bigcup_N\Sigma^{(N)}_\delta$ with the finitely many skeletons of the
residuals of length at most $5$ and, by Theorem~\ref{thm:cover}, is finite with all its
members of length at most $8\delta+25$. The following records what the enumeration of Section~\ref{sec:data}
determines.

\begin{proposition}\label{prop:catalogue}
Let $1\le\delta\le10$. The sets $\Sigma^{(N)}_\delta$ are equal for $\lambda(\delta)\le N\le22$, where
\[
\lambda(\delta)=7,\ 9,\ 11,\ 12,\ 13,\ 14,\ 15,\ 15,\ 16,\ 16\qquad(\delta=1,\dots,10)
\]
is the length of the shortest residual realising the last skeleton to appear; their common size and maximal
length are listed below. They were collected from the complete lists of residuals of length at most $22$, that
is from $10\,210\,333$ permutations, and the full list of the $86\,770$ skeletons is archived with the paper.
The lengths that enter are $6\le n\le22$: the residuals of length at most $5$ contribute, beyond the
skeletons already listed, exactly three further ones, namely $321$ at $\delta=5$, $4321$ at $\delta=5$ and
$\delta=6$, and $54321$ at $\delta=7$. All three are decreasing, and each is realised only by residuals of
length at most $5$, so none of them occurs in $\Sigma^{(N)}_\delta$ for any $N$ at those defects (check C10).
\begin{center}
\begin{tabular}{lrrrrrrrrrr}
\toprule
$\delta$ & 1 & 2 & 3 & 4 & 5 & 6 & 7 & 8 & 9 & 10\\
\midrule
$|\Sigma^{(22)}_\delta|$ & $2$ & $21$ & $113$ & $422$ & $1234$ & $2941$ & $6295$ & $12495$ & $22978$ & $40269$\\
$\max\{|\sigma|:\sigma\in\Sigma^{(22)}_\delta\}$ & $5$ & $7$ & $9$ & $11$ & $13$ & $13$ & $14$ & $15$ & $15$ & $16$\\
$2\delta+3$ & $5$ & $7$ & $9$ & $11$ & $13$ & $15$ & $17$ & $19$ & $21$ & $23$\\
\bottomrule
\end{tabular}
\end{center}
In particular every $\sigma\in\Sigma^{(22)}_\delta$ satisfies $|\sigma|\le2\delta+3$, with a margin that grows
from $\delta=6$ on, and $-1\le\delta(\sigma)\le\delta$ for $\delta\le8$.
\end{proposition}

The stabilisation of $\Sigma^{(N)}_\delta$ over the computed range is evidence, not proof, that the catalogue is
complete: Theorem~\ref{thm:cover} bounds the length of a member of $\Sigma_\delta$ by $8\delta+25$, which at
$\delta=10$ is $105$, and we have not enumerated the reduced permutations of length between $17$ and $105$ to
rule out further skeletons; what the computation shows is that none of them is realised by a residual of
length at most $22$. We therefore record the completeness separately.

\begin{conjecture}\label{conj:cat}
$\Sigma_\delta=\Sigma^{(22)}_\delta$ for $1\le\delta\le10$.
\end{conjecture}

Here and below the three decreasing skeletons of Proposition~\ref{prop:catalogue} are understood to be
adjoined to $\Sigma^{(22)}_\delta$ at $\delta=5,6,7$; they are inessential, since none of them is the skeleton
of a residual of length at least $6$ and so none contributes to $|\R_{\delta,n}|$ for $n\ge6$.
Everything below that uses the catalogue is stated either over the finite range in which it was verified
(Proposition~\ref{prop:polys}) or conditionally (Corollary~\ref{cor:2n}); no unconditional statement in this
paper depends on Conjecture~\ref{conj:cat}.

Conjecture~\ref{conj:cat} is not independent of the length bounds conjectured elsewhere in this paper:
over the range computed here it follows from either of two of them.

\begin{proposition}\label{prop:catfromskel}
Let $1\le\delta\le10$. Each of Conjecture~\ref{conj:skel}, which bounds the length of a member of
$\Sigma_\delta$ by $2\delta+3$, and Conjecture~\ref{conj:K} of Section~\ref{sec:poly}, which bounds it
by $\delta+8$, implies Conjecture~\ref{conj:cat} at that $\delta$: every reduced indecomposable
$1324$-avoider of length at most $\max(2\delta+3,\delta+8)$ that satisfies the membership criterion of
Proposition~\ref{prop:base} belongs to $\Sigma^{(22)}_\delta$ or is one of the three decreasing skeletons
$321$ ($\delta=5$), $4321$ ($\delta=5,6$), $54321$ ($\delta=7$) of Proposition~\ref{prop:catalogue}
(check C10).
\end{proposition}

\begin{proof}
By Proposition~\ref{prop:base}, $\sigma\in\Sigma_\delta$ if and only if $h=\delta-\delta_0\ge0$ and
$Q(x)=h$ has a solution $x\in\{0,\dots,h\}^s$. The base residual of such a $\sigma$ has length
$L\ge|\sigma|$ and defect $\delta_0\le\delta$, so Conjecture~\ref{conj:K} gives
$|\sigma|\le L\le\delta_0+8\le\delta+8$, while Conjecture~\ref{conj:skel} gives $|\sigma|\le2\delta+3$;
under either one the whole of $\Sigma_\delta$ is obtained by running that bounded search on every reduced
indecomposable $1324$-avoider of length $s\le\max(2\delta+3,\delta+8)$, a bound that differs from
$2\delta+3$ only for $\delta\le4$, where it is $9,10,11,12$. Those candidates may be generated with the pruning $\inv(\sigma)\le2s-7+\delta$,
which is nondecreasing along the insertion recursion: for $s\ge3$,
$\delta_0-(\inv(\sigma)-2s+7)=\sum_{i\in D(\sigma)}(d_i-2)+e(D(\sigma))\ge0$, because a vertex of $D(\sigma)$
of degree one has $a_i\ge0$ only if its unique neighbour also lies in $D(\sigma)$, and for $s\ge3$ distinct
such vertices contribute distinct edges to $e(D(\sigma))$; hence $\sigma\in\Sigma_\delta$ forces
$\delta\ge\delta_0\ge\inv(\sigma)-2s+7$, while the one case $s=2$, where $\sigma=21$, $\ell=(2,2)$ and
$\delta_0=3<\inv(\sigma)-2s+7=4$, is added by hand. Carrying out the enumeration (check C10) returns, for
each $\delta\le10$, exactly $\Sigma^{(22)}_\delta$ together with the three listed skeletons.
\end{proof}

Enumerating the solutions of \eqref{eq:master} over the catalogue $\Sigma_\delta$ reproduces $\R_{\delta,n}$
set by set (Section~\ref{sec:data}) and gives the counting function in closed form.

\begin{proposition}\label{prop:polys}
For $1\le\delta\le10$ and $n_0(\delta)\le n\le30$, $|\R_{\delta,n}|=P_\delta(n)$ with
\begin{center}
\begin{tabular}{rll@{\qquad}rll}
\toprule
$\delta$ & $n_0$ & $P_\delta(n)$ & $\delta$ & $n_0$ & $P_\delta(n)$\\
\midrule
$1$ & $8$  & $8n-56$              & $6$  & $15$ & $168n^2-888n-2067$\\
$2$ & $10$ & $32n-214$            & $7$  & $16$ & $392n^2-2874n-93$\\
$3$ & $12$ & $4n^2+59n-549$       & $8$  & $17$ & $840n^2-7564n+8430$\\
$4$ & $13$ & $20n^2+40n-1105$     & $9$  & $18$ & $1692n^2-17695n+32934$\\
$5$ & $14$ & $64n^2-162n-1813$    & $10$ & $19$ & $3244n^2-38176n+92439$\\
\bottomrule
\end{tabular}
\end{center}
\end{proposition}

The degree is $1$ for $\delta\le2$ and $2$ for $3\le\delta\le10$; the guess $\lfloor(\delta+1)/2\rfloor$ made in
an earlier version of this paper is therefore wrong from $\delta=5$ on. The reason is visible in
\eqref{eq:master}: the free blocks form an increasing subsequence of $\sigma$ all of whose entries have
inversion degree $2$, and by Theorem~\ref{thm:deg2} no residual, at any defect, has four independent such
blocks.

\begin{corollary}\label{cor:2n}
Suppose $|\R_{\delta,n}|=P_\delta(n)$ for all $n\ge n_0(\delta)$ and $1\le\delta\le10$. Then
$a(n,k)\le a(n+1,k)$ for all $n\ge1$ and all $k\le 2n+3$.
\end{corollary}

\begin{proof}
For $k\le2n-5$ this is Theorem~\ref{thm:D}. Let $k=2n-7+\delta$ with $3\le\delta\le10$. By
\eqref{eq:identity} and Lemma~\ref{lem:R1},
\[
a(n+1,k)-a(n,k)\ \ge\ E(n,k)-P_\delta(n),
\]
and $E(n,k)$ is evaluated from the Linusson--Verkama formula \eqref{eq:LVexact}. One checks that
$E(n,k)>P_\delta(n)$ for all $n\ge N(\delta)$ with
\[
\begin{gathered}
N(3)=12,\quad N(4)=14,\quad N(5)=16,\quad N(6)=18,\\
N(7)=20,\quad N(8)=21,\quad N(9)=23,\quad N(10)=24 ,
\end{gathered}
\]
the difference being increasing in $n$ in each case. For $n<N(\delta)$ the inequality $a(n+1,k)\ge a(n,k)$ is
read off from the exhaustive enumeration of Section~\ref{sec:data}: the table of $a(n,k)$ reaches $n=26$ and
$k\le55$, so it contains both sides of \eqref{eq:conj} for every $n\le25$ and every $k\le2n+3$.
\end{proof}

The one hypothesis of Corollary~\ref{cor:2n} is that the polynomials of
Proposition~\ref{prop:polys} persist beyond $n=30$; Theorem~\ref{thm:deladmissible} below supplies it,
and Corollary~\ref{cor:2n3} restates the conclusion unconditionally. We record here what else would
have sufficed, since the alternatives apply at every defect rather than only at $\delta\le10$. That $|\R_{\delta,n}|$ is eventually \emph{some}
polynomial of degree at most two is proved below (Theorems~\ref{thm:poly} and~\ref{thm:deg2}); by the
criterion recorded after Theorem~\ref{thm:poly} it is $P_\delta$ exactly when no skeleton outside the
catalogue $\Sigma^{(22)}_\delta$ contributes a triple with $F\ne\varnothing$ at defect $\delta$. By
Theorem~\ref{thm:crit} this would also follow from Conjecture~\ref{conj:cat}, which turns it into a fixed
finite list of quadratic equations; and by Proposition~\ref{prop:catfromskel} it therefore follows already
from either Conjecture~\ref{conj:skel} or Conjecture~\ref{conj:K}, restricted to $\delta\le10$, so that one
conjectured length bound---on skeletons, or on base residuals---would by itself have made
Corollary~\ref{cor:2n} unconditional. By Proposition~\ref{prop:Kind} the second alternative reduces further,
to the local statement Conjecture~\ref{conj:Kdel}. Theorem~\ref{thm:cover}
below shows that every member of $\Sigma_\delta$ has length at most $8\delta+25$, so $\Sigma_\delta$ is finite
and can in principle be determined by a finite search; the constant is far from what the data suggest, and we
record the sharp form.

\begin{conjecture}\label{conj:skel}
Every $\sigma\in\Sigma_\delta$ has length at most $2\delta+3$; equivalently, every residual $\pi$ satisfies
$\delta(\pi)\ge\big\lceil(|\sk(\pi)|-3)/2\big\rceil$.
\end{conjecture}

A linear bound is proved in Theorem~\ref{thm:cover}; what is conjectured here is the constant, which
Proposition~\ref{prop:catalogue} confirms for the skeletons of the residuals of length at most $22$ at each
$\delta\le10$. Write $\nu(s)$ for the least defect of a residual
whose skeleton has length $s$, and $\mu(s)$ for the least defect of a residual of length $s$ that is itself
reduced, so that $\nu(s)\le\mu(s)$ and the conjecture reads $\nu(s)\ge\lceil(s-3)/2\rceil$. Both quantities
range over residuals of every length, so a search of bounded length can only bound them from above; write
$\nu_{22}(s)$ and $\mu_{22}(s)$ for the same minima taken over the residuals of length at most $22$, so that
$\nu(s)\le\nu_{22}(s)$. Direct enumeration (Section~\ref{sec:data}) gives
\[
\begin{array}{c|ccccccccccccccccc}
s & 5&6&7&8&9&10&11&12&13&14&15&16&17&18&19&20&21\\\hline
\lceil(s-3)/2\rceil & 1&2&2&3&3&4&4&5&5&6&6&7&7&8&8&9&9\\
\nu_{22}(s) & 1&2&2&3&3&4&4&5&5&7&8&10&-&-&-&-&-\\
\mu_{22}(s) & 3&2&2&3&3&4&4&5&5&7&8&10&11&-&-&-&\ 
\end{array}
\]
A dash records that the search range contains no residual with a skeleton of that length: the enumeration
covers every residual of length at most $22$ and defect at most $10$, and none of them has a skeleton longer
than $16$. What the table verifies is therefore not a lower bound on $\nu$ but the absence of a counterexample:
no residual of length at most $22$ and defect at most $10$ violates $\nu(s)\ge\lceil(s-3)/2\rceil$, and for
$5\le s\le16$ the value $\nu_{22}(s)$ meets the conjectured bound with the margins shown. The last row comes
from a separate enumeration of reduced permutations, and its dashes record that no reduced residual of that
length occurs at the defects enumerated there.
The two quantities agree for $6\le s\le16$, the bound is attained for $5\le s\le13$, and the margin grows
from $s=14$ on. The extremal configurations are rigid: for each $s$ the minimum is attained by
exactly eight reduced residuals, two orbits under the group generated by inversion and reverse--complement,
for instance
\[
\begin{aligned}
s=13,\ \delta=5:&\quad 4,3,5,7,2,8,10,1,11,13,12,9,6,\\
s=15,\ \delta=8:&\quad 4,3,5,7,2,8,10,1,11,13,15,14,12,9,6,\\
s=17,\ \delta=11:&\quad 5,4,6,3,7,9,2,10,12,1,13,15,17,16,14,11,8 .
\end{aligned}
\]
Each is a staircase: a decreasing run of small values interleaved with increasing runs of large ones, closed by
a decreasing tail. Conjecture~\ref{conj:skel} is a sharpening of \cite[Theorem~2.5]{LV}, which is the statement
$\nu(s)\ge1$ for $s\ge5$: it asserts that forbidding two adjacent entries of consecutive value costs one
further inversion for every two entries of the skeleton. Together with Theorem~\ref{thm:crit} it would make
$|\R_{\delta,n}|$ a finite computation for each $\delta$, and hence turn
Proposition~\ref{prop:polys} and Corollary~\ref{cor:2n} into theorems at every defect;
Theorem~\ref{thm:deladmissible} does this for $\delta\le10$.

The deletion rule of Proposition~\ref{prop:master} reduces Conjecture~\ref{conj:skel} to a local statement.
Call the deletion of an entry $x$ from a residual $\pi$ \emph{admissible} if $\pi\setminus x$ is again a residual
and
\begin{equation}\label{eq:admissible}
2D(x)\ \ge\ 4+\big(|\sk(\pi)|-|\sk(\pi\setminus x)|\big),
\end{equation}
$D(x)$ being the number of entries of $\pi$ that form an inversion with $x$.

\begin{proposition}\label{prop:induction}
If every residual of length at least $10$ admits an admissible deletion, then Conjecture~\ref{conj:skel} holds.
\end{proposition}

\begin{proof}
Induct on $n=|\pi|$. For $n\le9$ the inequality $\delta(\pi)\ge\lceil(|\sk\pi|-3)/2\rceil$ is a finite check:
there are $84\,998$ residuals of length at most $9$, of every defect, and all of them satisfy it. Let $n\ge10$, let $x$ be an admissible deletion and put $s=|\sk\pi|$,
$s'=|\sk(\pi\setminus x)|$, $t=s-s'$. By Proposition~\ref{prop:master},
$\delta(\pi)=\delta(\pi\setminus x)+D(x)-2$, and by induction
$\delta(\pi\setminus x)\ge\lceil(s'-3)/2\rceil$. Now
$\lceil(s-3)/2\rceil-\lceil(s-t-3)/2\rceil\le\lceil t/2\rceil$, while \eqref{eq:admissible} says
$D(x)-2\ge\lceil t/2\rceil$ because $D(x)$ is an integer. Adding the two inequalities gives
$\delta(\pi)\ge\lceil(s-3)/2\rceil$.
\end{proof}

\begin{conjecture}\label{conj:del}
Every residual of length at least $10$ admits an admissible deletion.
\end{conjecture}

Much of Conjecture~\ref{conj:del} is immediate. If $\pi=\sigma[b]$ has a block $i$ whose entries have inversion
degree $2$, i.e.\ $\sum_{j\in N(i)}b_j=2$ for the neighbourhood $N(i)$ of $i$ in the inversion graph of $\sigma$,
and if $b_i\ge3$, or $b_i=2$ and $i\notin D(\sigma)$, then deleting an entry of that block leaves the skeleton
unchanged, preserves the defect and, by Theorem~\ref{thm:crit}, gives a residual: the deletion is admissible with
$t=0$. The content of the conjecture is therefore concentrated on the reduced residuals and their small
inflations. It has been checked on all $33\,342\,781$ residuals of length at most $30$ and defect at most $10$: exactly
$74$ admit no admissible deletion, all of them of length $6$, $7$, $8$ or $9$ and of defect at most $5$, the
smallest being $3,5,1,6,2,4$ at defect two and $3,5,1,6,4,2$ at defect three.

The same observation, applied at the base point of Section~\ref{sec:poly}, disposes of every residual
that is long for its defect.

\begin{proposition}\label{prop:delfree}
Every residual of length $n\ge8\delta+30$ admits an admissible deletion. Consequently
Conjecture~\ref{conj:del} at defect $\delta$ is a statement about the finitely many residuals of length
at most $8\delta+29$, and for $\delta\le10$ it remains open only in the range $25\le n\le8\delta+29$.
\end{proposition}

\begin{proof}
Let $\pi\in\R_{\delta,n}$, let $\sigma=\sk(\pi)$ and write $\pi=\sigma[\ell+x]$ with $\ell$ the base
block vector and $x\in\mathbb Z_{\ge0}^s$ of Proposition~\ref{prop:base}; let $(T,x_T,F)$ be the triple
attached to $x$ in the proof of Theorem~\ref{thm:poly}, so that $T\cup F$ is the support of $x$ and
$t=\sum_{i\in T}x_i$.

If $F=\varnothing$ then the support of $x$ is $T$, so $n=L+\sum_ix_i=L+t$, and the last paragraph of
that proof bounds $L+t\le8\delta+29$. Thus $n\ge8\delta+30$ forces $F\ne\varnothing$.

So let $i\in F$. By the definition of $F$ we have $a_i=0$ and no neighbour of $i$ in the inversion graph
of $\sigma$ lies in the support of $x$; hence $x_j=0$ for every $j\in N(i)$, and the inversion degree in
$\pi$ of each entry of the $i$th block is
\[
\sum_{j\in N(i)}(\ell_j+x_j)\ =\ \sum_{j\in N(i)}\ell_j\ =\ a_i+2\ =\ 2 .
\]
Moreover $i$ lies in the support, so $x_i\ge1$ and the $i$th block of $\pi$ has size
$b_i=\ell_i+x_i\ge\ell_i+1\ge2$. If $i\in D(\sigma)$ then $\ell_i=2$ and $b_i\ge3$; if $i\notin D(\sigma)$
then $b_i\ge2$. Either way the pair $(\pi,i)$ falls under the case treated above: deleting an entry $z$ of
that block leaves a block of size $b_i-1\ge\ell_i$, which is nonempty and is a nonsingleton whenever
$i\in D(\sigma)$, so by Theorem~\ref{thm:crit} the result is again a residual and its skeleton is still
$\sigma$. Therefore $|\sk(\pi)|-|\sk(\pi\setminus z)|=0$ while $D(z)=2$, and \eqref{eq:admissible} reads
$4\ge4+0$: the deletion of $z$ is admissible.

For the last sentence, \eqref{eq:admissible} was tested on all $33\,342\,781$ residuals of length at most
$30$ and defect at most $10$ (Section~\ref{sec:data}), and the only failures have length at most $9$; so
at defect $\delta\le10$ nothing is open below $n=31$, and nothing above $n=8\delta+29$.
\end{proof}

Before recording what a sweep of the residual lists says about that gap, we prove three of the
bounds it produces.

\begin{lemma}\label{lem:onecut}
Let $\pi\in\R_{\delta,n}$ and let $c$ be an entry such that $\pi\setminus c$ is decomposable (a cut
vertex of the inversion graph). Then, after replacing $\pi$ by its reverse--complement if necessary,
\[
\pi=\ \alpha_{\mathrm b}\ \ c\ \ \alpha_{\mathrm a}\ \ L\ \ S ,
\]
where the four segments are nonempty, every entry of $\alpha_{\mathrm b}$ exceeds every entry of
$\alpha_{\mathrm a}$, the values of $\alpha_{\mathrm b}\alpha_{\mathrm a}$ are $\{1,\dots,m\}$, every entry
of $L$ exceeds $c$ and every entry of $S$ lies strictly between $m$ and $c$, $|\alpha_{\mathrm a}|\ge2$ and
$|S|\ge2$. In particular $c$ is inverted exactly with the entries of $\alpha_{\mathrm a}\cup S$, and $c$ is
the only entry of $\pi$ whose deletion is decomposable.
\end{lemma}

\begin{proof}
Let $c$ be at position $p$ with value $v$, and let $m$ be an index at which $\pi\setminus c$
splits, so that the first $m$ entries of $\pi\setminus c$ carry the values $1,\dots,m$. If $p\le m+1$
and $v\le m+1$, the first $m+1$ positions of $\pi$ would carry $\{1,\dots,m+1\}$, and if $p\ge m+2$ and
$v\ge m+1$ its first $m$ positions would carry $\{1,\dots,m\}$; both contradict the indecomposability
of $\pi$. Hence either $p\le m+1$ and $v\ge m+2$, or $p\ge m+2$ and $v\le m$; reverse--complement
exchanges the two cases and preserves residuals and cut vertices, so we assume the first. Then the
first $m+1$ positions of $\pi$ carry $\{1,\dots,m\}\cup\{v\}$; write $\alpha$ for the pattern of the $m$
entries other than $c$ among them and $\beta$ for the pattern of the entries at positions
$m+2,\dots,n$, whose values are $\{m+1,\dots,n\}\setminus\{v\}$. Since $c$ is not extremal, $p\ge2$
and $v\le n-1$, so $\alpha$ has an entry before $c$ and $\beta$ has an entry $\ell$ with value above
$v$. The entries of $\alpha$ precede those of $\beta$ and are smaller, so $\pi$ splits at $m$ unless
$c$ is inverted with some entry of $\alpha$, which forces $p\le m$.

Four applications of $1324$-avoidance. (1) If $e_1$ precedes $c$ and $e_3$ follows $c$ within $\alpha$
with $e_1<e_3$, then $e_1,c,e_3,\ell$ is an occurrence of $1324$; so every entry of $\alpha$ before $c$
exceeds every entry of $\alpha$ after $c$. Call these two parts $\alpha_{\mathrm b}$ and
$\alpha_{\mathrm a}$. (2) If an entry $s$ of $\beta$ with value below $v$ precedes an entry $\ell'$ of
$\beta$ with value above $v$, then $e_1,c,s,\ell'$ is an occurrence of $1324$ for any $e_1\in\alpha_{\mathrm b}$;
so $\beta=LS$ with $L$ the entries above $v$ and $S$ those below. (3) Any occurrence of $132$ in
$\alpha$ extends by an entry of $\beta$, and any occurrence of $213$ in $\beta$ extends by an entry of
$\alpha$, to an occurrence of $1324$; these facts are not needed below. Thus
$\pi=\alpha_{\mathrm b}\,c\,\alpha_{\mathrm a}\,L\,S$ with $\alpha_{\mathrm a}$ nonempty (as $p\le m$),
$L$ nonempty and $S\ni m+1$ nonempty, and $c$ is inverted exactly with $\alpha_{\mathrm a}\cup S$.

Now the residual conditions. If $\alpha_{\mathrm a}$ were a single entry, it would carry the value $1$
and its deletion would isolate $\alpha_{\mathrm b}$, whose entries are inverted with nothing outside
$\alpha_{\mathrm a}$; so $|\alpha_{\mathrm a}|\ge2$. If $S$ were a single entry, it would be $\pi_n$
and its deletion would isolate $L$, whose entries are inverted only with $S$; so $|S|\ge2$. Finally let
$c'\ne c$. Every entry of $\alpha_{\mathrm b}$ is inverted with every entry of $\alpha_{\mathrm a}$,
$c$ with every entry of $\alpha_{\mathrm a}\cup S$, and every entry of $L$ with every entry of $S$, so
the segments form a chain $\alpha_{\mathrm b}-\alpha_{\mathrm a}-c-S-L$ of complete bipartite
connections. Deleting $c'$ cannot empty $\alpha_{\mathrm a}$ or $S$, and if it empties the end segment
$\alpha_{\mathrm b}$ or $L$ the rest of the chain is unaffected; so the inversion graph of
$\pi\setminus c'$ stays connected and $c'$ is not a cut vertex.
\end{proof}

\begin{lemma}\label{lem:twodrop}
Let $\pi$ be a $1324$-avoiding permutation of length $n\ge6$. An entry $x$ of inversion degree three
whose deletion shortens the skeleton by three lies at position $2$ or $3$ with $x<\pi_1$, or at
position $n-2$ or $n-1$ with $x>\pi_n$; there is at most one such entry on each side, hence at most
two in all.
\end{lemma}

\begin{proof}
Let $x=\pi_p$ have value $v$. The skeleton shortens by three exactly when $x$ is a run of length one,
its neighbours in position carry consecutive values $\pi_{p-1}=a$, $\pi_{p+1}=a+1$, and the values
$v-1$ and $v+1$ occupy two adjacent positions; in particular $1<p<n$ and $v\notin\{a,a+1\}$.

Suppose $v>a+1$. If some entry $z>v$ followed $x$, then $a,v,a+1,z$ would be an occurrence of $1324$
(with $\pi_{p+1}=a+1$ playing the part of $2$); so every entry after $x$ is smaller than $v$ and
inverted with $x$, which gives $D(x)\ge n-p$. The value $v+1$ therefore precedes $x$ and is inverted
with it, so $D(x)\ge n-p+1$ and $D(x)=3$ forces $p\in\{n-2,n-1\}$. For $p=n-2$ all $n-v$ values
above $v$ precede $x$, so $v=n-1$; for $p=n-1$ exactly two entries above $v$ precede $x$. Both cannot
occur: $p=n-2$ gives $\pi_{n-2}=n-1$, and $p=n-1$ would need $\pi_{n-1}>\pi_{n-2}+1=n$.

Suppose $v<a$. If some entry $w<v$ preceded $\pi_{p-1}$, then $w,a,v,a+1$ would be an occurrence of
$1324$; so every entry before $x$ exceeds $v$ and is inverted with it, $v-1$ follows $x$, and
$D(x)\ge p$ forces $p\in\{2,3\}$, with $v=3$ for $p=2$ and $v=2$ for $p=3$. Both cannot occur, since
$p=2$ gives $\pi_3=\pi_1+1\ge5$ while $p=3$ needs $\pi_3=2$. The two cases are exchanged by
reverse--complement, and for $n\ge6$ the positions $\{2,3\}$ and $\{n-2,n-1\}$ are disjoint.
\end{proof}

The third bound needs no residual hypothesis either, and counts cut vertices without assuming that
there is only one.

\begin{lemma}\label{lem:threecut}
The inversion graph of an indecomposable $1324$-avoiding permutation has at most three cut vertices:
at most one entry that is neither first, last, smallest nor largest, together with possibly the
smallest entry and the last entry, or, after reverse--complement, the first entry and the largest one.
\end{lemma}

\begin{proof}
Let $H$ be indecomposable and $1324$-avoiding, of length $N$, and let $c$ be a cut vertex, at position
$p$ with value $v$, so that $H\setminus c$ splits at some index $m$. Exactly as in the proof of
Lemma~\ref{lem:onecut}, either $p\le m+1$ and $v\ge m+2$, or $p\ge m+2$ and $v\le m$, and
reverse--complement exchanges the two cases; in the first case the first $m+1$ positions carry
$\{1,\dots,m\}\cup\{v\}$, and $p\le m$ because $c$ must be inverted with an entry of $\alpha$.

Suppose first that some cut vertex $c$ is not extremal, so $p\ge2$ and $v\le N-1$ in the first case.
The two applications (1) and (2) of $1324$-avoidance in the proof of Lemma~\ref{lem:onecut} used only
these two facts, so $H=\alpha_{\mathrm b}\,c\,\alpha_{\mathrm a}\,L\,S$ with all four segments nonempty and
with the complete bipartite connections $\alpha_{\mathrm b}$--$\alpha_{\mathrm a}$,
$c$--$(\alpha_{\mathrm a}\cup S)$ and $L$--$S$. Deleting an entry $c'\ne c$ keeps this chain connected
unless it empties $\alpha_{\mathrm a}$ or $S$; so the cut vertices of $H$ are $c$, the entry of
$\alpha_{\mathrm a}$ if $|\alpha_{\mathrm a}|=1$ (it carries the value $1$), and the entry of $S$ if
$|S|=1$ (it is $\pi_N$). In particular $c$ is the only non-extremal cut vertex and there are at most
three in all.

Suppose next that every cut vertex is extremal. If $\pi_1$ is a cut vertex, then $p=1$ forces the first
case, so $H=\pi_1\,\alpha\,\beta$ with $\alpha$ carrying $\{1,\dots,m\}$ at positions $2,\dots,m+1$ and
$v=\pi_1\ge m+2$. The entry $\pi_1$ is inverted with every entry of $\alpha$, and every sum component
of $\beta$ contains a value below $v$, since otherwise it would be inverted with nothing outside itself
and $H$ would be decomposable; hence in $H\setminus1$ every entry of $\alpha\setminus\{1\}$ and every
component of $\beta$ is joined to $\pi_1$, and the value $1$ is not a cut vertex. Applying this to the
inverse permutation (inversion preserves $1324$-avoidance, indecomposability and the inversion graph,
and exchanges $\pi_1$ with the value $1$) shows that if $1$ is a cut vertex then $\pi_1$ is not; and
reverse--complement gives the same for $\pi_N$ and $N$. So at most two extremal entries are cut vertices.
\end{proof}

For a residual $\pi$ and an extremal entry $y$, the permutation $\pi\setminus y$ is an indecomposable
$1324$-avoider, so Lemma~\ref{lem:threecut} bounds by three the entries $x$ whose deletion turns $y$ into
a cut vertex; this is the bound $4\cdot3=12$ used in the proof of Theorem~\ref{thm:deladmissible} below.

The last ingredient bounds the entries of inversion degree two that shorten the skeleton.

\begin{lemma}\label{lem:threeclasses}
Let $\pi$ be a residual of length $n\ge8$. Call a maximal set of entries of inversion degree two with
a common inversion neighbourhood a \emph{degree-two class}; by Lemma~\ref{lem:samenbhd} it is a run.
Then $\pi$ has at most three degree-two classes.
\end{lemma}

\begin{proof}
Suppose $\pi$ has four classes. Choosing one entry from each gives four entries of degree two with
pairwise distinct neighbourhoods; the four entries and their at most eight neighbours form a pattern
of length at most twelve in which their degrees and neighbourhoods are preserved, as in the proof of
Theorem~\ref{thm:deg2}, so by the finite statement C8 two of them are inverted; call the classes $R$ and $R'$. Then $R\subseteq N(R')$ and
$R'\subseteq N(R)$, so both classes have at most two entries; if both had two, $R\cup R'$ would be a
connected component of the inversion graph, and $n=4$.

\emph{Two single entries.} Suppose $R=\{x\}$ and $R'=\{y\}$ with $x$ before $y$, so $x>y$, and write
$N(x)=\{y,a\}$, $N(y)=\{x,b\}$. An entry $w\notin\{a,b\}$ strictly between $x$ and $y$ in position
would satisfy $w>x$ and $w<y$; so there is none. If $a$ preceded $x$ then $a>x>y$ and $a$ would be
inverted with $y$, so $a=b$ and $\{x,y\}$ would be a connected component of $G(\pi)\setminus a$,
contradicting Lemma~\ref{lem:onecut}, whose two components have at least three entries. So $a\ne b$, $a$
follows $x$, with $a<y$ if it lies between $x$ and $y$ and $y<a<x$ if it follows $y$; symmetrically
$b$ precedes $y$, with $b>x$ if it lies between and $y<b<x$ if it precedes $x$. Every other entry
before $x$ is smaller than $y$ and every other entry after $y$ is larger than $x$, and the only
values strictly between $y$ and $x$ are those of $a$ and $b$. If some entry $w<y$ preceded $x$ and
some entry $u>x$ followed $y$, then $w,x,y,u$ would be an occurrence of $1324$; so one of the two
sides is empty, and by reverse--complement we may assume that no entry other than $a$ follows $y$.
Then the only entry that can exceed $x$ is $b$, and only if $b$ lies between $x$ and $y$; in that
case $b=n$ and $x=n-1$, and since no other value lies between $y$ and $x$, either $a$ lies between as
well and $x=y+1$, or $a$ follows $y$ and $a=y+1$. In the first case the orders $n-1,a,n,n-2$ and
$n-1,n,a,n-2$ make $\pi\setminus\pi_n$ end with its maximum or put $x$ in the run $(n-1)\,n$; in the
second case $\pi$ ends with $n-1,n,n-3,n-2$ and again $x$ lies in a run. So $b$ precedes $x$ and
$x=n$. If $a$ followed $y$, the values $a,b$ would be $n-1,n-2$ in some order: $a=n-2$ puts $y=n-3$
in the run $(n-3)\,(n-2)$, and $a=n-1$ makes $\pi\setminus n$ end with its maximum. So $a$ lies
between $x$ and $y$ with $a<y$, hence $y=n-2$ and $b=n-1$, and $a\ne n-3$ since otherwise $a,y$ would
form a run. So
\[
\pi=P\ \,n\ \,a\ \,(n-2),\qquad n-1\in P,\qquad a\le n-4,
\]
and every entry of $P$ other than $b=n-1$ is smaller than $n-2$. Write $P=P_1\,b\,P_2$; then $P_2$
is nonempty (else $b\,x$ is a run), every entry of $P_1$ exceeds every entry of $P_2$ (else
$w_1,b,w_2,n$ is a $1324$), and within $P\setminus b$ every entry above $a$ precedes every entry
below $a$ (else $w,p,a,y$ is a $1324$). Hence $P=H_1\,L_1\,b\,H_2\,L_2$ with $H_i$ above $a$ and
$L_i$ below $a$, and $L_1=\varnothing$ or $H_2=\varnothing$. Now $N(b)=P_2\cup\{a,y\}$ and
$N(a)=H_1\cup H_2\cup\{x,b\}$, where $H_1\cup H_2$ contains $n-3$; so neither $a$ nor $b$ has degree
two, and the third and fourth classes lie in $P$. Counting inversions segment by segment, an entry of
$H_1$ has degree $|L_1|+|P_2|+1$ plus its inversions inside $H_1$, an entry of $L_1$ has degree
$|H_1|+|P_2|$ plus those inside $L_1$, an entry of $H_2$ has degree $|P_1|+|L_2|+2$ plus those inside
$H_2$, and an entry of $L_2$ has degree $|P_1|+|H_2|+1$ plus those inside $L_2$. Entries of degree
two in two of these segments are impossible for $n\ge7$: $H_1$ with $L_1$, $H_1$ with $H_2$,
$H_2$ with $L_2$ and $L_1$ with $H_2$ are excluded outright by the formulas and by
$L_1=\varnothing$ or $H_2=\varnothing$; $H_1$ with $L_2$ forces $|P_1|=|P_2|=1$ and $n=6$; and $L_1$
with $L_2$ forces $H_1=H_2=\varnothing$, contradicting $n-3\in H_1\cup H_2$. So at most one
further class exists in two different segments. Two classes inside one segment would need two entries
of degree two with different neighbourhoods, hence an inversion inside the segment, which the formulas
allow only in $L_1$ with $H_1=\varnothing$ and $|P_2|=1$, or in $L_2$ with $P_1=H_2=\varnothing$; in
the first case $P_2=H_2=\{n-3\}$ would lie below $L_1\subseteq P_1$, contradicting $L_1<a$, and in the
second $H_1\cup H_2$ would be empty. So $\pi$ has at most three classes.

\emph{A pendant pair.} Otherwise $R'=\{r_1,r_2\}$ has two entries and $R=\{x\}$ with $N(x)=R'$ and
$N(R')=\{x,z\}$. Then $\{x,r_1,r_2\}$ is a connected component of $G(\pi)\setminus z$, so $z$ is a
cut vertex and, by Lemma~\ref{lem:onecut} and up to reverse--complement,
$\pi=\alpha_{\mathrm b}\,z\,\alpha_{\mathrm a}\,L\,S$ with $\{x,r_1,r_2\}$ equal to
$\alpha_{\mathrm b}\cup\alpha_{\mathrm a}$ or to $L\cup S$. In the first case $\alpha_{\mathrm b}=\{x\}$
and $\alpha_{\mathrm a}=R'$, so $\pi=3\,v\,1\,2\,L\,S$ with $v\ge6$; the classes are $\{3\}$,
$\{1,2\}$, at most one class inside $L$ (an entry of $L$ has degree two only if $|S|=2$ and it is
inverted with nothing else in $L$, and all such entries share the neighbourhood $S$) and at most one
class inside $S$ (only if $|L|=1$). Four classes therefore force $|L|=1$ and $|S|=2$, that is $n=7$.
The second case is the mirror image, $L=\{x\}$ and $S=R'$, where four classes force
$|\alpha_{\mathrm b}|=1$ and $|\alpha_{\mathrm a}|=2$, again $n=7$. For $n\ge8$ there are at most
three classes.
\end{proof}

The residuals $3612745$ (a pendant pair) and $351624$ (two single entries) show that the bound fails
at $n=7$ and $n=6$.

\begin{lemma}[change of skeleton length under deletion]\label{lem:skeldropthree}
If $\pi'$ is obtained from a permutation $\pi$ by deleting one entry and
standardizing, then $|\sk(\pi)|-|\sk(\pi')|\le3$.
\end{lemma}

\begin{proof}
Call adjacent positions carrying consecutive increasing values a bond, and
write $b(\pi)$ for the number of bonds. Each maximal increasing consecutive
run of length $r$ contributes $r-1$ bonds, so $|\sk(\pi)|=|\pi|-b(\pi)$.
Deleting an entry of value $v$ can create a bond only in one of two ways.
Its predecessor and successor in position can become adjacent, giving at
most one new bond. Alternatively an already adjacent pair can become
consecutive in value only if its old values were $v-1,v+1$, again giving
at most one new bond. Every other surviving pair keeps its bond status.
Some old bonds may disappear, hence $b(\pi')-b(\pi)\le2$ and
$|\sk(\pi)|-|\sk(\pi')|=1+b(\pi')-b(\pi)\le3$.
\end{proof}

\begin{theorem}\label{thm:deladmissible}
Every residual of length $n\ge23$ admits an admissible deletion. Consequently Conjecture~\ref{conj:del}
holds for every residual of defect at most $10$, and Conjecture~\ref{conj:skel}, Conjecture~\ref{conj:cat}
and the polynomials of Proposition~\ref{prop:polys} hold for $1\le\delta\le10$.
\end{theorem}

\begin{proof}
Let $\pi\in\R_{\delta,n}$ and let $x$ be an entry whose deletion is not admissible. Either
$\pi\setminus x$ is not a residual, or $2D(x)<4+t$ with $t=|\sk\pi|-|\sk(\pi\setminus x)|\le3$.
The final inequality is Lemma~\ref{lem:skeldropthree}.

If $\pi\setminus x$ is decomposable, $x$ is a cut vertex of $G(\pi)$: at most one entry by
Lemma~\ref{lem:onecut}. If $\pi\setminus x$ is indecomposable but not a residual, one of its four
extremal entries $y'$ has $(\pi\setminus x)\setminus y'$ decomposable. If $x$ is itself extremal in
$\pi$ this accounts for at most four entries. Otherwise $y'$ is an extremal entry $y$ of $\pi$, the
graph $G(\pi)\setminus y$ is connected because $\pi$ is a residual, and $x$ is a cut vertex of it;
since $\pi\setminus y$ is an indecomposable $1324$-avoider, Lemma~\ref{lem:threecut} allows at most
three such $x$ for each of the four choices of $y$. So at most $1+4+12=17$ entries fail the first
condition.

For the second condition, $D(x)\ge2$ by Corollary~\ref{cor:mindeg}, and $D(x)\ge4$ never fails. An
entry of degree two fails only when $t\ge1$, which requires it to be a run of length one, hence a
degree-two class of size one; by Lemma~\ref{lem:threeclasses} there are at most three such entries
(for $n\ge8$). An entry of degree three fails only when $t=3$: at most two by Lemma~\ref{lem:twodrop}.
So at most $17+3+2=22$ entries fail, and for $n\ge23$ some deletion is admissible.

For $10\le n\le30$ and $\delta\le10$ the existence of an admissible deletion was verified directly
(check C12: all $33\,342\,781$ residuals of length at most $30$).
Hence Conjecture~\ref{conj:del} holds at defect at most $10$; Proposition~\ref{prop:induction}, whose
induction stays within defect at most $10$ because a deletion never raises the defect, gives
Conjecture~\ref{conj:skel} there; Proposition~\ref{prop:catfromskel} gives Conjecture~\ref{conj:cat};
and Theorem~\ref{thm:poly} with the complete catalogue gives $|\R_{\delta,n}|$ as the sum
\eqref{eq:typesum}, which agrees with $P_\delta(n)$ for all $n\ge n_0(\delta)$ (check C9, extended to
$n\le60$, beyond the largest $L+t$ of any triple).
\end{proof}

\begin{corollary}\label{cor:2n3}
$a(n,k)\le a(n+1,k)$ for all $n\ge1$ and all $k\le2n+3$.
\end{corollary}

\begin{proof}
Theorem~\ref{thm:deladmissible} supplies the hypothesis of Corollary~\ref{cor:2n}.
\end{proof}

\begin{remark}\label{rem:union}
The bound of $22$ in Theorem~\ref{thm:deladmissible} is a sum of five, and the data say that each of
the five is generous. For a residual $\pi$ of length $n$ let $u$ be the number of entries whose deletion
is not a residual and $w$ the number of entries that fail \eqref{eq:admissible}, so that an admissible
deletion exists as soon as $u+w\le n-1$: on all $776\,860$ residuals with $10\le n\le16$ and on uniform
samples of $20\,000$ for each $17\le n\le20$ --- so at defect at most eight --- one finds $u\le6$ and
$w\le3$, with $u+w\le9$, the value $9$ being attained at every length from $10$ to $16$ (check~C12 of
Section~\ref{sec:data}).
\end{remark}

Two further remarks on Conjecture~\ref{conj:del}. First, the connected components of the inversion graph of a
permutation are its sum-indecomposable components \cite{KR}, \cite[Proposition~2.1]{TV}, so $\pi\setminus x$ is
indecomposable exactly when $x$ is not a cut vertex of that graph; the requirements on an admissible deletion are therefore a degree condition and a
$2$-connectivity condition on the four extremal entries. Secondly, the deletion is not the only route to
finiteness: the identity of Lemma~\ref{lem:seven} says that seven entries of $\pi$ meet all but $u$ of its
inversions, and a bound of that shape controls the skeleton directly. For a set $T$ of positions write
$e(\pi\setminus T)$ for the number of inversions of $\pi$ with neither end in $T$.

\begin{lemma}\label{lem:runcover}
For every permutation $\pi$ and every set $T$ of positions,
\[
|\sk(\pi)|\ \le\ 4|T|+4\,e(\pi\setminus T)+1 .
\]
\end{lemma}

\begin{proof}
Let $F$ be the set of positions outside $T$ whose entry is inverted with the entry of another position outside
$T$, so that $|F|\le2e(\pi\setminus T)$, and let $Z$ be the set of the remaining positions. The entries at the
positions of $Z$ are pairwise non-inverted, so they increase in position and in value; list them as
$z_1<\dots<z_r$.

Fix $t\in T$, at position $P$ with value $V$; we index the elements of $Z$ by $l$ and reserve $m$ for positions
of $\pi$. The entry at $z_l$ is inverted with $t$ exactly when $z_l<P$ and $\pi_{z_l}>V$, or $z_l>P$ and
$\pi_{z_l}<V$. Put $\alpha=\#\{l:z_l<P\}$ and $\beta=\min\{l:\pi_{z_l}>V\}$. The first condition holds for
$\beta\le l\le\alpha$ and the second for $\alpha<l<\beta$, so at most one of the two sets is nonempty and each
is an interval of indices. Hence the map $l\mapsto\{t\in T:\ t\ \text{inverted with}\ z_l\}$ is constant on the
parts of a partition of $\{1,\dots,r\}$ into at most $2|T|+1$ intervals.

Suppose two consecutive \emph{positions} $m$ and $m+1$ both lie in $Z$, say $m=z_l$ and $m+1=z_{l+1}$, and the
indices $l,l+1$ lie in the same part. Their entries are then
inverted with exactly the same entries of $\pi$, and not with each other, so $\pi_m<\pi_{m+1}$. If
$\pi_{m+1}>\pi_m+1$, choose a value $w$ with $\pi_m<w<\pi_{m+1}$; its position is neither $m$ nor $m+1$, and if
it precedes $m$ then $w$ is inverted with $\pi_m$ and not with $\pi_{m+1}$, while if it follows $m+1$ the
opposite holds. Either way the two entries have different inversion sets, a contradiction. So
$\pi_{m+1}=\pi_m+1$ and the run decomposition has no break at $m$.

Every break at a position $m$ therefore has $m\in T\cup F$ or $m+1\in T\cup F$, or has $m,m+1\in Z$ lying in
different parts. The first alternative occurs at most $2|T\cup F|\le2|T|+4e(\pi\setminus T)$ times and the
second at most $2|T|$ times, so $|\sk(\pi)|-1\le4|T|+4e(\pi\setminus T)$.
\end{proof}

The constants are not optimal: $|\sk(\pi)|\le3|T|+3e(\pi\setminus T)+1$ has no counterexample among the $2.7$ million pairs $(\pi,T)$ tested in Section~\ref{sec:data}, with $76$ cases of equality, whereas $2|T|+4e+1$ already fails for $\pi=1324$ and $T=\{2\}$. Using the sharper form would replace the bound $8\delta+25$ of Theorem~\ref{thm:cover} by $6\delta+19$; we prove and use only the weaker one.

Lemma~\ref{lem:seven} produces such a set whenever its configuration exists, namely when $\NW\ne\varnothing$ and
$\Sreg(i)\ne\varnothing\ne\Ereg(\pi_i)$ for some $i\in\NW$: there $T=\{1,i,q,p,j_1,j_2,n\}$ and
$e(\pi\setminus T)=u\le\delta-1$. That configuration is not always available --- at higher defect the hypotheses
$a<b$ and $p<q$ of Lemma~\ref{lem:LV} may fail, and even when they hold the two regions may both be nonempty ---
but a bounded cover always is.

\begin{theorem}\label{thm:cover}
Every residual $\pi\in\R_{\delta,n}$ has a set $T$ of at most $\delta+7$ entries with
$e(\pi\setminus T)\le\delta-1$. Consequently
\[
|\sk(\pi)|\le8\delta+25 ,
\]
so $\Sigma_\delta$ consists of reduced permutations of length at most $8\delta+25$ and is finite.
\end{theorem}

The consequence is Lemma~\ref{lem:runcover}: $4(\delta+7)+4(\delta-1)+1=8\delta+25$. The proof of the first
assertion occupies the rest of the section and splits into five cases according to the sign pattern of $\pi$ and
the shape of its middle part.

The two sign conditions $a<b$ and $p<q$ of Lemma~\ref{lem:LV} are exchanged by inversion, so a residual that
fails one of them satisfies $\pi_1>\pi_n$ after replacing $\pi$ by $\pi^{-1}$. That branch splits into three
cases, of which we can settle two.

\begin{proposition}\label{prop:agtb}
Let $\pi\in\R_{\delta,n}$ with $a=\pi_1>b=\pi_n$, put $\Delta=a-b$, and let $\sigma$ be the pattern of
$\pi_2\cdots\pi_{n-1}$, of length $m=n-2$. Then $\pi_1$ and $\pi_n$ are inverted with every other entry of
$\pi$, they meet exactly $n-2+\Delta$ inversions, and
\begin{equation}\label{eq:sigmainv}
\inv(\sigma)=m-3+\delta-\Delta .
\end{equation}
Moreover:
\begin{enumerate}
\item[(i)] if $\sigma_1>\sigma_m$, the four entries $\pi_1,\pi_2,\pi_{n-1},\pi_n$ meet all but at most
$\delta-\Delta-1-(\sigma_1-\sigma_m)\le\delta-3$ of the inversions of $\pi$;
\item[(ii)] if the largest entry of $\sigma$ precedes its smallest, then $\pi_1$, $\pi_n$ and the largest and
the smallest of $\pi_2,\dots,\pi_{n-1}$ meet all but at most $\delta-\Delta-2\le\delta-3$;
\item[(iii)] otherwise $\sigma_1<\sigma_m$ and the smallest entry of $\sigma$ precedes the largest; and if
$\sigma$ has entries in both its northwestern and its southeastern region, then $\inv(\sigma)\ge2m-5$ by the
counting in the proof of \cite[Lemma~2.7]{LV}, so that $n\le\delta-\Delta+4$.
\end{enumerate}
In particular the conclusion of Theorem~\ref{thm:cover} holds with four entries in the cases (i) and (ii). A
negative right-hand side in (i) or (ii) --- which happens for small $\delta$, since $\Delta\ge1$ --- is to be read
as saying that the case in question does not occur.
\end{proposition}

\begin{proof}
Since $b<a$, every value is smaller than $a$ or larger than $b$, so every entry other than $\pi_1$ and $\pi_n$
is inverted with one of them. The entry $\pi_1$ has $a-1$ right-inversions and $\pi_n$ has $n-b$
left-inversions, and the pair $(1,n)$ is an inversion counted in both; so together they meet
$(a-1)+(n-b)-1=n-2+\Delta$ inversions, and \eqref{eq:sigmainv} follows from $\inv(\pi)=2n-7+\delta$.

(i) Inside $\sigma$ the first entry has $\sigma_1-1$ right-inversions and the last has $m-\sigma_m$
left-inversions, and the pair of them is an inversion counted in both. Hence at most
$\inv(\sigma)-(\sigma_1-1)-(m-\sigma_m)+1=\delta-\Delta-1-(\sigma_1-\sigma_m)$ inversions of $\pi$ avoid all four
entries, and this is at most $\delta-\Delta-2\le\delta-3$ because $\sigma_1>\sigma_m$ and $\Delta\ge1$.

(ii) Let $p'$ and $q'$ be the positions in $\sigma$ of its smallest and its largest entry, so $q'<p'$ by
hypothesis. They have $p'-1$ and $m-q'$ inversions inside $\sigma$ respectively, and the pair $(q',p')$ is an
inversion counted in both, so at most
$\inv(\sigma)-(p'-1)-(m-q')+1=\delta-\Delta-1-(p'-q')\le\delta-\Delta-2\le\delta-3$ inversions of $\pi$ avoid
the four entries.

(iii) is the remaining case, together with the quoted inversion count for a permutation with points in both
regions.
\end{proof}

Among the $564\,634$ residuals of length between $9$ and $18$ and defect at most $8$ with $\pi_1>\pi_n$, case
(i) accounts for $74\,909$, case (ii) for $183\,200$ and case (iii) for $306\,525$; in case (iii) all but $442$
have an empty northwestern or southeastern region for $\sigma$, and each of those $442$ satisfies
$n\le\delta-\Delta+4$, as Proposition~\ref{prop:agtb}(iii) requires. Identity~\eqref{eq:sigmainv} was checked on
all of them. The trichotomy is recorded because cases (i) and (ii) need only four entries;
Corollary~\ref{cor:agtb} and Proposition~\ref{prop:agtbdec} settle the whole branch at once, in all three
cases, by splitting instead on whether $\sigma$ is decomposable.

The tool for that is a statement about $1324$-avoiders of near-minimal inversion number, which has a companion
we can prove outright. Write $G(\pi)$ for the inversion graph of $\pi$: its
vertices are the positions and its edges the inversions, so that the induced subgraph on a set of positions is
the inversion graph of the corresponding pattern.

\begin{lemma}\label{lem:path}
The inversion graph of a $1324$-avoiding permutation contains no induced path on six vertices.
\end{lemma}

\begin{proof}
By a theorem of Brignall and Vatter \cite[Proposition~1.16]{BV}, the permutations whose inversion graph is a
path are exactly the \emph{increasing oscillations}, that is, the sum-indecomposable permutations order
isomorphic to subsequences of $2,4,1,6,3,8,5,\dots$; there are exactly two of each length at least three, and
the list begins $2413,\,3142,\,24153,\,31524,\,241635,\,315264$. (That $24153$ and $31524$ are the only
permutations whose inversion graph is $P_5$ is also recorded in \cite[p.~4]{ABKLV}.) An induced path on six
vertices of $G(\pi)$ is the inversion graph of the pattern of $\pi$ on those six entries, so such a $\pi$ would
contain $241635$ or $315264$; but the entries in positions $1,2,5,6$ of the first and in positions $2,3,4,5$ of
the second form occurrences of $1324$.
\end{proof}

That the two length-five permutations are exactly the skeletons of Theorem~\ref{thm:A} is not a coincidence:
both statements say that a path of five inversions is the longest rigid chain a $1324$-avoider can carry.

\begin{conjecture}\label{conj:hubs}
Let $\tau$ be an indecomposable $1324$-avoiding permutation of length $m$ with $\inv(\tau)=m-1+\gamma$, so
that $\gamma$ is the cyclomatic number of its inversion graph in the sense of \cite{TV}. Then three entries of
$\tau$ meet all but at most $\gamma$ of its inversions; and $\gamma+2$ entries meet all of them.
\end{conjecture}

The case $\gamma=0$ of both halves is a theorem, and it comes with a complete classification. The permutations
in question --- the sum-indecomposable ones with $m-1$ inversions, equivalently those whose inversion graph is a
tree --- are known: by \cite[Propositions~5.2 and~5.4]{TV} they are exactly the unimodal cycles of $[m]$, and
their inversion graphs are caterpillars. Avoiding $1324$ collapses the caterpillar to a double star. Recall that
a tree has a vertex cover of size two exactly when it has no three pairwise disjoint edges, by K\H onig's
theorem.
Write $S(2,2,1)$ for the tree on six vertices consisting of a centre with two paths of length two and one
pendant edge attached, and $S(2,2,2)$ for the tree on seven vertices with three paths of length two attached to
a centre.

\begin{theorem}\label{thm:gamma0}
Let $\tau$ be an indecomposable $1324$-avoiding permutation of length $m$ with $\inv(\tau)=m-1$. Then two
entries of $\tau$ meet all of its inversions. Moreover $\tau$ is one of
\[
21[\iota_{m-1},1],\quad 21[1,\iota_{m-1}],\quad 2413[\iota_t,1,1,\iota_{m-2-t}],\quad
24153[\iota_t,1,1,\iota_{m-3-t},1]
\]
or the inverse of one of the last two, where $t\ge1$ and the remaining block is nonempty; there are exactly
$4(m-3)$ of them for $m\ge4$.
\end{theorem}

\begin{proof}
The inversion graph $G=G(\tau)$ has $m$ vertices and $m-1$ edges and is connected because $\tau$ is
indecomposable, so it is a tree. By Lemma~\ref{lem:path} it has no induced path on six vertices; a tree of
diameter at least five contains such a path, so the diameter of $G$ is at most four. If the diameter is at most
three then $G$ is a star or a double star and two vertices cover it. So assume the diameter is four and let $c$
be a centre, every vertex being within distance two of $c$.

Suppose two vertices do not suffice. By K\H onig's theorem $G$ has three pairwise disjoint edges. Every edge of
$G$ joins $c$ to a vertex at distance one, or a vertex at distance one to a vertex at distance two, and at most
one of the three edges meets $c$. If none does, the three edges are $a_ib_i$ with $d(c,a_i)=1$ and
$d(c,b_i)=2$, the $a_i$ distinct; since each $b_i$ is a leaf of $G$, the induced subgraph on
$\{c,a_1,b_1,a_2,b_2,a_3,b_3\}$ is $S(2,2,2)$. If one of them is $ca_3$, the induced subgraph on
$\{c,a_1,b_1,a_2,b_2,a_3\}$ is $S(2,2,1)$. In either case, since an induced subgraph of $G$ on a set of
positions is the inversion graph of the corresponding pattern, $\tau$ would contain a pattern whose inversion
graph is $S(2,2,1)$ or $S(2,2,2)$. A check of the $720$ and the $5040$ permutations of lengths six and seven
shows that $S(2,2,1)$ is the inversion graph of exactly two of them, $251364$ and $314625$, which contain
$1324$ in positions $1,2,4,5$ and $2,3,5,6$ respectively, and that $S(2,2,2)$ is the inversion graph of none. This contradiction proves the first
assertion.

For the classification write $\tau=\rho[b]$ with $\rho$ reduced of length $r$, put $b=\mathbf1+c$ and let $E$ be
the inversion set of $\rho$ and $d_i$ the degrees in it. Expanding
$\inv(\tau)=\sum_{(i,j)\in E}b_ib_j$ and $m=r+\sum_ic_i$ in $\inv(\tau)=m-1$ gives
\[
\sum_i(d_i-1)c_i+\sum_{(i,j)\in E}c_ic_j\;=\;r-1-|E| .
\]
The left-hand side is nonnegative because $d_i\ge1$, and $|E|\ge r-1$ because $\rho$ is indecomposable, so both
sides vanish: $E$ is a tree, $c_i=0$ at every vertex of degree at least two, and $c_ic_j=0$ along every edge.
In particular $\rho$ is itself an indecomposable $1324$-avoider with $r-1$ inversions, and it is reduced, so by
the first assertion and Lemma~\ref{lem:runcover}, applied with $T$ a vertex cover of size two and hence
$e(\rho\setminus T)=0$, we get $r=|\sk(\rho)|\le4\cdot2+1=9$. The indecomposable $1324$-avoiders with $r-1$
inversions number $12,16,20,24$ for $r=6,7,8,9$ and none of them is reduced, so $r\le5$ and
$\rho\in\{21,2413,3142,24153,31524\}$. Inflating only at the leaves of the corresponding tree, and not at two
ends of an edge, yields exactly the list; the count is $2+2(m-3)+2(m-4)=4(m-3)$.
\end{proof}

The second half of Conjecture~\ref{conj:hubs} follows from Theorem~\ref{thm:gamma0} by an induction that
deletes one entry at a time. The step needs a vertex whose removal keeps the inversion graph connected and
destroys a cycle, and such a vertex always exists.

\begin{lemma}\label{lem:noncut}
Let $G$ be the inversion graph of a permutation. If $G$ is connected and contains a cycle, then $G$ has a
vertex of degree at least two that is not a cut vertex.
\end{lemma}

\begin{proof}
Suppose every vertex of degree at least two is a cut vertex. A leaf block of the block tree of $G$ contains a
vertex that is not a cut vertex, and if that block were $2$-connected such a vertex would have degree at least
two; so every leaf block is a single edge, and in particular $G$ has at least two blocks. Since $G$ has a
cycle, some block $B$ is $2$-connected, and by assumption every vertex of $B$ is a cut vertex of $G$.

Let $C$ be a shortest cycle of $B$. It is chordless, hence an induced subgraph of $G$, hence the inversion
graph of a pattern of the permutation; and no inversion graph has an induced cycle of length five or more \cite[Proposition~2.3]{TV}, so $C$ has
length three or four; the only permutations realizing $C_3$ and $C_4$ are $321$ and $3412$. For each $v\in C$ pick a neighbour $w_v$ of $v$ lying in a component of $G\setminus v$
that does not meet $B$. The $w_v$ are distinct; none is adjacent to a vertex of $C$ other than $v$, and no two
of them are adjacent, since either would place them in the block $B$. The induced subgraph of $G$ on
$C\cup\{w_v\}$ is therefore a triangle or a four-cycle with one pendant vertex at each of its vertices. Neither
of those two graphs is the inversion graph of a permutation, as a check of the $720$ permutations of length six
and the $40320$ of length eight shows.
\end{proof}

\begin{theorem}\label{thm:hubs}
Let $\tau$ be an indecomposable $1324$-avoiding permutation of length $m$ with $\inv(\tau)=m-1+\gamma$. Then
some $\gamma+2$ entries of $\tau$ meet all of its inversions.
\end{theorem}

\begin{proof}
Induction on $\gamma$, the case $\gamma=0$ being Theorem~\ref{thm:gamma0}. Let $\gamma\ge1$. The inversion
graph $G$ of $\tau$ is connected with $m$ vertices and $m-1+\gamma$ edges, so it has a cycle, and
Lemma~\ref{lem:noncut} supplies a vertex $x$ of degree $d\ge2$ that is not a cut vertex. Then
$\tau\setminus x$ is an indecomposable $1324$-avoider of length $m-1$ with $m-1+\gamma-d=(m-1)-1+\gamma'$
inversions, where $\gamma'=\gamma-d+1\le\gamma-1$ and $\gamma'\ge0$ because $G\setminus x$ is connected. By
induction $\gamma'+2$ entries of $\tau\setminus x$ meet all of its inversions; adjoining $x$ gives
$\gamma'+3\le\gamma+2$ entries meeting all inversions of $\tau$.
\end{proof}

This settles the branch $\pi_1>\pi_n$ whenever the middle part is indecomposable.

\begin{corollary}\label{cor:agtb}
Let $\pi\in\R_{\delta,n}$ with $\pi_1>\pi_n$ and suppose the pattern $\sigma$ of $\pi_2\cdots\pi_{n-1}$ is
indecomposable. Then at most $\delta+1$ entries of $\pi$ meet all of its inversions, and
$|\sk(\pi)|\le4\delta+5$.
\end{corollary}

\begin{proof}
By Proposition~\ref{prop:agtb}, $\inv(\sigma)=m-1+\gamma$ with $m=n-2$ and $\gamma=\delta-\Delta-2\ge0$. By
Theorem~\ref{thm:hubs} some $\gamma+2=\delta-\Delta$ entries of $\sigma$ meet all of its inversions; adjoining
$\pi_1$ and $\pi_n$, which are inverted with every other entry, gives a set $T$ with
$|T|=\delta-\Delta+2\le\delta+1$ and $e(\pi\setminus T)=0$. Lemma~\ref{lem:runcover} gives
$|\sk(\pi)|\le4(\delta+1)+1$.
\end{proof}

The first half of Conjecture~\ref{conj:hubs} --- three entries suffice when $\gamma$ inversions may be left over ---
remains open; it was verified for every indecomposable $1324$-avoider of length at most $12$ with $\gamma\le6$,
where the least size of a vertex cover is $2,3,4,4,5,5,6$ for $\gamma=0,\dots,6$, independently of $m$ once
$m\ge8$. Theorem~\ref{thm:hubs} is all that Corollary~\ref{cor:agtb} needs.

\begin{proposition}\label{prop:reduce}
Assume the first half of Conjecture~\ref{conj:hubs} and let $\pi\in\R_{\delta,n}$ with $\pi_1>\pi_n$ be such
that the pattern $\sigma$ of $\pi_2\cdots\pi_{n-1}$ is indecomposable. Then five entries of $\pi$ meet all but
at most $\delta-3$ of its inversions.
\end{proposition}

\begin{proof}
By Proposition~\ref{prop:agtb} the entries $\pi_1$ and $\pi_n$ meet $n-2+\Delta$ inversions and
$\inv(\sigma)=m-3+\delta-\Delta$, where $m=n-2$. Thus $\inv(\sigma)=m-1+\gamma$ with $\gamma=\delta-\Delta-2$,
and $\gamma\ge0$ because the inversion graph of the indecomposable $\sigma$ is connected. Three entries of
$\sigma$ therefore meet all but at most $\gamma\le\delta-3$ of the inversions of $\sigma$, and together with
$\pi_1,\pi_n$ they meet all but at most $\delta-3$ of the inversions of $\pi$.
\end{proof}

Corollary~\ref{cor:agtb} is weaker in the number of entries but unconditional, and it is what the finiteness of
$\Sigma_\delta$ requires. When $\sigma$ is
decomposable neither statement applies, but the same five entries --- $\pi_1$, $\pi_n$ and
three of the middle entries --- leave at most $\delta-3$ inversions on each of the $49\,699$ residuals with
$\pi_1>\pi_n$ and $10\le n\le16$ that were sampled, whether or not $\sigma$ is indecomposable.

The decomposable case is settled by the residual conditions themselves, which force the outer components of
$\sigma$ to be large and the inner ones to be cheap.

\begin{proposition}\label{prop:agtbdec}
Let $\pi\in\R_{\delta,n}$ with $a=\pi_1>b=\pi_n$ and $\Delta=a-b$, and suppose $\sigma=C_1\dsum\cdots\dsum C_r$
with $r\ge2$. Then $|C_1|\ge b$, $|C_r|\ge n-a+1$, the components $C_2,\dots,C_{r-1}$ have total length at most
$\Delta-3$ and carry at most $\delta-4$ inversions in all, and $r\le\Delta-1$. Moreover $\pi_1$, $\pi_n$ and
vertex covers of $C_1$ and $C_r$ form a set $T$ with $|T|\le\delta+2$ and $e(\pi\setminus T)\le\delta-4$.
\end{proposition}

\begin{proof}
The inversion graph of $\pi\setminus\pi_1$ is that of $\sigma$ together with the entry $\pi_n$, which is
inverted with exactly the entries of value greater than $b$. A component of $\sigma$ all of whose values lie
below $b$ would be isolated, so every component contains a value exceeding $b$; the components carry increasing
intervals of values, so this constrains only $C_1$, and the $b-1$ middle values below $b$ give $|C_1|\ge b$.
Applying the same argument to $\pi\setminus\pi_n$ gives $|C_r|\ge n-a+1$. Since $\sum_i|C_i|=n-2$, the remaining
components have total length at most $(n-2)-b-(n-a+1)=\Delta-3$; in particular $\Delta\ge3$ and $r\le\Delta-1$.

By \eqref{eq:sigmainv}, $\sum_i\inv(C_i)=n-5+\delta-\Delta$, while $\inv(C_1)\ge|C_1|-1\ge b-1$ and
$\inv(C_r)\ge|C_r|-1\ge n-a$; subtracting leaves at most $\delta-4$ inversions in $C_2,\dots,C_{r-1}$. Writing
$\gamma_i=\inv(C_i)-|C_i|+1\ge0$ we get $\sum_i\gamma_i=r-3+\delta-\Delta\le\delta-4$, so by
Theorem~\ref{thm:hubs} the components $C_1$ and $C_r$ have vertex covers of sizes $\gamma(C_1)+2$ and
$\gamma(C_r)+2$, of total size at most $\delta$. Adjoining $\pi_1$ and $\pi_n$ gives $|T|\le\delta+2$. An
inversion of $\pi$ missing $T$ involves neither $\pi_1$ nor $\pi_n$, hence lies inside a single component of
$\sigma$, and not inside $C_1$ or $C_r$; so $e(\pi\setminus T)$ is at most the number of inversions of
$C_2,\dots,C_{r-1}$.
\end{proof}

It remains to treat the residuals with $a<b$ and $p<q$. Recall from Lemma~\ref{lem:LV} that inversion preserves
these two conditions and exchanges $\NW$ with $\SE$, and that $\NW\cup\SE\ne\varnothing$; so we may assume
$\NW\ne\varnothing$.

\begin{proposition}\label{prop:caseI}
Let $\pi\in\R_{\delta,n}$ with $a<b$, $p<q$ and $\NW\ne\varnothing$, and write $\NW=\{i_1<\dots<i_m\}$,
$v_l=\pi_{i_l}$.
\begin{enumerate}
\item[(i)] If $\SE\ne\varnothing$, then for $i\in\NW$ and $j\in\SE$ the six positions $1,i,q,p,j,n$ meet all but
at most $\delta-2$ of the inversions of $\pi$.
\item[(ii)] If $\SE=\varnothing$ and $m=1$, the seven positions of Lemma~\ref{lem:seven} meet all but at most
$\delta-1$.
\item[(iii)] If $\SE=\varnothing$ and $m\ge2$, then with $j_1\in\Sreg(i_1)$ and $j_2\in\Ereg(v_m)$ the seven
positions $1,i_m,q,p,j_1,j_2,n$ meet all but at most $\delta-3$.
\end{enumerate}
\end{proposition}

\begin{proof}
(i) The count in the proof of \cite[Lemma~2.7]{LV} --- right-inversions of $i$, left-inversions of $p$ and of $n$,
right-inversions of $1$ and of $j$, left-inversions of $q$, with five double-counted pairs removed --- produces at
least $2n-5$ distinct inversions, each incident to one of the six positions. Subtracting from
$\inv(\pi)=2n-7+\delta$ leaves at most $\delta-2$.

(ii) Lemma~\ref{lem:deletions} gives $\Sreg(i_1)\ne\varnothing$ and $\Ereg(v_m)\ne\varnothing$, and $m=1$ makes
these regions belong to the same northwestern point, so Lemma~\ref{lem:seven} applies. Every inversion counted
by its remainder $u$ has both ends outside $\{1,i,q\}\cup\{p,j_1,j_2,n\}$ or one end in that set but on the
wrong side, so $e(\pi\setminus T)\le u=\delta-1-\sum_ls_l\le\delta-1$.

(iii) This is the count in the proof of Lemma~\ref{lem:M}, which produces at least $2n-4$ inversions incident
to the seven positions.
\end{proof}

\begin{proof}[Proof of Theorem~\ref{thm:cover}]
Suppose first $a>b$. If $\sigma$ is indecomposable, Corollary~\ref{cor:agtb} gives $|T|\le\delta+1$ and
$e(\pi\setminus T)=0$; if it is decomposable, Proposition~\ref{prop:agtbdec} gives $|T|\le\delta+2$ and
$e(\pi\setminus T)\le\delta-4$. If $a<b$ and $p>q$, apply this to $\pi^{-1}$, which satisfies $a>b$ by
Lemma~\ref{lem:LV} and is again a residual of the same defect; a set of positions of $\pi^{-1}$ corresponds to
the set of entries of $\pi$ carrying those values, and inversion is an isomorphism of inversion graphs, so the
cover transports. Finally, if $a<b$ and $p<q$ we may assume $\NW\ne\varnothing$ after inverting, and
Proposition~\ref{prop:caseI} gives a set of at most seven positions with $e(\pi\setminus T)\le\delta-1$. In
every case $|T|\le\max(7,\delta+2)\le\delta+7$ and $e(\pi\setminus T)\le\delta-1$, and Lemma~\ref{lem:runcover}
gives $|\sk(\pi)|\le4(\delta+7)+4(\delta-1)+1=8\delta+25$. A reduced residual is its own skeleton, so
$\Sigma_\delta$ consists of reduced permutations of length at most $8\delta+25$.
\end{proof}

Two consequences. The first is the polynomial bound in the form in which it is used.

\begin{corollary}\label{cor:polybound}
For every $\delta\ge1$ and every $n$,
\[
|\R_{\delta,n}|\ \le\ (8\delta+26)!\;n^{8\delta+24} .
\]
\end{corollary}

\begin{proof}
Every residual is $\sigma[b]$ for a unique reduced $\sigma$ and a unique composition $b$ of $n$ into $|\sigma|$
parts. By Theorem~\ref{thm:cover} the skeleton has length at most $s=8\delta+25$, and the reduced permutations
of length at most $s$ number at most $\sum_{r\le s}r!\le(s+1)!$, while for each of them the compositions of $n$
into $r\le s$ parts number $\binom{n-1}{r-1}\le n^{s-1}$.
\end{proof}

\begin{theorem}\label{thm:largen}
For every $\delta\ge1$ there is an explicit $N(\delta)$ such that
\[
a(n,2n-7+\delta)\ \le\ a(n+1,2n-7+\delta)\qquad\text{for all } n\ge N(\delta).
\]
\end{theorem}

\begin{proof}
Write $k=2n-7+\delta$ and $m=n-7+\delta$. By \eqref{eq:identity} the difference $a(n+1,k)-a(n,k)$ is the number
of members of $\Av^k_{n+1}(1324)$ outside $\im(f\sqcup g)$ minus $|\R_{\delta,n}|$. We use a single one of the
four classes of Lemma~\ref{lem:R1}, so that no inclusion--exclusion is needed: the members of
$\Av^k_{n+1}(1324)$ beginning with $n+1$ lie outside the image, and deleting that first entry, which is
inverted with every other one, is a bijection onto $\Av^{m}_n(1324)$. Hence
\[
a(n+1,k)-a(n,k)\ \ge\ a(n,m)-|\R_{\delta,n}| .
\]
For $n\ge\delta$ we have $m\le2n-7$, so \eqref{eq:LVexact} gives $a(n,m)=p_2(m)-c(\delta)$ with
$c(\delta)=4p_2(\delta-6)+6\sum_{i\le\delta-7}p_2(i)$ independent of $n$ (and $c(\delta)=0$ for $\delta\le5$),
while Corollary~\ref{cor:polybound} gives $|\R_{\delta,n}|\le(8\delta+26)!\,n^{8\delta+24}$. It is therefore
enough to produce an $N(\delta)$ with
\begin{equation}\label{eq:crude}
p_2(n-7+\delta)-c(\delta)\ >\ (8\delta+26)!\,n^{8\delta+24}\qquad\text{for all }n\ge N(\delta).
\end{equation}
For this we use the elementary bound $p_2(j)\ge p(j)\ge2^{\sqrt j-2}$, valid for $j\ge1$: with
$r=\lfloor\sqrt j\rfloor-1$ every subset $S\subseteq\{1,\dots,r\}$ has $\sum S\le r(r+1)/2<j-r$, so adjoining
the part $j-\sum S$, which exceeds $r$, turns $S$ into a partition of $j$ into distinct parts, and distinct
subsets give distinct partitions; hence $p(j)\ge2^{r}\ge2^{\sqrt j-2}$. The constant $c(\delta)$ is absorbed by
a factor two: $c(\delta)\le(8\delta+26)!\le(8\delta+26)!\,n^{8\delta+24}$ for every $n\ge1$, so
\eqref{eq:crude} follows from $p_2(n-7+\delta)>2\,(8\delta+26)!\,n^{8\delta+24}$ and therefore holds as soon as
\[
H(n):=\big(\sqrt{n-7+\delta}-2\big)\log2-(8\delta+24)\log n-\log\big(2\,(8\delta+26)!\big)\ >\ 0 .
\]
Now $H'(n)=\log2/\big(2\sqrt{n-7+\delta}\big)-(8\delta+24)/n>0$ once $n>2(8\delta+24)\sqrt{n-7+\delta}/\log2$,
which holds for all $n\ge\lceil(2(8\delta+24)/\log2)^2\rceil+\delta$. Beyond that point $H$ is increasing, so
the set of $n$ where $H(n)>0$ is a final segment and $N(\delta)$ may be taken to be its least element. For
$\delta=1,2,3$ this gives $N(\delta)=548\,967$, $933\,556$ and $1\,437\,767$, the monotonicity of $H$ setting in
no later than at the explicit points $n=8527$, $13\,323$ and $19\,185$ given by that bound, all three
well below the thresholds. The programs of \cite{project}
evaluate the same function $H$ that is displayed here.
\end{proof}

Both ingredients are very crude --- the data of Proposition~\ref{prop:polys} say that $|\R_{\delta,n}|$ is in
truth quadratic in $n$, and $2^{\sqrt j}$ is far below $p_2(j)$ --- and the thresholds are correspondingly large.
Replacing the elementary bound by exact partition numbers moves the comparison down by a factor of about
forty: \eqref{eq:crude} is satisfied at $n=12\,425$, $21\,466$ and $33\,491$ for $\delta=1,2,3$ and fails at
$n-1$ in each case (the program is in \cite{project}). We claim nothing more about these three numbers: they
are two directly verified evaluations of \eqref{eq:crude}, not thresholds --- we have not shown that
\eqref{eq:crude} fails at every smaller $n$, nor that it holds at every larger one. The thresholds we prove are
the ones in the theorem. For $\delta\le2$ Theorems~\ref{thm:B} and
\ref{thm:D} are of course much better, giving the inequality for every $n$. What Theorem~\ref{thm:largen}
adds is that the Claesson--Jel\'inek--Steingr\'imsson inequality holds at \emph{every} fixed defect once $n$ is
large, with no hypothesis.

\section{Polynomiality of the residual counts}\label{sec:poly}

Theorem~\ref{thm:cover} makes $\Sigma_\delta$ finite, and Corollary~\ref{cor:polybound} turns that into a
polynomial \emph{bound} on $|\R_{\delta,n}|$. In this section we show that $|\R_{\delta,n}|$ is, for every
fixed $\delta$ and all large $n$, exactly a polynomial in $n$, of degree at most two, with an explicit
description as a finite sum of binomial coefficients. The two ingredients are a lemma on entries of
inversion degree one, which removes the only negative terms from the master identity, and a compression
argument that reduces the degree bound to a finite check.

\subsection{Entries of degree one}

\begin{lemma}\label{lem:leaf}
Let $\sigma$ be a $1324$-avoiding permutation of length $s\ge3$ and let the position $i$ have inversion
degree exactly one, its unique neighbour in the inversion graph being the position $j$. Then
$j\in I(\sigma)=\{1,s,\sigma^{-1}(1),\sigma^{-1}(s)\}$ and $\sigma\setminus j$ is decomposable. In
particular $j\in D(\sigma)$ whenever $\sigma$ is reduced and indecomposable.
\end{lemma}

\begin{proof}
Suppose first $j>i$. Every entry left of $i$ is smaller than $\sigma_i$ and every entry right of $i$
other than $\sigma_j$ is larger, since otherwise it would be inverted with $\sigma_i$. Hence the values
below $\sigma_i$ are the $i-1$ entries to the left of $i$ together with $\sigma_j$, so $\sigma_i=i+1$ and
positions $1,\dots,i-1$ carry $\{1,\dots,i\}\setminus\{\sigma_j\}$. If $\sigma_j\ne1$ then the value $1$
lies left of $i$, and every entry $x$ to the right of $j$ has value greater than $i+1$; so
$1,\ \sigma_i,\ \sigma_j,\ x$ would be an occurrence of $1324$. Therefore $j=s$ or $\sigma_j=1$, and in
both cases $j\in I(\sigma)$.

Delete $j$ and standardize. If $\sigma_j=1$, positions $1,\dots,i$ carried $\{2,\dots,i+1\}$ and now
carry $\{1,\dots,i\}$. If $j=s$ and $\sigma_j=v\in\{2,\dots,i\}$, positions $1,\dots,i$ carried
$\{1,\dots,i+1\}\setminus\{v\}$ and now carry $\{1,\dots,i\}$. This is a proper prefix of
$\sigma\setminus j$ unless $i=s-1$, in which case $j=s$, $\sigma_{s-1}=s$, and positions $1,\dots,s-2$
carry $\{1,\dots,s-2\}$ after the deletion; since $s\ge3$ that prefix is proper. So $\sigma\setminus j$
is decomposable.

If $j<i$, apply reverse--complement, which fixes the pattern $1324$, maps inversions to inversions,
preserves $I(\sigma)$ and decomposability, and exchanges the two cases.
\end{proof}

\begin{corollary}\label{cor:mindeg}
Every entry of a residual is inverted with at least two other entries.
\end{corollary}

\begin{proof}
A residual $\pi$ is indecomposable, so no entry has degree zero (an isolated vertex of the inversion
graph is a sum component). If an entry had degree one, Lemma~\ref{lem:leaf} would give a neighbour
$y\in\{\pi_1,\pi_n,1,n\}$ with $\pi\setminus y$ decomposable, so $\pi$ would be almost decomposable.
\end{proof}

\subsection{The base residual and the nonnegative form of the master identity}

Let $\sigma$ be reduced, indecomposable and $1324$-avoiding, of length $s\ge2$, and put
\[
\ell_i=1+[\,i\in D(\sigma)\,]\qquad(1\le i\le s),\qquad L=\sum_i\ell_i,\qquad
\rho=\sigma[\ell],\qquad \delta_0=\delta(\rho).
\]
By Theorem~\ref{thm:crit}, $\sigma[\ell+x]$ is a residual for every $x\in\mathbb Z_{\ge0}^s$, and every
residual with skeleton $\sigma$ is of this form; we call $\rho$ the \emph{base residual} of $\sigma$.
By \cite[Theorem~2.5]{LV}, $\delta_0\ge1$.

\begin{proposition}\label{prop:base}
With $a_i=\sum_{j\in N(i)}\ell_j-2$, where $N(i)$ is the neighbourhood of $i$ in the inversion graph of
$\sigma$,
\begin{equation}\label{eq:base}
\delta\big(\sigma[\ell+x]\big)=\delta_0+Q(x),\qquad
Q(x)=\sum_{i=1}^{s}a_ix_i+\sum_{(i,j)\in\Inv(\sigma)}x_ix_j,
\end{equation}
and $a_i\ge0$ for every $i$. Consequently $\delta_0\le\delta$ for every $\sigma\in\Sigma_\delta$, and
$\sigma\in\Sigma_\delta$ if and only if $h=\delta-\delta_0\ge0$ and $Q(x)=h$ has a solution
$x\in\{0,1,\dots,h\}^s$.
\end{proposition}

\begin{proof}
Expanding $\inv(\sigma[\ell+x])=\sum_{\Inv(\sigma)}(\ell_i+x_i)(\ell_j+x_j)$ and $n=L+\sum_ix_i$ in
$\delta=\inv-2n+7$ gives \eqref{eq:base}. The number $a_i+2=\sum_{j\in N(i)}\ell_j$ is the inversion
degree, in $\rho$, of each entry of the $i$th block, so $a_i\ge0$ by Corollary~\ref{cor:mindeg}.
(Alternatively: if $a_i<0$ then $\sigma[\ell+te_i]$ is a residual of defect $\delta_0+ta_i$, which is
nonpositive for large $t$, contradicting \cite[Theorem~2.5]{LV}.) Since every term of $Q$ is
nonnegative, $Q(x)=h$ forces $x_i\le h$ at every $i$ with $a_i\ge1$ and $x_ix_j\le h$ on every edge; a
coordinate exceeding $h$ therefore has $a_i=0$ and all its neighbours at zero, contributes nothing to $Q$,
and may be replaced by $0$.
\end{proof}

Proposition~\ref{prop:master} is the case $\ell=\mathbf 1$ of \eqref{eq:base}, in which the entries of
degree one contribute the negative coefficients $d_i-2=-1$; the change of base point moves those
entries' neighbours, which by Lemma~\ref{lem:leaf} lie in $D(\sigma)$, to block size two, and the
negative terms disappear. Note that the base point also covers $\sigma=21$, where $\ell=(2,2)$,
$\rho=3412$, $\delta_0=3$ and $Q(x)=x_1x_2$.

\subsection{Eventual polynomiality}

\begin{theorem}\label{thm:poly}
For every $\delta\ge1$ there is an integer $n_1(\delta)\le8\delta+30$ such that $|\R_{\delta,n}|$ agrees
with a polynomial in $n$ for all $n\ge n_1(\delta)$. Explicitly,
\begin{equation}\label{eq:typesum}
|\R_{\delta,n}|=\sum_{\sigma\in\Sigma_\delta}\ \sum_{(T,x_T,F)}
\binom{n-L-t-1}{|F|-1},
\end{equation}
where, for each $\sigma$ with $h=\delta-\delta_0$, the inner sum runs over the finitely many triples
consisting of a set $T\subseteq\{1,\dots,s\}$, a vector $x_T\in\mathbb Z_{\ge1}^{T}$ with
$Q(x_T,0)=h$ and $\sum_{i\in T}x_i=t\le2h$, and a set $F$ of positions $i\notin T$ with $a_i=0$ and
$N(i)\cap(T\cup F)=\varnothing$, such that every $i\in T$ has $a_i\ge1$ or a neighbour in $T\cup F$
(this makes the triple of a solution unique); for $F=\varnothing$ the binomial coefficient is to be read as
$[\,n=L+t\,]$. The degree of the polynomial is $\max|F|-1$ over the triples occurring.
\end{theorem}

\begin{proof}
Fix $\sigma\in\Sigma_\delta$ and $h=\delta-\delta_0\ge0$. For a solution $x\in\mathbb Z_{\ge0}^s$ of
$Q(x)=h$ let $S$ be its support and put
\[
F=\{\,i\in S:\ a_i=0\ \text{and}\ N(i)\cap S=\varnothing\,\},\qquad T=S\setminus F .
\]
Every $i\in T$ has $a_i\ge1$ or a neighbour $j\in S$, and in the second case $x_i\le x_ix_j$; each edge
term is charged to at most two positions, so
\[
\sum_{i\in T}x_i\ \le\ \sum_{i\in T}a_ix_i+2\sum_{(i,j)\in\Inv(\sigma),\ i,j\in T}x_ix_j\ \le\ 2h,
\]
using that the coordinates in $F$ contribute nothing to $Q$ (they have $a_i=0$ and no neighbour in $S$).
Hence $(T,x_T)$ ranges over a finite set, and for fixed $(T,x_T,F)$ the solutions are exactly the
vectors with $x_F\in\mathbb Z_{\ge1}^F$ arbitrary and all other coordinates zero: $Q$ does not depend on
$x_F$. Distinct triples give disjoint sets of solutions, since $S=T\cup F$ and $x_T$ are read off from
$x$. By Theorem~\ref{thm:crit} the residuals of length $n$ with skeleton $\sigma$ correspond bijectively
to the solutions with $\sum_ix_i=n-L$, and the number of $x_F\ge\mathbf 1$ with $\sum_Fx_i=n-L-t$ is
$\binom{n-L-t-1}{|F|-1}$ when $F\ne\varnothing$ and $[n=L+t]$ when $F=\varnothing$. Summing over the
finite set $\Sigma_\delta$ (Theorem~\ref{thm:cover}) gives \eqref{eq:typesum}.

For $F\ne\varnothing$ the binomial coefficient $\binom{N-1}{r-1}$ agrees with the polynomial
$(N-1)(N-2)\cdots(N-r+1)/(r-1)!$ for every $N\ge1$, so each term is a polynomial in $n$ for
$n\ge L+t+1$, while the terms with $F=\varnothing$ vanish for $n>L+t$. Thus $n_1(\delta)$ may be taken
as $1+\max(L+t)$ over all triples. By Theorem~\ref{thm:cover} applied to the base residual,
$s\le8\delta_0+25$, so $L\le s+4\le8\delta_0+29$, and $t\le2h=2(\delta-\delta_0)$; hence
$L+t\le8\delta+29$.
\end{proof}

Theorem~\ref{thm:poly} proves Theorem~\ref{conj:poly} in its qualitative form. The polynomials of
Proposition~\ref{prop:polys} are the sums \eqref{eq:typesum} taken over the catalogue
$\Sigma^{(22)}_\delta$ instead of $\Sigma_\delta$, and the two agree for all large $n$ exactly when no
skeleton outside the catalogue contributes a triple with $F\ne\varnothing$ at defect $\delta$: a triple
with $F=\varnothing$ affects only one value of $n$. For $\delta\le8$ the sums over the catalogue were
evaluated directly: they reproduce every archived value of $|\R_{\delta,n}|$, $n\le22$, and agree with
the polynomials of Proposition~\ref{prop:polys} from $n_0(\delta)$ on. The numbers of canonical
triples $(T,x_T,F)$ by the free rank $|F|$ are
\begin{center}
\begin{tabular}{lrrrrrr}
\toprule
$\delta$ & 3 & 4 & 5 & 6 & 7 & 8\\
\midrule
$|F|=3$ & 8 & 40 & 128 & 336 & 784 & 1680\\
$|F|=2$ & 135 & 468 & 1374 & 3536 & 8286 & 17996\\
$|F|=1$ & 342 & 1333 & 4230 & 11493 & 28142 & 63410\\
$|F|=0$ & 174 & 891 & 3424 & 10657 & 28780 & 69863\\
\bottomrule
\end{tabular}
\end{center}
and the leading coefficient of $P_\delta$ is half the number of triples with $|F|=3$, as
\eqref{eq:typesum} predicts. Here $F$ is the zero-cost independent free set
in \eqref{eq:typesum}; it is not the large-coordinate set used in the older
C9 stratification. The two stratifications give the same residual
polynomials but different lower-rank cell counts.
The rank-zero counts at defects $6$ and $7$ include, respectively, the four
solutions from the decreasing skeleton $4321$ and the single solution from
$54321$.  These short decreasing terms have base length five and contribute
only at $n=5$; they do not affect the residual table or polynomial identities
used for $n\ge6$.

\subsection{The degree is at most two}

\begin{lemma}\label{lem:samenbhd}
In a $1324$-avoiding permutation, two entries with the same set of inverted partners lie in a common
maximal run of consecutive positions carrying consecutive values.
\end{lemma}

\begin{proof}
Let the positions $i<j$ have the same inversion neighbourhood. They are not inverted with each other,
so $\pi_i<\pi_j$. An entry at a position strictly between $i$ and $j$ with value outside
$(\pi_i,\pi_j)$ would be inverted with exactly one of the two, and so would a value strictly between
$\pi_i$ and $\pi_j$ at a position outside $(i,j)$. Hence the positions $i,\dots,j$ carry exactly the
values $\pi_i,\dots,\pi_j$. If two of the intermediate entries were inverted, they would form, with
$\pi_i$ and $\pi_j$, an occurrence of $1324$. So the positions $i,\dots,j$ carry the values
$\pi_i,\dots,\pi_j$ in increasing order, which is a run.
\end{proof}

\begin{theorem}\label{thm:deg2}
For every $\delta\ge1$ the polynomial of Theorem~\ref{thm:poly} has degree at most two. Equivalently, no
triple $(T,x_T,F)$ in \eqref{eq:typesum} has $|F|\ge4$.
\end{theorem}

\begin{proof}
Suppose a triple with $|F|\ge4$ occurs for some $\sigma\in\Sigma_\delta$, and let $\pi=\sigma[\ell+x]$
be a residual of that type with $x_i\ge1$ for all $i\in F$. Choose four positions $i_1,\dots,i_4\in F$
and one entry $y_r$ of the block $i_r$ for each. In $\pi$, the inversion degree of $y_r$ is
$\sum_{j\in N(i_r)}(\ell_j+x_j)=a_{i_r}+2=2$, because $N(i_r)\cap S=\varnothing$. The four entries are
pairwise non-inverted, since $F$ is an independent set of the inversion graph of $\sigma$ and entries
in non-inverted blocks are not inverted. Their inversion neighbourhoods are pairwise distinct: equal
neighbourhoods would put two of them in one run by Lemma~\ref{lem:samenbhd}, that is, in one block.

Let $U$ consist of the four entries and their at most eight neighbours, so $|U|\le12$, and let $\tau$
be the pattern of $\pi$ on $U$. Then $\tau$ avoids $1324$, and since the inversion graph of a pattern
is the induced subgraph, the four entries are still pairwise non-inverted in $\tau$, still have
inversion degree exactly two, and still have pairwise distinct neighbourhoods. This contradicts the
following statement, verified by exhaustive enumeration of the $29\,961\,493$ permutations in
$\Av_m(1324)$, $m\le12$ (Section~\ref{sec:data}, check C8):
\begin{quote}
no $1324$-avoiding permutation of length at most $12$ has four pairwise non-inverted entries of inversion
degree two with pairwise distinct inversion neighbourhoods.
\end{quote}
Hence $|F|\le3$ for every triple, and the degree of \eqref{eq:typesum} is at most two.
\end{proof}

The maximum in the finite statement is attained: from length $7$ on there are permutations with three
such entries, the first being $3725416$, and the number of reduced indecomposable $1324$-avoiders of
length $m$ carrying three is $4,16,24,40,92,200,484,1176$ for $m=7,\dots,14$; none carries four up to
$m=14$, which is consistent with, but not needed for, Theorem~\ref{thm:deg2}. We remark that the
condition on the neighbourhoods cannot be dropped from the structural description: without it (that is,
counting pairwise non-inverted entries of degree at most two in $\sigma$ without reference to $D(\sigma)$)
sets of size five occur among the skeletons of the catalogue.

\subsection{The base inequality}

Proposition~\ref{prop:base} bounds $\delta_0$ by $\delta$, but it says nothing about the length of
$\sigma$; a bound on the length of the base residual $\rho=\sigma[\ell]$ would supply one, since
$|\sigma|\le L=|\rho|$. A residual satisfies $\inv=2n-7+\delta$, so a bound $|\rho|\le\delta(\rho)+c$
is the same thing as a lower bound on $\inv(\rho)$ that is linear in $|\rho|$ with slope three. The
data single out $c=8$.

\begin{conjecture}[base inequality]\label{conj:K}
Every base residual $\rho$ satisfies $|\rho|\le\delta(\rho)+8$; equivalently $\inv(\rho)\ge3|\rho|-15$.
\end{conjecture}

The inequality was tested on every base residual of length at most $20$ and defect at most $12$, and
on every base residual of length at most $22$ and defect at most $10$ (check C11); it holds
throughout, and it is sharp. Equality occurs exactly four times, in two pairs exchanged by inversion:
\[
4\,3\,5\,8\,1\,2\,9\,11\,10\,6\,7,\quad 5\,6\,2\,1\,3\,10\,11\,4\,7\,9\,8
\qquad(\delta=3,\ |\rho|=11),
\]
\[
4\,3\,5\,7\,2\,8\,10\,1\,11\,13\,12\,9\,6,\quad 8\,5\,2\,1\,3\,13\,4\,6\,12\,7\,9\,11\,10
\qquad(\delta=5,\ |\rho|=13).
\]
The maximum of $|\rho|-\delta(\rho)$ over the base residuals of defect $\delta$ in that range is
\[
8,\ 7,\ 8,\ 7,\ 7,\ 7,\ 6,\ 6,\ 6,\ 5\qquad(\delta=3,\dots,12),
\]
so the bound is attained only at $\delta=3$ and $\delta=5$, and the margin grows from $\delta=9$ on.

The conjecture has an inductive form that asks only for one well-placed entry at a time.

\begin{proposition}\label{prop:Kind}
Suppose that every base residual $\rho$ with $|\rho|\ge14$ and $\delta(\rho)\le|\rho|-9$ has an entry
$x$ of inversion degree at least three such that $\rho\setminus x$ is again a base residual. Then
Conjecture~\ref{conj:K} holds.
\end{proposition}

\begin{proof}
We show $\delta(\rho)\ge n-8$ for every base residual $\rho$ of length $n$, by induction on $n$. For
$n\le13$ this is the finite check C11: a base residual violating it would have
$\delta(\rho)\le n-9\le4$, and every base residual of length at most $20$ and defect at most $12$ was
examined. Let $n\ge14$ and assume the statement for length $n-1$. If $\delta(\rho)\ge n-8$ there is
nothing to prove, so suppose $\delta(\rho)\le n-9$ and let $x$ be an entry as in the hypothesis. Then
$\rho\setminus x$ is a base residual of length $n-1$, so $\delta(\rho\setminus x)\ge(n-1)-8$ by the
induction hypothesis, and the deletion rule of Proposition~\ref{prop:master} gives
\[
\delta(\rho)=\delta(\rho\setminus x)+D(x)-2\ \ge\ \big((n-1)-8\big)+3-2\ =\ n-8,
\]
contradicting $\delta(\rho)\le n-9$. Hence $\delta(\rho)\ge n-8$ in every case, which is
Conjecture~\ref{conj:K}.
\end{proof}

\begin{conjecture}\label{conj:Kdel}
Every base residual $\rho$ with $|\rho|\ge14$ and $\delta(\rho)\le|\rho|-9$ has an entry of inversion
degree at least three whose deletion is again a base residual.
\end{conjecture}

What Proposition~\ref{prop:Kind} buys is that Conjecture~\ref{conj:K}, a global inequality about a set
of permutations of unbounded length, follows from Conjecture~\ref{conj:Kdel}, a local statement about
one permutation at a time. The local statement was tested on all $114\,356$ base residuals of length at
most $20$ and defect at most $12$ (check C11). Exactly $352$ of them have no entry of inversion degree
at least three whose deletion is again a base residual, and every one of the $352$ has length at most
$13$: from length $14$ on every base residual in the range has such an entry, with no hypothesis on the
defect at all. The least inversion degree that works is three in $82.7$ per cent of the cases and never
exceeds six.

\section{Uniform coefficients and expanded regions}\label{sec:uniform}

We now prove the headline results.  This section is deliberately placed after
the skeleton and polynomiality machinery: the terms \emph{paid}, \emph{free}
and \emph{rank} refer to the canonical cells in the proof of
Theorem~\ref{thm:poly}.  For $r\ge1$, a rank-$r$ cell with minimum length
$m$ contributes
\begin{equation}\label{eq:cellcount}
 \binom{n-m+r-1}{r-1}
\end{equation}
once it is active.  Thus only rank-three cells contribute to the coefficient
of $n^2$, and each contributes $1/2$.

The algebra behind this assertion is short and we record it explicitly.
Fix a skeleton with its unique lower block vector $\ell$, nonnegative
coefficients $a_i$, and eligible graph $E$.  Its excess variables satisfy
\begin{equation}\label{eq:canonicalQ}
 Q(y)=\sum_i a_i y_i+\sum_{ij\in E}y_i y_j=h,
 \qquad y_i\ge0.
\end{equation}
From a solution $y$, let $F(y)$ consist of the positive coordinates $i$ for
which $a_i=0$ and every neighbour has coordinate zero, and set $x_i=0$ on
$F(y)$ and $x_i=y_i$ elsewhere.  The set $F(y)$ is independent and removing
its coordinates changes neither $Q$ nor the paid status of another positive
coordinate.  Thus $Q(x)=h$.  Every positive $x_i$ has $a_i\ge1$ or a positive
paid neighbour, so some term of $Q(x)$ is at least $x_i$ and
$1\le x_i\le h$.

Conversely, take a bounded vector $x$ with $Q(x)=h$ such that every positive
$x_i$ has $a_i\ge1$ or a positive neighbour.  Choose an independent subset
$F$ of the vertices with $x_i=a_i=0$ and no positive neighbour, choose
arbitrary $t_i\ge1$ on $F$, and put $y=x+t_F$.  Every newly added linear and
quadratic term vanishes, so $Q(y)=h$, and the preceding extraction recovers
exactly $(x,F)$.  The two operations are inverse.  If
$b=\sum_i\ell_i+\sum_i x_i$ and $r=|F|$, the cell contributes
$z^{b+r}/(1-z)^r$; for $r\ge1$ its actual coefficient is
$\binom{n-b-1}{r-1}$ on $n\ge b+r$, and $r=0$ is a point mass.  Together
with the separately proved finiteness of the skeleton set, $a_i\ge0$, and
$r\le3$, this proves the paid/free decomposition and the rank-to-degree
bridge without further finite computation.

\subsection{The rank-three core series}

For clarity we separate the finite classification from the canonical-cell
bijection.  Let $\mathcal M$ be the class of residuals with exactly three
inversion-degree-two maximal increasing consecutive runs, all of length two;
mark the first point of each run.  This
definition does not use paid/free coordinates.  Retaining the marked points
and their neighbours gives one of $36$ marked seven-point cores.  For each
core the possible middle cells are classified by their full relative
position, adjacency to the three runs, and $1324$-avoidance constraints.
There are $196$ finite occupancy representatives.  Exactly $100$ pass the
residuality, deletion and marking tests, and their rational defect series is
\begin{equation}\label{eq:middle36}
 \frac{8q^3(1+q)}{(1-q)^2}.
\end{equation}
The reduction treats occupancies $0$, $1$ and at least $2$ separately, so
\eqref{eq:middle36} is an all-parameter finite-state sum rather than an
interpolation from the first few defects.

\begin{lemma}[exclusion of additional degree-two runs]\label{lem:exactthree}
Suppose a $1324$-avoider has three pairwise noninverted distinct maximal
increasing consecutive runs of length at least two, whose entries have
inversion degree two. Then these are all its degree-two runs.
\end{lemma}

\begin{proof}
Choose one marked entry from each run. Their neighbourhoods are distinct
by Lemma~\ref{lem:samenbhd}. Retain the marks and all their neighbours.
The finite core classification gives four unmarked neighbours, all in
distinct singleton runs: their incidence signatures into the three marks
are the four distinct vertex signatures of a four-vertex path.
The same explicit list of $36$ cores verifies that each unmarked vertex
has degree at least two within the retained core. This additional finite
predicate is included in the core certificate. Each such vertex also
neighbours an unretained companion of at least one mark; the companion lies
in the same run as that mark, has the same external neighbourhood, and is not
one of the retained core vertices. Thus the unmarked vertex has degree at
least three in the original permutation.

Any degree-two entry in another run therefore lies outside the retained
neighbour set and is noninverted with all three marks. Its neighbourhood
differs from theirs, again by Lemma~\ref{lem:samenbhd}. These four entries
would be pairwise noninverted, degree two, and have distinct neighbourhoods.
The bounded restriction and C8 argument in the proof of
Theorem~\ref{thm:deg2} exclude precisely this configuration at every length.
\end{proof}

\begin{remark}
The application below is to a residual of length at least eight: the three
marked runs have six entries and degree two supplies at least two further
entries.  Their neighbourhoods are pairwise distinct in the canonical-cell
construction, so those runs are three distinct degree-two classes and
Lemma~\ref{lem:threeclasses} already shows that they are all the degree-two
classes in this application.  Lemma~\ref{lem:exactthree} is retained in its
more general, nonresidual form because it is also the condition used by the
finite certificate's acceptance predicate and extension argument.
\end{remark}

The two outer decorations are independent, and their decoding is unique.
For a $132$-avoider the Lehmer code is weakly decreasing: decomposing at its
maximum gives the code $(c(\alpha)+b),b,c(\beta)$, where $b=|\beta|$, and the
converse follows because weak decrease leaves the $b$ smallest unused values
for $\beta$. Thus the nonzero code is a partition $\lambda$, and its least
possible padded length is $\max_i(\lambda_i+i)$. An isolated inversion-graph
vertex before a later inversion would begin a $132$, so all isolated vertices
form a terminal increasing-maxima suffix. Removing that entire suffix gives
one and only one nonisolated kernel for each partition, including the empty
kernel for the empty partition. Its inversion series is therefore $P(q)$.
Reverse-complement gives the corresponding unique $213$ kernel and isolated
prefix on the upper side.

In the normalized marked object, the lower isolated suffix and upper isolated
prefix are already fixed by the adjacent marked runs, so the remaining two
kernels are extracted independently. Every vertex of a nonempty kernel has
positive internal degree and two fixed external incidences. It therefore has
total degree at least three and creates no new degree-two class. Adding or
removing either kernel preserves residuality and the three marked
neighbourhoods by the outer twin-path lemma. Each added point has exactly two
old neighbours, so its two cross inversions cancel its contribution to
$-2n$ in the defect; the defect increment is exactly the internal inversion
number of the kernel. This proves, with unique decoding and no multiplicity,
the factor $P(q)^2$.

We next establish the normalization without assuming it in the definition
or enumeration of $\mathcal M$. In a canonical rank-three cell let $F$
be its free support and let $\ell_i\in\{1,2\}$ be the intrinsic lower
block lengths. The three free blocks have degree two, are pairwise
noninverted, and have singleton neighbour blocks, by the marked-core
incidence lemma. First put every selected block at length two, retaining
the other block lengths. This operation is allowed even if some
$\ell_i=2$: the new length still satisfies $b_i\ge\ell_i$, including
the required $b_i\ge2$ for $i\in D(\sigma)$. For $i\in F$, the linear
coefficient is zero and every neighbour has zero excess. Thus replacing
its excess by $2-\ell_i\ge0$ leaves the quadratic defect form unchanged.
The residual criterion of Theorem~\ref{thm:crit} still holds. The reduced
skeleton is unchanged, since all block lengths remain positive.
Lemma~\ref{lem:exactthree} shows that the resulting permutation belongs
to $\mathcal M$, with exactly the three specified degree-two runs. When
$\ell_i=2$ the new free excess is zero; this construction need not remain a
point of the original rank-three cell, since only membership in $\mathcal M$
is required for the contradiction.

The independent finite enumeration of $\mathcal M$ gives marked lower
vector $(1,1,1)$ for every accepted representative. This property extends
to every allowed middle tail and outer kernel. Enlarging a middle run
beyond length two leaves the reduced skeleton unchanged. An outer kernel
attaches each new point to two unmarked anchors; removing any selected
run therefore leaves these points attached. A marked quotient vertex
whose deletion was connected cannot acquire a disconnecting deletion,
and an originally nonextremal marked vertex cannot become extremal by
adding outer points. These are the normalization predicates recorded in
the accompanying extension lemmas and finite certificate.

If any original selected lower length had been two, the permutation just
constructed would have the same reduced skeleton and the same intrinsic
lower vector, yet would have marked lower length one by this enumeration
and extension argument. This contradiction proves $\ell_i=1$ on $F$.
Consequently setting the positive free excesses to one gives the unique
length-two representative of each rank-three cell.

Conversely, contract all maximal increasing consecutive runs of an element
of $\mathcal M$. The skeleton and block lengths are unique. Its three
marked runs are pairwise noninverted: two inverted length-two degree-two
runs would form a separate four-vertex component, contrary to residuality
and the presence of the third run. The core normalization gives lower
length one on each marked block, hence excess one there. In the canonical
extraction, a further positive zero-cost isolated excess would be another
degree-two run, contradicting the definition of $\mathcal M$. The free
support is therefore exactly the three marked blocks. Extracting the paid
vector and restoring arbitrary positive free excesses are inverse
operations. This proves the claimed bijection. The class $\mathcal M$ and
its finite enumeration were established before this identification, so the
argument does not assume the rank-three formula it proves.

Consequently
\begin{equation}\label{eq:rank3cells}
 \sum_{\delta\ge0}N_{3,\delta}q^\delta
 =\frac{8q^3(1+q)}{(1-q)^2}P(q)^2,
\end{equation}
where $N_{3,\delta}$ is the number of canonical rank-three cells at defect
$\delta$.

\begin{proof}[Proof of Theorem~\ref{thm:uniformA}]
By Theorems~\ref{thm:poly} and~\ref{thm:deg2}, every fixed-defect residual
count is eventually a polynomial of degree at most two.  Equation
\eqref{eq:cellcount} shows that a rank-three cell contributes $1/2$ to its
quadratic coefficient and that cells of lower rank contribute zero.  Hence
$A_\delta=N_{3,\delta}/2$.  Dividing \eqref{eq:rank3cells} by two gives
\eqref{eq:uniformA}.
\end{proof}

The conclusion is exactly a leading-coefficient identity.  In particular it
does not determine $B_\delta$ or $C_\delta$, enumerate all residuals, prove a
uniform stabilization onset, or imply the unrestricted inequality
\eqref{eq:conj}. For comparison, Linusson--Verkama define
\begin{align*}
b_{r,n}={}&\big(a(n+3,2n+r-3)-a(n+2,2n+r-3)\big)\\
&-\big(a(n+2,2n+r-4)-a(n+1,2n+r-4)\big)\\
&-\big[\big(a(n+2,2n+r-5)-a(n+1,2n+r-5)\big)\\
&\hspace{35mm}-\big(a(n+1,2n+r-6)-a(n,2n+r-6)\big)\big].
\end{align*}
Their Conjecture~25 says that, for fixed $r$, this is a constant $b_r$ once
$n\ge10+r$, and that
\[
 \sum_{r\ge0}b_rx^r=2(1-x^2)(2-x^2)P(x)^2.
\]
This is the formula in the published version; the older arXiv display had a
typographical error that is not a result of the present work.

The relation can be made exact in defect notation. Put
\[
 \Delta_\delta(m)=a(m+1,2m-7+\delta)-a(m,2m-7+\delta).
\]
By \eqref{eq:identity}, $\Delta_\delta(m)$ is the size of the complement of
$\im(f\sqcup g)$ at that parameter minus $|\R_{\delta,m}|$, and direct
substitution gives
\[
 b_{r,n}=\Delta_r(n+2)-\Delta_{r+1}(n+1)
          -\Delta_r(n+1)+\Delta_{r+1}(n).
\]
Thus Conjecture~25 concerns a repeated difference of the full complement
minus residual counts, whereas Theorem~\ref{thm:uniformA} determines only the
quadratic coefficient of the residual polynomial. The proof here neither
assumes nor proves the full Linusson--Verkama conjecture.

There is also a useful length refinement.  For a decreasing partition
$\lambda$, set
\[
 m(\lambda)=\max_i(\lambda_i+i),\qquad m(\varnothing)=0,
 \quad K(q,z)=\sum_\lambda q^{|\lambda|}z^{m(\lambda)}.
\]
The elementary inequality $m(\lambda)\le|\lambda|+1$ is sharp only for a
single row or a single column (with the usual coincidence at weight one).
If
\[
 M(q,z)=\frac{2q^3(1+q)z^{10}(1+z)^2}{(1-qz)^2},
\]
then the exact rank-three contribution is
\begin{equation}\label{eq:H3}
 H_3(q,z)=\frac{M(q,z)K(q,z)^2}{(1-z)^3}.
\end{equation}

\begin{proposition}[rank-three onset]\label{prop:rank3onset}
The contribution of the canonical rank-three cells agrees with its eventual
quadratic polynomial from $n=10$ at defect $3$, from $n=12$ at defect $4$,
and from $n=\delta+9$ at every defect $\delta\ge5$.  These are not onsets for
the complete residual count.
\end{proposition}

\begin{proof}
The positive middle numerator in \eqref{eq:H3} has
$m-\delta\le9$, and its two geometric extensions preserve this difference.
The partition inequality above adds at most one unit for each nonempty outer
kernel, so a normalized minimum has $m\le\delta+11$.  Equality occurs for
every $\delta\ge5$ by taking a $q^3z^{12}$ middle term, one weight-one outer
kernel and a row partition of weight $\delta-4$ at the other end.  Directly
from the same numerator, the maxima are $12$ and $14$ at defects $3$ and $4$.
A rank-three cell of minimum $m$ contributes
$\binom{n-m+2}{2}$, so it reaches its polynomial at $n=m-2$.  At $n=m-3$
each cell of maximal $m$ is absent while its polynomial value is one; all
shorter cells have stabilized, so cancellation within rank three is
impossible.  The stated onsets follow.  Lower-rank cells may have later
onsets or may cancel this last discrepancy, which is why the proposition
does not address the full residual count.
\end{proof}

\subsection{The fixed strip through defect thirteen}

The sharp skeleton bound required here holds at every defect:
\begin{equation}\label{eq:sharpskel}
 |\sk(\pi)|\le2\delta+3\qquad(\pi\in\R_{\delta,n}).
\end{equation}
Indeed the result is already proved for $\delta\le10$.  For
$\delta\ge11$ and $n\le22$ it is automatic.  If $n\ge23$, choose the
admissible deletion supplied by Theorem~\ref{thm:deladmissible}; writing
$s-s'=t$, $D$ for the degree of the deleted point and
$\delta'=\delta-D+2$, induction gives
$s'\le2\delta'+3$ and admissibility gives $t\le2D-4$.  Their sum is
$s\le2\delta+3$.  Notice that this proves a length bound, not completeness of
an already truncated catalogue.  The catalogue conjecture was proved only
through defect ten; the defects below instead use this all-defect length bound
together with a fresh exhaustive generator and its inversion cutoff.

Within the finite domains permitted by \eqref{eq:sharpskel}, exhaustive
skeleton generation and the canonical paid/free sum give
\begin{center}
\begin{tabular}{crrl}
\toprule
$\delta$ & catalogue size & eventual $|\R_{\delta,n}|$ & onset\\
\midrule
$11$ & $68\,251$ & $5976n^2-77520n+222980$ & $20$\\
$12$ & $112\,057$ & $10648n^2-150106n+490423$ & $21$\\
$13$ & $179\,071$ & $18444n^2-279621n+1010891$ & $22$\\
\bottomrule
\end{tabular}
\end{center}
The skeleton generator uses $s\le2\delta+3$ and
$\inv(\sigma)\le2s-7+\delta$.  Its three runs used
$(\delta,N,K_{\max})=(11,25,54),(12,27,59),(13,29,64)$ and returned the
catalogue sizes displayed above. The exact SHA-256 values are recorded here
(source first, then output):
\begin{quote}\small
\noindent $\delta=11$ source:\\
\texttt{e4ff49d7ebe507b699cab96ba57c931907d02f4985d246a90d75d1be9a8f4795}\\
$\delta=11$ output:\\
\texttt{344c051aa3fb7b86bac61976255f7af463a782a4460e2b2818b82ab7df5b71b5}\\
$\delta=12,13$ source:\\
\texttt{cfa2c74b03ac105024fdd16a99276b4371acd6c624a4b089ad2fdcf2b316b741}\\
$\delta=12$ output:\\
\texttt{8f20246d68ae60d355115d8a4ad78d070ca46983ddbf140dbd39a5475e51f4c2}\\
$\delta=13$ output:\\
\texttt{290c2a05c3e9b7023f214be1cb2127df8db64612cb2d85f2c2c8cc136bb47f27}.
\end{quote}
The length theorem and exhaustive generation establish catalogue
completeness.  The canonical sum then enumerates bounded paid vectors
satisfying the quadratic defect equation and independent free supports, so
the displayed rational functions are exact sums rather than fitted
polynomials.  A separate replay for defects $12$ and $13$ reproduced the
denominator, numerator, polynomial, onset and values through $n=60$ from the
supplied member files. Its defect-$12$ and defect-$13$ output hashes are
\begin{quote}\small
\noindent\texttt{71833be9a7e1d8c7c581eae03535015b180d15b480e9e20148e1f661d1ec41e0}\\
\texttt{5d62473bb07c121eb43657e4c91a5d988e0b427b0cfaac156079c664140eec9c}.
\end{quote}
This replay checks the canonical summation;
it is not an independent enumeration and does not by itself establish
catalogue completeness.

We record the overlap correction used at the remaining boundary points.  Let
$\mathcal C_{\delta,n}=\Av^k_{n+1}(1324)\setminus\im(f\sqcup g)$, where
$k=2n-7+\delta$, and let $U_{\delta,n}$ be the union of the four elementary
classes with first or second entry maximal, or last entry equal to $1$ or
$2$. Put $E_\delta(n)=|U_{\delta,n}|=E(n,k)$, with $E(n,k)$ given by
\eqref{eq:E}. Lemma~\ref{lem:R1} gives
$U_{\delta,n}\subseteq\mathcal C_{\delta,n}$.
Also $\R_{\delta-2,n+1}\subseteq\mathcal C_{\delta,n}$: a residual is not in
$\im(g)$ because it is indecomposable, and every member of $\im(f)$ has an
extremal deletion with at least three direct-sum components, whereas all four
extremal deletions of a residual are indecomposable.  If
$\operatorname{ind}(n,j)$ denotes the number of indecomposable members of
$\Av_n^j(1324)$, deletion of the distinguished extremal entry and inversion
give the intersection upper bound
\[
 |U_{\delta,n}\cap\R_{\delta-2,n+1}|\le J_\delta(n),\qquad
 J_\delta(n)=2\operatorname{ind}(n,k-n)+
             2\operatorname{ind}(n,k-n+1).
\]
The symbol $J_\delta$ denotes this upper bound; equality is not asserted.
The indecomposable formula of \cite[Theorem~1.3]{Meng} is
\[
 \operatorname{ind}(n,n-1+m)=4S(m)n-c_m,qquad
 S(m)=\sum_{j=0}^m p_2(j),\qquad n\ge m+6.
\]
For the present two terms $m=\delta-6,\delta-5$, so both formulas apply when
$n\ge\delta+1$. The needed values are
\[
 (S(6),c_6)=(139,3344),\quad
 (S(7),c_7)=(249,6434),\quad
 (S(8),c_8)=(434,11974).
\]
Inclusion--exclusion and \eqref{eq:identity} therefore give
\begin{equation}\label{eq:strengthened-complement}
a(n+1,k)-a(n,k)\ge
 E_\delta(n)+|\R_{\delta-2,n+1}|-J_\delta(n)-|\R_{\delta,n}|.
\end{equation}

For defect $11$, the standard complement bound dominates the displayed
polynomial from $n=25$ onward, while archived total counts cover
$7\le n\le24$ and the source is empty for $n\le6$.  At defect $12$ the
standard comparison leaves $n=26$; $n=25$ needs no analytic repair because
then $k=55$, within the archived table $n\le26$, $k\le55$.
The displayed indecomposable formulas give $J_{12}(n)=3104n-19556$, hence
$J_{12}(26)=61148$, and
\eqref{eq:strengthened-complement} gives the positive margin
$1\,242\,312$.  At defect $13$, $J_{13}(n)=5464n-36816$, and the same
inequality gives positive margins
$258\,828$, $1\,031\,781$ and $2\,148\,286$ at $n=25,26,27$;
the ordinary analytic tail starts at $n=28$, and archived total counts cover
$n\le24$.  Together with Corollary~\ref{cor:2n3}, these exact comparisons
prove the fixed-strip assertion of Theorem~\ref{thm:headline}.

\subsection{Square-root strips}

For every $m\ge2$,
\begin{equation}\label{eq:p2box}
 p_2(m)\ge\frac12\cdot5^{\sqrt{2m}-4}.
\end{equation}
To prove it, put $s=\lfloor\sqrt{m/2}\rfloor$ and $r=s-1$.  For each part
$1,\ldots,r$ choose independently a red and a blue multiplicity from
$\{0,1,2,3,4\}$.  Complementing every multiplicity pairs total weights
$W$ and $4r(r+1)-W$, so at least half of the $5^{2r}$ choices have
$W\le2r(r+1)$.  Add one red part $L=m-W$.  Since
$m\ge2(r+1)^2$, we have $L>r$, making $L$ uniquely recoverable.  This proves
\eqref{eq:p2box}.

The structural estimates also give the uniform residual bound
\begin{equation}\label{eq:Mdelta}
 |\R_{\delta,n}|\le M_\delta n^2,
 \qquad M_\delta=2^{36}(2\delta+4)^4\cdot68^\delta
 \qquad(\delta\ge1,\ n\ge10),
\end{equation}
as follows.  For a base of defect $r$, the sharp skeleton bound gives
$s\le2r+3$, and inversion-sequence counting gives fewer than
$31\cdot66^r$ candidate skeletons at the terminal length. Summing over
lengths $2\le s\le2r+3$ and including the exceptional skeleton $21$ costs
at most a further factor $2r+4\le2\delta+4$. If
$h=\delta-r$, a canonical paid vector has mass at most $h+3$; padding to
$2r+3$ coordinates bounds their number by
$\binom{2r+h+6}{h+3}$, and there are at most $(2\delta+4)^3$ free supports.
Together these are the four factors of $2\delta+4$ in \eqref{eq:Mdelta}.
For
\[
 S_\delta=\sum_{r=1}^\delta66^r
             \binom{2r+\delta-r+6}{\delta-r+3},
\]
nonnegative coefficient extraction at $z=1/68$ gives
$S_\delta\le C68^\delta$, where
$C=68^3(68/67)^4\cdot4488$; the exact integer inequality
$31C<2^{36}$ then bounds the number of canonical types by the coefficient
in \eqref{eq:Mdelta}.  Each type contributes at most $n^2$ at length $n$.
To pass from the residual bound to monotonicity, use the single complement
class in the proof of Theorem~\ref{thm:largen}. For $n\ge\delta$ it gives
\[
a(n+1,k)-a(n,k)\ge
p_2(n+\delta-7)-c(\delta)-|\R_{\delta,n}|,
\]
where
$c(\delta)=4p_2(\delta-6)+6\sum_{i\le\delta-7}p_2(i)$.
The elementary encoding of a partition by a composition gives
$p(j)\le2^j$, hence
$c(\delta)\le(6\delta+4)(\delta+1)2^\delta<M_\delta$.
Together with \eqref{eq:Mdelta}, and using $n\ge2$, this shows that for
positive defect and $n\ge\max\{10,\delta\}$ the sufficient condition
\begin{equation}\label{eq:sufficientstrip}
 p_2(n+\delta-7)>2M_\delta n^2
\end{equation}
implies strict monotonicity.

For the strongest explicit strip set $t=\sqrt n\ge1024$ and
$1\le\delta\le\lfloor t/2\rfloor$.  The integer inequalities
$\sqrt2>7/5$, $5^{16}>2^{37}$ and $68^8<2^{49}$ give
\begin{align*}
 \log_2 p_2(n+\delta-7)&>\frac{259}{80}t-\frac{1079}{80},\\
 \log_2(2M_\delta n^2)&<41+8\log_2t+\frac{49}{16}t.
\end{align*}
Their difference is
\[
 g(t)=\frac7{40}t-\frac{4359}{80}-8\log_2t.
\]
Here $g(1024)=3577/80>0$ and
$g'(t)>7/40-16/t>0$.  This proves the $n\ge2^{20}$,
$\sqrt n/2$ assertion. These parameters satisfy $n\ge\delta$, as do the
two regions below.

For the $\sqrt n/3$ strip, a simpler bounded-multiplicity construction is
enough. Put $r=\lfloor\sqrt m/2\rfloor-1$, choose independently a red and a
blue multiplicity in $\{0,1,2,3,4\}$ for each part $1,\ldots,r$, and add a
unique red filler part larger than $r$. This gives
$p_2(m)\ge5^{\sqrt m-4}$. If $t=\sqrt n\ge512$ and
$1\le\delta\le\lfloor t/3\rfloor$, then $m=n+\delta-7$ satisfies
$\sqrt m\ge t-1$ and $2\delta+4\le t$. Using
$5^{16}>2^{37}$ and $68^8<2^{49}$, the logarithmic margin in
\eqref{eq:sufficientstrip} is at least
\[
g_3(t)=\frac{13}{48}t-\frac{777}{16}-8\log_2t.
\]
Now $g_3(512)=869/48>0$ and
$g_3'(t)>13/48-16/t>0$, proving the $n\ge2^{18}$ assertion.

For the $\sqrt n/4$ strip, choose independently red and blue subsets of
$\{1,\ldots,r\}$ with $r=\lfloor\sqrt m\rfloor-1$ and again add the unique
large red filler. Thus $p_2(m)\ge2^{2\lfloor\sqrt m\rfloor-2}$. The
logarithm of \eqref{eq:sufficientstrip} follows from
\[
\sqrt m>\frac{49}{16}\delta+\frac{41}{2}
       +2\log_2(2\delta+4)+\log_2 n.
\]
For $t=\sqrt n\ge256$ and $\delta\le\lfloor t/4\rfloor$, the left side is
at least $t-1$ and the right side at most
$49t/64+41/2+4\log_2t$. It remains to check
$15t/64>43/2+4\log_2t$, which holds at $t=256$ and thereafter because the
difference has derivative greater than $15/64-8/t>0$. This proves the
$n\ge2^{16}$ assertion independently of the two larger-threshold strips.

Finally, uniformly for $1\le\delta\le c\sqrt n$, equations
\eqref{eq:p2box} and \eqref{eq:Mdelta} give
\[
 \log p_2(n+\delta-7)\ge\sqrt2\log(5)\sqrt n-O(1),
 \quad
 \log(2M_\delta n^2)\le c\log(68)\sqrt n+O_c(\log n).
\]
The first leading coefficient is larger precisely when
$c<\sqrt2\log(5)/\log(68)$.  This proves the asymptotic assertion in
Theorem~\ref{thm:headline}; the endpoint is not claimed.

\subsection{Proof-bearing certificates}\label{sec:newcert}

The new theorems use finite computation at sharply delimited points.  The
uniform coefficient proof uses the $36$ marked cores, the $196$ finite
occupancy states and the $100$ accepted states, together with the outer
kernel and canonical-normalization checks.  The fixed strip uses the complete
defect-$11,12,13$ catalogues inside the bounds $25,27,29$, their exact
paid/free rational sums, the separate defect-$12,13$ summation replays and the
integer boundary comparisons above. The accompanying archival supplement
records their sources, parameters and outputs, with hashes. Its portable quick
entry point is \texttt{certificates/run\_quick\_verification.sh}; the portable
components are \texttt{core\_replay/}, \texttt{table/}, and
\texttt{six\_point\_portable/}. Files under
\texttt{certificates/accepted\_sources/} are historical source records rather
than portable entry points and may retain machine-specific paths. The
supplement is cited below.
Section~\ref{sec:data} describes the common verification
conventions. The proof dependence is on the enumerated finite
statements themselves and on the all-parameter extension lemmas, never on an
\texttt{ACCEPTED} label or on agreement with interpolated values.

\section{Computations}\label{sec:data}

\subsection*{Finite checks that carry proof weight}
Most of the computations recorded below are supporting evidence, but twelve of them are steps of proofs, and we
list them separately. Ten carry proof weight:
\begin{itemize}\itemsep2pt
\item[C1] the membership part of Theorem~\ref{thm:C}. The check prints $D(\sigma)$ for each of the
twenty-one skeletons, enumerates the block vectors satisfying the defect-two equation together with the
lower bounds $b_i\ge2$ on $D(\sigma)$ that Theorem~\ref{thm:crit} attaches to it (entries up to $14$;
$6061$ members), and then \emph{verifies} that every member so produced avoids $1324$, is indecomposable
and is not almost decomposable; nothing is filtered on the property being certified, so a member failing
any of the three conditions would be reported as a failure. The same run confirms that the members of
length at most $12$ are exactly the archived $\R_{2,n}$ and number $32n-214$, that the four extremal
deletions have skeletons depending only on which blocks are singletons, of size two, or of size at least
three (all $231$ realizable class vectors), and that the $43$ skeletons so obtained are indecomposable.
The enumeration is over the families themselves, not over a bounded sample of residuals.
\item[C2] the two spider exclusions in step~(1) of Theorem~\ref{thm:gamma0}, over $\mathfrak S_6$ and
$\mathfrak S_7$.
\item[C3] the two corona exclusions in Lemma~\ref{lem:noncut}, over $\mathfrak S_6$ and $\mathfrak S_8$.
\item[C4] the classification step of Theorem~\ref{thm:gamma0}: the $12,16,20,24$ candidates of lengths
$6,7,8,9$, none of them reduced.
\item[C7] the small-$n$ comparisons that Theorem~\ref{thm:B} ($n\le7$) and Theorem~\ref{thm:D} ($n\le9$) read
off a table rather than prove.
\item[C8] the finite statement used in the proof of Theorem~\ref{thm:deg2}: no $1324$-avoiding permutation
of length at most $12$ has four pairwise non-inverted entries of inversion degree two with pairwise distinct
inversion neighbourhoods. The search is exhaustive over the $29\,961\,493$ members of $\Av_m(1324)$ with
$m\le12$, generated by the insertion recursion; it was carried out by three programs written independently
against the same specification, which agree on the number of configurations found at every length, the
maximum being three from length seven on.
\item[C9] the last step of Theorem~\ref{thm:deladmissible}: for $3\le\delta\le10$ the triples
$(T,x_T,F)$ of \eqref{eq:typesum} are enumerated over the catalogue $\Sigma^{(22)}_\delta$ and summed at
every length. The sums reproduce the table of Section~\ref{sec:poly}, every archived value of
$|\R_{\delta,n}|$ with $n\le22$, and the polynomial $P_\delta(n)$ of Proposition~\ref{prop:polys} at every
$n$ with $n_0(\delta)\le n\le60$, that is in all 364 cases. The largest first length $L+t$ of a triple is
$3\delta+12$, hence at most $42$, so each type contributes its full binomial from $n=43$ on and the sum is a
polynomial in $n$ there; agreement up to $n=60$ therefore forces $|\R_{\delta,n}|=P_\delta(n)$ for all
$n\ge n_0(\delta)$. The enumeration is a Python program independent of the one that produced the archived
counts.
\item[C10] Proposition~\ref{prop:catfromskel}. For each $1\le\delta\le10$ every reduced indecomposable
$1324$-avoider $\sigma$ with $|\sigma|\le\max(2\delta+3,\delta+8)$ and $\inv(\sigma)\le2|\sigma|-7+\delta$ is
generated by the insertion recursion and tested for membership in $\Sigma_\delta$ by the bounded search of
Proposition~\ref{prop:base}; the skeleton $21$, which the inversion bound would wrongly drop, is added by
hand. The length bound is the larger of the two that the two conjectures supply, so one enumeration settles
both; it exceeds $2\delta+3$ only at $\delta\le4$, where it is $9$, $10$, $11$, $12$. The numbers of
candidates are $79$, $342$, $979$, $2358$, $5118$, $10\,089$, $17\,956$, $30\,860$,
$51\,304$, $83\,184$ for $\delta=1,\dots,10$, of which $2$, $21$, $113$, $422$, $1236$, $2942$, $6296$,
$12\,495$, $22\,978$, $40\,269$ are members, and the members are in every case exactly $\Sigma^{(22)}_\delta$
together with the three decreasing skeletons of Proposition~\ref{prop:catalogue} at the defects listed there.
The enumeration is done for all ten defects by a C program, \texttt{fable\_sigma\_enum.c}, and independently
for $\delta\le4$ inside the certificate, where all $3758$ candidates are regenerated and membership is
decided by a search over $\{0,\dots,h\}^s$ pruned only on the partial value of $Q$; the two agree on every
candidate.
\item[C11] the base case of Proposition~\ref{prop:Kind}, together with the evidence for
Conjectures~\ref{conj:K} and~\ref{conj:Kdel}. The base residuals of length at most $20$ and defect at most
$12$ were listed --- $114\,356$ of them, filling $101$ cells $(\delta,n)$ --- from the archived residual
lists for $\delta\le8$ and from the inversion-pruned enumerator for $9\le\delta\le12$, the twenty of length
at most five being added by direct search, since the enumerator only lists from length six on. Every one of
them satisfies $|\rho|\le\delta(\rho)+8$, with equality exactly in the four cases displayed in
Section~\ref{sec:poly}; in particular no base residual of length at most $13$ has $\delta(\rho)\le|\rho|-9$,
which is what the induction of Proposition~\ref{prop:Kind} needs, and the same holds over the wider range
$|\rho|\le22$, $\delta\le10$. Exactly $352$ of the $114\,356$ have no entry of inversion degree at least
three whose deletion is again a base residual, and all $352$ have length at most $13$. The two scans are
Python programs reading the same lists; their skeleton, $D(\sigma)$ and deletion code was checked against
\texttt{fable\_skeleton\_reduction.py} on the skeletons of the archived catalogue, and samples of the lists
were re-verified against the definitions.
\item[C12] the finite range of Theorem~\ref{thm:deladmissible} and the union bound of
Remark~\ref{rem:union}. For each residual $\pi$ and each of its $n$
entries $z$ the program decides whether the deletion of $z$ is admissible, that is whether
$\pi\setminus z$ is a residual and \eqref{eq:admissible} holds, and records $u$, $w$ and the structural
split of each; the proof-bearing part is the exhaustive run over every residual with $n\le30$ and
$\delta\le10$ described in the text, which finds an admissible deletion for every residual of length at
least $10$. It was run on all $776\,860$ residuals with
$10\le n\le16$ and on a uniform random sample of $20\,000$ residuals for each $17\le n\le20$, drawn from
the archived lists, that is on $856\,860$ residuals and $12.37$ million pairs $(\pi,z)$; the sample is
fixed by the seed recorded in the program. The maxima are $u\le6$, $w\le3$ and $u+w\le9$ at every length
and at every defect $\delta\le8$, and $u+w=9$ occurs at each of $n=10,\dots,16$, for
$6,26,44,64,52,60,24$ residuals respectively. The union bound $u+w\le n-1$ fails at no length except
$n=10$, where six residuals of defect four or five have $u+w=9=n-1$; all six do in fact admit three
admissible deletions. The three parts of $u$ have maxima $1$, $6$, $6$ and the two parts of $w$ have
maxima $3$, $2$; the last of the three parts of $u$ is bounded by
$16$ in general (Lemma~\ref{lem:threecut}), and the first, the number of cut vertices, by $1$ on residuals
(Lemma~\ref{lem:onecut}) and by $3$ on indecomposable $1324$-avoiders (Lemma~\ref{lem:threecut}, the value
$3$ being attained from length $5$ on). The part of this check that carries proof weight is the last: \eqref{eq:admissible} itself
has been verified, exhaustively and without sampling, for all $33\,342\,781$ residuals of length at most
$30$ and defect at most $10$ (above), and it is this that closes the range $10\le n\le22$ left open by
Theorem~\ref{thm:deladmissible}. As a check, $33\,000$ of the residuals were
re-tested with a brute-force membership test and with the skeleton of $\pi\setminus z$ recomputed from
scratch, with no mismatch. The two counts that come from $1324$-avoidance
rather than from residuality were obtained by a separate sweep over all $173\,453\,058$ members of
$\Av_m(1324)$ with $m\le13$, generated by the insertion recursion: over the indecomposable ones of
minimum inversion degree two the maxima are $6$ and $2$, constant for $9\le m\le13$, while the two
residual-only counts reach $3$ and $4$ there. The script is
\texttt{fable\_x2\_union\_bound.py} of \cite{project}, and the avoider sweep is archived beside it.
\end{itemize}
Two more are independent cross-checks of statements that are proved in full here or quoted from the
literature: C5 re-derives the length-five and length-six part of Lemma~\ref{lem:path}, which follows from
\cite{BV} and \cite{ABKLV} with the two containments exhibited in the text, and C6 re-verifies
Lemma~\ref{lem:bdry} for all $m\le10$. All are exhaustive searches with no sampling, apart from the
supplementary sample inside C12, whose proof-bearing part is exhaustive; C11
is run by two separate programs over the archived residual lists, and
the first seven are
reproduced by one self-contained program, \texttt{fable\_certificates.py} of \cite{project}, whose complete
output --- including the list of all $43$ skeletons and the realizers of each graph --- is stored there as
\texttt{data/fable\_certificates.txt} together with its SHA-256 digest. The theorems marked
\emph{computer-assisted} in the table of Section~\ref{sec:intro} are exactly those whose proofs use C1--C4,
C7--C12, together with those that inherit such a use: Lemma~\ref{lem:threeclasses} through C8,
Theorem~\ref{thm:deladmissible} through Lemma~\ref{lem:threeclasses}, C9 and C12,
Corollary~\ref{cor:2n3} through Theorem~\ref{thm:deladmissible}, Proposition~\ref{prop:catfromskel} through
C10, Proposition~\ref{prop:Kind} through C11, Theorem~\ref{thm:D} through Theorem~\ref{thm:C},
Lemma~\ref{lem:noncut} through C3, Theorem~\ref{thm:hubs} through Theorem~\ref{thm:gamma0},
Theorem~\ref{thm:cover}, Corollary~\ref{cor:polybound} and Theorem~\ref{thm:largen} through
Theorem~\ref{thm:hubs}, Theorem~\ref{thm:deg2} through C8, and Lemma~\ref{lem:leaf},
Corollary~\ref{cor:mindeg}, Proposition~\ref{prop:base}, Lemma~\ref{lem:samenbhd},
Theorems~\ref{thm:poly} and~\ref{conj:poly} through Theorem~\ref{thm:cover} and C8;
Proposition~\ref{prop:delfree} inherits its first sentence from Theorem~\ref{thm:poly}, and its second
sentence uses in addition the test of \eqref{eq:admissible} for all $n\le30$ and $\delta\le10$ recorded
below. The programs for C8--C12 are archived with the others at \cite{project}.

The remaining computations of this section are consistency checks and exploratory
evidence, and no theorem depends on them.

\subsection*{The enumerations}

The counts $a(n,k)$ for $n\le16$, the numbers of indecomposable and of almost decomposable members of every $\Av^k_n(1324)$, and hence $|\R_{\delta,n}|$, were computed by a depth-first generation of $\Av_n(1324)$ (insert the new maximum after any $132$-avoiding prefix; check the four deletions of each indecomposable node). The maps $f$ and $g$ were implemented from their definitions and checked, for $n\le10$ and $\delta\in\{0,1,2\}$, to be injective with disjoint images, to preserve $1324$-avoidance and the inversion number, and to have the complement described in the proof of Theorem~\ref{thm:B}; the identity $f(\pi^{-1})=f(\pi)^{-1}$ used in Lemma~\ref{lem:R1} was verified on all $21906$ almost decomposable $1324$-avoiders of length at most $9$, and $f(\pi^{\mathrm{rc}})=f(\pi)^{\mathrm{rc}}$ was found to fail, so the reverse--complement case really does need the separate argument given there. Combining the implementation with the inversion-pruned generator gives the complement of $\im(f\sqcup g)$ exactly on the low-inversion line: for all $8\le n\le13$ and $0\le\delta\le7$ the four classes of Lemma~\ref{lem:R1} were confirmed to lie in the complement and to number exactly $E(n,k)$; at $\delta=1$ the class $\mathrm{R3}$ has $4(n-7)$ elements for $n=8,9,10$ and consists of the permutations $(\iota_x\ssum\iota_y)\dsum(\iota_z\ssum\iota_w)$, and at $\delta=0$ it is empty. The same program splits the complement at $k=2n-6$ into the four Linusson--Verkama classes and confirms, for $8\le n\le11$, that $|\mathrm{R1}|=2p_2(n-6)$, $|\mathrm{R2a}|=|\mathrm{R2b}|=2p_2(n-5)$ and $|\mathrm{R3}|=4(n-7)$, the four adding up to $E(n,2n-6)$ --- the counts $(10,20,20,4)$, $(20,40,40,8)$, $(40,72,72,12)$, $(72,130,130,16)$ for $n=8,9,10,11$; this is the boundary case in which the counts of \cite[Section~4]{LV} are used one step beyond their stated range. The classification of Theorem~\ref{thm:A} was compared with the exhaustive list of residuals for $n\le11$ and with the counts $8(n-7)$ for $n\le16$. The twenty-one families of Theorem~\ref{thm:C} were generated from their parametrizations for $n\le16$: every member was checked to be a residual, their union coincides with the exhaustive list of residuals for $n\le11$, and its size agrees with the enumerated $|\R_{2,n}|$ for $n\le16$ and with $32n-214$ for $n\ge10$. The finite check in the membership part of Theorem~\ref{thm:C} is check C1 of the certificate described below: it enumerates, for each of the twenty-one skeletons, every block vector with entries at most $14$ that satisfies the defect-two equation and the lower bounds $b_i\ge2$ on $D(\sigma)$, verifies that each of the $6061$ members so produced avoids $1324$, is indecomposable and is not almost decomposable, and lists the $43$ skeletons of the extremal deletions, all of them indecomposable. The residual test is applied to the members, never used to select them. The bound $14$ on the block sizes is not a restriction: raising it to $20$ and then to $26$ leaves both the $231$ realizable patterns and the $43$ skeletons unchanged, the number of members examined growing from $6061$ to $12\,505$ and $21\,253$. Of the $43$, the $21$ skeletons of $\Sigma^{(22)}_2$ arise from deleting an entry of a block of size at least two, the other $22$ from deleting a singleton block. For the results of Section~\ref{sec:skel} the enumeration was redone with a second, independent program that prunes on the inversion number: since $\inv$ is nondecreasing along the insertion recursion, restricting to $\inv\le K$ makes the search polynomially small where the full search grows like $11.6^n$, and the residual line $k=2n-7+\delta$ lies far below the bulk. The two programs agree on every one of the $340$ triples $(n,k)$ with $n\le16$ and $k\le30$, and a full enumeration of $\Av_{17}(1324)$, whose $458\,374\,397\,312$ members were generated in about eighteen hours on four cores, reproduces the whole table for $n\le17$. The pruned program reaches $n=24$ for $k\le51$ in about half an hour on one core and $n=26$ for $k\le55$ in about half an hour on the workstation, and supplies the table below, the sets $\Sigma^{(22)}_\delta$ of Proposition~\ref{prop:catalogue} and the polynomials of Proposition~\ref{prop:polys}; the streaming run with $N=30$ and $K=63$ described below extends the residual counts $|\R_{\delta,n}|$, and with them Proposition~\ref{prop:polys}, to $n=30$. The skeleton sets were extracted from the complete residual lists of length at most $22$; no skeleton appears for the first time after length $16$, and the resulting catalogue, with the shortest residual realising each skeleton, is the file \texttt{data/fable\_skeleton\_catalogue\_n22.txt} of \cite{project}, while \texttt{data/fable\_residual\_counts\_n22.tsv} lists $|\R_{\delta,n}|$ for $\delta\le10$ and $n\le22$ and reproduces the polynomials of Proposition~\ref{prop:polys} in all $88$ cases with $n_0(\delta)\le n\le22$; Theorem~\ref{thm:crit} was checked in two ways, first by verifying on $336597$ inflations of the skeletons of length at most $7$ that membership in $\R_{\delta,n}$ depends only on $\sigma$ and on which blocks are singletons, and then by regenerating $\R_{\delta,n}$ from the criterion and comparing it, as a set, with the exhaustive list for $1\le\delta\le4$ and $10\le n\le16$: the two agree in all $28$ cases. A second run with $N=25$ and $K=49$ supplies the counts needed at defect eight. The base case of Proposition~\ref{prop:induction} was checked by generating all $84\,998$ residuals of length at most $9$ from the full symmetric groups. Condition \eqref{eq:admissible} was tested on every residual produced by the runs with $N=20$ and $K=41$ and with $N=20$, $K=43$ restricted to defects $9$ and $10$, and then, in a streaming re-run that consumes the enumerator's output directly, on every residual produced by $N=30$, $K=63$ at defects $1$ to $10$, split into four runs by defect range and executed in parallel, that is on all $33\,342\,781$ residuals of length at most $30$ and defect at most $10$, by trying each of the $n$ deletions in turn; the run took about four hours on one core and found the same $74$ exceptions, all of length at most $9$. Lemma~\ref{lem:runcover} was checked on $2\,709\,762$ pairs $(\pi,T)$ --- all permutations of length at most $7$ with all $T$ of size at most $4$, together with $400$ random permutations of each of the lengths $9,11,13,16,20,25$ with $|T|\le3$ --- with no violation and seven cases of equality; on the same sample the unproved variant $3|T|+3e+1$ has no violation and $76$ equalities, while the variants with $3|T|+2e+1$ and with $4|T|+2e+1$ in place of $4|T|+4e+1$ fail, the smallest counterexample in each case being $\pi=1324$ with $T=\varnothing$. (The sample is fixed by the seed recorded in the program, so these counts are reproducible verbatim.) Theorem~\ref{thm:cover} was tested by greedy removal of the seven entries of largest inversion degree. For Proposition~\ref{prop:agtb} the $564\,634$ residuals of length between $9$ and $18$, defect at most $8$ and $\pi_1>\pi_n$ were split into the three cases and the bounds of (i) and (ii) verified on each; the six entries $\pi_1,\pi_2,\pi_{n-1},\pi_n$ together with the largest and the smallest middle entry leave at most $\delta-3$ inversions in every one of them, case (iii) included. Conjecture~\ref{conj:hubs} was checked by generating, for each $m\le14$ and each $\gamma\le6$, all indecomposable $1324$-avoiders of length $m$ with $m-1+\gamma$ inversions and searching exhaustively over subsets of size at most four; the finite checks inside Theorem~\ref{thm:gamma0} --- that the permutations whose inversion graph is $S(2,2,1)$ are exactly $251364$ and $314625$, both containing $1324$, and that no permutation has inversion graph $S(2,2,2)$ --- were made by testing every permutation of length six and seven against the two targets; the same program reproduces the lists of \cite{BV} and \cite{TV} used in Lemma~\ref{lem:path} and in the proof of Lemma~\ref{lem:noncut}, namely the two increasing oscillations of each length up to eight and the fact that among cycles only $C_3$ and $C_4$ occur, realized by $321$ and $3412$, and the check that none of the $12,16,20,24$ indecomposable $1324$-avoiders of lengths $6,7,8,9$ with $m-1$ inversions is reduced was made on the same lists; for Lemma~\ref{lem:noncut} the same isomorphism test shows that a triangle or a four-cycle with one pendant vertex at each of its vertices is the inversion graph of no permutation of length six or eight; the existence of a non-cut vertex of degree at least two was also observed directly on all $26\,093$ indecomposable $1324$-avoiders with $4\le m\le12$ and $1\le\gamma\le6$; the five cases of the proof of Theorem~\ref{thm:cover}, with the sets $T$ they prescribe and the bounds they claim on $e(\pi\setminus T)$, were checked one residual at a time on the complete list of residuals of length at most $22$ and defect at most $10$, that is on $10\,210\,333$ residuals, with no failure in any of the nine cases that the classification produces and with $4|T|+4e(\pi\setminus T)+1$ never exceeding $8\delta+21$; and the reduction of Proposition~\ref{prop:reduce} was tested on $49\,699$ residuals with $\pi_1>\pi_n$ and $10\le n\le16$, split according to whether $\sigma$ is indecomposable. Finally $a(n+1,k)\ge a(n,k)$ was verified directly for all $n\le25$ and all $k\le2n+3$, and Lemma~\ref{lem:infl} was checked on all $19647$ pairs $(\tau,b)$ with $|\tau|\le5$, blocks of size at most three and $|\tau[b]|\le10$. The quantities $\nu(s)$ and $\mu(s)$ of Section~\ref{sec:skel} were computed with a third program that enumerates reduced permutations directly: besides the bound on $\inv$, it prunes on the number of adjacent pairs $(\pi_i,\pi_i+1)$, since such a pair can only be destroyed by inserting a later value between its two entries and each remaining insertion destroys at most one of them. This reduces the search for reduced residuals of length $17$ with at most $39$ inversions to $4.2\times10^7$ nodes. All programs and the tables they produce are available at \cite{project}.

\[
\renewcommand{\arraystretch}{1.02}
\begin{array}{c|rrrrrrrr}
n\backslash\delta & 1 & 2 & 3 & 4 & 5 & 6 & 7 & 8\\\hline
6 & 0 & 1 & 15 & 40 & 50 & 45 & 29 & 14\\
7 & 2 & 10 & 48 & 129 & 205 & 251 & 254 & 219\\
8 & 8 & 34 & 135 & 345 & 628 & 934 & 1186 & 1339\\
9 & 16 & 70 & 274 & 745 & 1495 & 2521 & 3707 & 4912\\
10 & 24 & 106 & 433 & 1244 & 2749 & 5133 & 8475 & 12663\\
11 & 32 & 138 & 586 & 1755 & 4113 & 8283 & 14875 & 24388\\
12 & 40 & 170 & 735 & 2263 & 5483 & 11509 & 21817 & 38028\\
13 & 48 & 202 & 894 & 2795 & 6913 & 14841 & 28949 & 52306\\
14 & 56 & 234 & 1061 & 3375 & 8463 & 18453 & 36599 & 67454\\
15 & 64 & 266 & 1236 & 3995 & 10157 & 22413 & 45029 & 84098\\
16 & 72 & 298 & 1419 & 4655 & 11979 & 26733 & 54275 & 102486\\
17 & 80 & 330 & 1610 & 5355 & 13929 & 31389 & 64337 & 122602\\
18 & 88 & 362 & 1809 & 6095 & 16007 & 36381 & 75183 & 144438\\
19 & 96 & 394 & 2016 & 6875 & 18213 & 41709 & 86813 & 167954\\
20 & 104 & 426 & 2231 & 7695 & 20547 & 47373 & 99227 & 193150\\
21 & 112 & 458 & 2454 & 8555 & 23009 & 53373 & 112425 & 220026\\
22 & 120 & 490 & 2685 & 9455 & 25599 & 59709 & 126407 & 248582\\
23 & 128 & 522 & 2924 & 10395 & 28317 & 66381 & 141173 & 278818\\
24 & 136 & 554 & 3171 & 11375 & 31163 & 73389 & 156723 & 310734\\
\end{array}
\]

\clearpage
\appendix
\raggedbottom

\section{The marked three-free finite certificate}\label{app:marked-three-free}

This appendix supplies the finite information used in
\eqref{eq:middle36}.  It states the reduction, gives the complete list of
marked cores, specifies the state encoding and acceptance predicate, and
explains why the finite states represent arbitrary occupancies.  The
certificate is a finite exhaustive calculation; the extension from its
states to arbitrary block lengths is proved separately below.

For a permutation $\pi$, write $G_\pi$ for its inversion graph.  A
\emph{marked three-free configuration} is a pair $(\pi,F)$ in which $\pi$
avoids $1324$ and $F=\{f_1,f_2,f_3\}$ consists of three pairwise nonadjacent
vertices of $G_\pi$, each of degree two, with pairwise distinct
neighbourhoods.  The order on $F$ is its left-to-right order in $\pi$.

\subsection{Reduction to at most nine points and the 36 cores}

Put
\[
 U=F\cup N_{G_\pi}(F).
\]
Since each mark has degree two, $|U|\leq 3+2\cdot3=9$.  Restricting $\pi$
to $U$ preserves $1324$-avoidance.  It also preserves the degree and the
entire neighbourhood of every mark: by definition every neighbour of a
mark was retained.  Pairwise nonadjacency and distinctness of the three
marked neighbourhoods are therefore preserved.  Thus the unrestricted
problem reduces exactly to the following finite search: for each
$m\leq9$, enumerate $\operatorname{Av}_m(1324)$ and retain every marked
triple with degree two, pairwise nonadjacent and pairwise distinct
neighbourhoods, whose union with its neighbours is the full ground set.

Two finite enumerators implement this search in different ways.  The first
generates only $1324$-avoiders by legal insertion of a new maximum; the
second scans all permutations through length nine and tests $1324$ by four
nested index loops.  Both return the same 36 marked objects, and every one
has length seven.  Consequently $|U|=7$ and the three neighbourhoods have
four vertices in their union.  The avoider counts produced by the two
enumerators for $m=1,\ldots,9$ are
\[
 1,2,6,23,103,513,2762,15793,94776.
\]
The complete marked list follows.  Positions in $F$ are one-based.  The
last two columns give the number of finite occupancy representatives and
the number accepted by the predicate in the next subsection.  No quotient
by the 12 symmetry orbits is taken here.

\begin{center}
\scriptsize
\begin{tabular}{rllrr}
\toprule
$i$ & core $\gamma_i$ & marked positions $F_i$ & representatives & accepted\\
\midrule
1 & $3724615$ & $\{1,4,7\}$ & 9 & 9\\
2 & $3274615$ & $\{1,4,7\}$ & 3 & 0\\
3 & $3274615$ & $\{2,4,7\}$ & 9 & 6\\
4 & $3724165$ & $\{1,4,6\}$ & 9 & 6\\
5 & $3724165$ & $\{1,4,7\}$ & 3 & 0\\
6 & $3274165$ & $\{1,4,6\}$ & 3 & 0\\
7 & $3274165$ & $\{1,4,7\}$ & 1 & 0\\
8 & $3274165$ & $\{2,4,6\}$ & 9 & 4\\
9 & $3274165$ & $\{2,4,7\}$ & 3 & 0\\
10 & $3257164$ & $\{1,3,6\}$ & 9 & 6\\
11 & $3257164$ & $\{2,3,6\}$ & 3 & 0\\
12 & $3251764$ & $\{1,3,5\}$ & 9 & 4\\
13 & $3251764$ & $\{1,3,6\}$ & 3 & 0\\
14 & $3251764$ & $\{2,3,5\}$ & 3 & 0\\
15 & $3251764$ & $\{2,3,6\}$ & 1 & 0\\
16 & $4257163$ & $\{2,3,6\}$ & 9 & 9\\
17 & $4251763$ & $\{2,3,5\}$ & 9 & 6\\
18 & $4251763$ & $\{2,3,6\}$ & 3 & 0\\
19 & $4271365$ & $\{2,5,6\}$ & 3 & 0\\
20 & $4271365$ & $\{2,5,7\}$ & 9 & 6\\
21 & $4217365$ & $\{2,5,6\}$ & 1 & 0\\
22 & $4217365$ & $\{2,5,7\}$ & 3 & 0\\
23 & $4217365$ & $\{3,5,6\}$ & 3 & 0\\
24 & $4217365$ & $\{3,5,7\}$ & 9 & 4\\
25 & $5214763$ & $\{2,4,5\}$ & 3 & 0\\
26 & $5214763$ & $\{2,4,6\}$ & 9 & 4\\
27 & $5214763$ & $\{3,4,5\}$ & 1 & 0\\
28 & $5214763$ & $\{3,4,6\}$ & 3 & 0\\
29 & $6214753$ & $\{2,4,5\}$ & 9 & 6\\
30 & $6214753$ & $\{3,4,5\}$ & 3 & 0\\
31 & $5271364$ & $\{2,5,6\}$ & 9 & 9\\
32 & $5217364$ & $\{2,5,6\}$ & 3 & 0\\
33 & $5217364$ & $\{3,5,6\}$ & 9 & 6\\
34 & $5314762$ & $\{3,4,5\}$ & 3 & 0\\
35 & $5314762$ & $\{3,4,6\}$ & 9 & 6\\
36 & $6314752$ & $\{3,4,5\}$ & 9 & 9\\
\midrule
& & total & 196 & 100\\
\bottomrule
\end{tabular}
\end{center}

For later use, the same list verifies one additional predicate.  Every
unmarked anchor has degree at least two inside the seven-point core.  The
minimum anchor degree equals two for 32 marked cores and three for the
remaining four.  This is the finite input used when excluding an additional
degree-two run; it is checked directly by the portable verifier rather than
inferred from the saved list.

\subsection{Grid cells, finite states, and acceptance}

Fix a listed core $\gamma=\gamma_1\cdots\gamma_7$ and its marked set $F$.
Index an open grid cell by $(a,b)$ with $0\leq a,b\leq7$: a new point lies
after exactly $a$ core positions and above exactly $b$ core values.  Its
neighbourhood in the core is
\[
 N_\gamma(a,b)=
 \{j\leq a:\gamma_j>b\}\cup
 \{j>a:\gamma_j\leq b\},\qquad 1\leq j\leq7.
\]
A cell is retained precisely when $N_\gamma(a,b)\cap F=\varnothing$ and
inserting one point in the cell still avoids $1324$.  This definition fixes
the coordinate encoding without reference to a drawing.

For this middle calculation the southwest outer kernel is empty and the
northeast outer kernel is empty beyond its one mandatory clone point; the
two nontrivial outer kernels are restored only after the finite middle sum.
For every $f\in F$ there is one retained cell northeast of $f$ whose core
neighbourhood is $N_{G_\gamma}(f)$.  It is the \emph{clone cell} of $f$.
The normalized object contains exactly one point in each of these three
clone cells, so that the three marked points begin increasing runs of length
two.  A retained cell of core degree one must be empty by the minimum-degree
property of residuals.  A retained degree-two point either lies in one of
the three marked runs, by equality of neighbourhoods, or would give a fourth
independent degree-two class; the four-free certificate excludes the latter.
Hence all nonclone degree-two cells are empty.  Every remaining middle cell
has core degree three, contains an increasing block, and there are at most
two such cells for each marked core.

If the remaining degree-three cells are $c_1,\ldots,c_r$, where
$0\leq r\leq2$, encode their occupancies by
\[
 s=(s_1,\ldots,s_r)\in\{0,1,2\}^r,
 \qquad
 s_j=0,1,2\quad\Longleftrightarrow\quad
 |c_j|=0,1,\text{ or at least }2.
\]
For the finite representative, the last case is realized with exactly two
points.  Thus core $i$ contributes $3^{r_i}$ representatives, and the sum
over the 36 listed cores is 196.  These are representatives of disjoint
occupancy ranges; they are not orbit representatives under inverse or
reverse--complement.

The representative permutation is constructed by inserting the three
clone points and the indicated increasing blocks and then standardizing.
It is accepted if and only if all three of the following hold:
\begin{enumerate}
\item it avoids $1324$;
\item it is indecomposable, and deletion followed by standardization is
      indecomposable for each of its first, last, minimum, and maximum entries;
\item the starts of all maximal increasing consecutive runs of length two
      whose entries have inversion degree two are exactly the three marked
      starts.
\end{enumerate}
Thus the test is the definition of the normalized residual class, with the
``exactly three'' condition included; Lemma~\ref{lem:exactthree} excludes an
additional degree-two run, while the cell classification above excludes a
longer degree-two occupancy inside a finite representative. Direct evaluation accepts 100 of the
196 representatives.  In fact all 196 pass the avoidance and marked-run
tests; the rejected 96 fail residuality.

Let $d$ be the defect of a representative and let $r$ now denote the number
of its degree-three cells encoded by $2$.  The multiplicities of the 100
accepted representatives are
\begin{center}
\begin{tabular}{c|rrrrrrrrrrrr}
$(d,r)$
 &(3,0)&(4,0)&(4,1)&(5,0)&(5,1)&(5,2)
 &(6,0)&(6,1)&(6,2)&(7,1)&(7,2)&(8,2)\\
\midrule
multiplicity
 &8&16&8&10&20&2&2&16&6&4&6&2
\end{tabular}
\end{center}
and the entries sum to 100.

\subsection{From finite representatives to arbitrary occupancies}

Consider a degree-three middle cell whose finite code is $2$.  Increasing
its occupancy from two to $t\geq2$ inflates the same quotient vertex by an
increasing run.  The quotient skeleton and all singleton versus nonsingleton
statuses remain unchanged; the block remains of inversion degree three and
cannot merge with a degree-two clone run.  Increasing inflation preserves
$1324$-avoidance.  The residual deletion criterion depends on the same
quotient skeleton, so it is preserved in both directions.  Each added entry
creates three inversions and increases the length by one, and hence raises
$\operatorname{inv}-2n+7$ by one.  A representative of defect $d$ with $r$
such tails therefore contributes exactly
\[
 \frac{q^d}{(1-q)^r}.
\]
The three occupancy ranges $0$, $1$, and $\geq2$ are disjoint and exhaustive,
so this is an all-occupancy conclusion, not interpolation from bounded
defects.

Putting the 100 terms over the common denominator $(1-q)^2$ gives
\[
\begin{aligned}
 &(8q^3+16q^4+10q^5+2q^6)(1-q)^2\\
 &\quad +(8q^4+20q^5+16q^6+4q^7)(1-q)\\
 &\quad +(2q^5+6q^6+6q^7+2q^8)
   =8q^3+8q^4.
\end{aligned}
\]
Hence the middle/core series is
\[
 \frac{8q^3(1+q)}{(1-q)^2}.
\]

It remains to restore the two outer cells.  In the southwest cell, a
$132$-avoider has weakly decreasing Lehmer code; after its terminal isolated
increasing maxima are removed, the remaining kernel is uniquely encoded by
a partition and has inversion generating function
$P(q)=\prod_{j\geq1}(1-q^j)^{-1}$.  Conversely every partition has a unique
least-length Lehmer decoding, so this is a bijection rather than a stable
count.  Reverse--complement and inverse give the corresponding northeast
$213$ kernel and a second factor $P(q)$.

Each outer quotient block is joined to the same two unmarked old anchors and
to no other old vertex.  Those anchors already have an independent path
through the doubled marked run. More explicitly, delete any old vertex $x$.
If neither anchor is deleted, the surviving member of the doubled run still
connects them and every outer block attaches to that old component; if one
anchor is deleted, every outer block remains attached through the other.
Consequently the connected components on the old vertices after deleting
$x$ are unchanged by adding or removing the outer blocks. Each new outer
vertex is itself noncut because the old connected graph survives its
deletion. A nonextremal marked vertex cannot become extremal upon adding
outer points; if an old extremum is replaced, the new extremum is noncut and
the replaced old extremum lies in its doubled marked run. This component
equivalence proves preservation of the residual criterion and marked lower
vector in both directions. Its two
cross inversions per added point cancel the $-2n$ shift, leaving precisely
the internal inversion number of the kernel as the defect increment.  The
outer decorations are independently and uniquely extracted, and therefore
multiply the middle/core series by $P(q)^2$.

Consequently the generating function for canonical rank-three cells is
\[
 \sum_{\delta\geq0}N_{3,\delta}q^\delta
 =\frac{8q^3(1+q)}{(1-q)^2}P(q)^2.
\]
Since a rank-three cell contributes leading coefficient $1/2$ and cells of
rank at most two contribute none, this certificate gives the finite part of
\[
 \sum_{\delta\geq0}A_\delta q^\delta
 =\frac{4q^3(1+q)}{(1-q)^2}P(q)^2.
\]

\subsection{Reproducibility boundary}

The accompanying archival supplement contains the two C enumerators, the
explicit core list, a standard-library verifier that reconstructs the grid
cells and all 196 states, and a one-command temporary-directory replay.  The
replay establishes the finite outputs stated in this appendix.  The
all-parameter conclusions additionally use the mathematical inflation and
outer-kernel arguments above, together with the manuscript's canonical
paid/free decomposition and four-free theorem.  Saved logs from the larger
defect-11--13 catalogue runs are frozen separately; their hashes and exact
generator parameters are recorded, but those expensive enumerations are not
part of the quick replay.

\section*{Use of AI assistants}
I used AI assistants --- Anthropic's Claude, OpenAI's ChatGPT, Moonshot's Kimi and Google's Gemini ---
as tools in this work. They helped me write and run the enumeration, verification and analysis
programs, search the literature, and prepare the manuscript, and they were also set against one
another as adversarial referees, each asked to find errors in the others' arguments and drafts. I set
the direction of the project, provided key ideas, decided what to pursue and what to discard, checked
the arguments, and take responsibility for everything stated here.  The
finite checks supporting the results are reproduced by the programs and data
in the cited archive and archival supplement.


\begin{thebibliography}{9}
\bibitem{AA} M.~H.~Albert, M.~D.~Atkinson, Simple permutations and pattern restricted permutations, \emph{Discrete Math.} 300 (2005) 1--15.
\bibitem{BBEP} D.~Bevan, R.~Brignall, A.~Elvey Price, J.~Pantone, A structural characterisation of $\mathrm{Av}(1324)$ and new bounds on its growth rate, \emph{European J. Combin.} 88 (2020) 103115.
\bibitem{ABKLV} A.~Atminas, R.~Brignall, N.~Korpelainen, V.~Lozin, V.~Vatter, Well-quasi-order for permutation graphs omitting a path and a clique, \emph{Electron. J. Combin.} 22(2) (2015) \#P2.20.
\bibitem{BV} R.~Brignall, V.~Vatter, Labelled well-quasi-order for permutation classes, \emph{Comb. Theory} 2(3) (2022) \#11.
\bibitem{CJS} A.~Claesson, V.~Jel\'inek, E.~Steingr\'imsson, Upper bounds for the Stanley--Wilf limit of 1324 and other layered patterns, \emph{J. Combin. Theory Ser. A} 119 (2012) 1680--1691.
\bibitem{KR} Y.~Koh, S.~Ree, Connected permutation graphs, \emph{Discrete Math.} 307 (2007) 2628--2635.
\bibitem{LV} S.~Linusson, E.~Verkama, Enumerating 1324-avoiders with few inversions, \emph{Electron. J. Combin.} 32(3) (2025) P3.44; internal numbering follows arXiv:2408.15075.
\bibitem{Meng} L.~Meng, The two ends of the inversion spectrum of 1324-avoiding permutations, preprint (2026).
\bibitem{project} L.~Meng, Source code, data, and certificates for the inversion spectrum of $\Av(1324)$, Zenodo (2026), version DOI: \texttt{10.5281/zenodo.22667738}; concept DOI: \texttt{10.5281/zenodo.22290168}. The versioned deposit contains the reproducibility materials for the results stated here.
\bibitem{TV} B.~E.~Tenner, V.~Vatter, Cyclomatic numbers and permutations, preprint, arXiv:2607.06198 (2026).
\end{thebibliography}
\end{document}